\documentclass{amsart}     % Basic GT/GTM/AGT style
\usepackage{amssymb, amsfonts}
\usepackage[all, arc]{xy}
\title{What precisely are $E_{\infty}$ ring spaces and $E_{\infty}$ ring spectra?}

\author{J\,P May}
\address{Department of Mathematics\\
The University of Chicago\\
Chicago, Illinois 60637}
\email{may@math.uchicago.edu}
\urladdr{http://www.math.uchicago.edu/~may}

\usepackage{enumerate}
\usepackage{mathrsfs}

\newtheorem{thm}{Theorem}[section]
\newtheorem{cor}[thm]{Corollary}
\newtheorem{prop}[thm]{Proposition}
\newtheorem{lem}[thm]{Lemma}

\theoremstyle{definition}
\newtheorem{defn}[thm]{Definition}

\newtheorem{step}[thm]{Step}

\theoremstyle{remark}
\newtheorem{rem}[thm]{Remark}

\makeatletter
\let\c@equation\c@thm
\makeatother
\numberwithin{equation}{section}

\DeclareFontFamily{OMS}{rsfs}{\skewchar\font'60}
\DeclareFontShape{OMS}{rsfs}{m}{n}{<-5>rsfs5 <5-7>rsfs7 <7->rsfs10 }{}
\DeclareSymbolFont{rsfs}{OMS}{rsfs}{m}{n}
\DeclareSymbolFontAlphabet{\scr}{rsfs}

\newcommand{\sA}{\scr{A}}

\newcommand{\sC}{\scr{C}}
\newcommand{\sD}{\scr{D}}

\newcommand{\sF}{\scr{F}}
\newcommand{\sG}{\scr{G}}

\newcommand{\sI}{\scr{I}}

\newcommand{\sK}{\scr{K}}
\newcommand{\sL}{\scr{L}}
\newcommand{\sM}{\scr{M}}
\newcommand{\sN}{\scr{N}}
\newcommand{\sO}{\scr{O}}
\newcommand{\sP}{\scr{P}}

\newcommand{\sS}{\scr{S}}
\newcommand{\sT}{\scr{T}}
\newcommand{\sU}{\scr{U}}
\newcommand{\sV}{\scr{V}}
\newcommand{\sW}{\scr{W}}

\newcommand{\bC}{\mathbb{C}}

\newcommand{\bE}{\mathbb{E}}
\newcommand{\bF}{\mathbb{F}}

\newcommand{\bL}{\mathbb{L}}

\newcommand{\bN}{\mathbb{N}}

\newcommand{\bP}{\mathbb{P}}

\newcommand{\bR}{\mathbb{R}}

\newcommand{\bU}{\mathbb{U}}

\newcommand{\bZ}{\mathbb{Z}}

\newcommand{\al}{\alpha}
\newcommand{\be}{\beta}
\newcommand{\ga}{\gamma}
\newcommand{\de}{\delta}
\newcommand{\epz}{\varepsilon}
\newcommand{\ph}{\phi}

\newcommand{\et}{\eta}
\newcommand{\io}{\iota}

\newcommand{\la}{\lambda}
\newcommand{\tha}{\theta}

\newcommand{\rh}{\rho}
\newcommand{\si}{\sigma}
\newcommand{\ta}{\tau}
\newcommand{\ch}{\chi}

\newcommand{\ze}{\zeta}
\newcommand{\om}{\omega}
\newcommand{\GA}{\Gamma}
\newcommand{\LA}{\Lambda}
\newcommand{\DE}{\Delta}
\newcommand{\SI}{\Sigma}
\newcommand{\THA}{\Theta}
\newcommand{\OM}{\Omega}
\newcommand{\XI}{\Xi}

\newcommand{\PS}{\Psi}
\newcommand{\PH}{\Phi}

\newcommand{\com}{\circ}     % composition of functions
\newcommand{\iso}{\cong}     % preferred isomorphism symbol
\newcommand{\htp}{\simeq}    % homotopy symbol
\newcommand{\ten}{\otimes}   % tensor product
\newcommand{\thp}{\ltimes}   % twisted half-smash product
\newcommand{\sma}{\wedge}    % smash product
\newcommand{\wed}{\vee}      % wedge sum

\newcommand{\rtarr}{\longrightarrow}

\def\quickop#1{\expandafter\newcommand\csname #1\endcsname{\operatorname{#1}}}
\quickop{Hom} \quickop{End} \quickop{Aut} \quickop{Tel} \quickop{Mic}
\quickop{Ext} \quickop{Tor} \quickop{Id} \quickop{Coker} \quickop{Ker}
\quickop{Lim} \quickop{Colim} \quickop{Holim} \quickop{Hocolim}
\quickop{id} \quickop{tel} \quickop{mic} \quickop{coker}
\quickop{colim} \quickop{holim} \quickop{hocolim} \quickop{im}

\begin{document}

\begin{abstract} $E_{\infty}$ ring spectra were defined in 1972,
but the term has since acquired several alternative meanings.
The same is true of several related terms.  The new formulations 
are not always known to be 
equivalent to the old ones and even when they are, the notion of 
``equivalence'' needs discussion: Quillen equivalent categories 
can be quite seriously inequivalent.  Part of the confusion stems 
from a gap in the modern resurgence of interest in $E_{\infty}$ 
structures. $E_{\infty}$ ring spaces were also defined in 1972 and 
have never been redefined.  They were central to the early 
applications and they tie in implicitly to modern applications.  
We summarize the relationships between the old notions and various 
new ones, explaining what is and is not known. We take the opportunity 
to rework and modernize many of the early results.  New proofs and
perspectives are sprinkled throughout.
\end{abstract}

\maketitle

\tableofcontents

\section*{Introduction}

In the early 1970's, the theory of $E_{\infty}$ rings was intrinsically intertwined with a host of constructions and calculations that centered around the relationship between $E_{\infty}$ ring spectra and $E_{\infty}$ ring spaces 
\cite{CLM, MQR}.  The two notions were regarded as being on essentially the 
same footing, and it was understood that the homotopy categories of ringlike 
$E_{\infty}$ ring spaces ($\pi_0$ is a ring and not just a semi-ring) and of 
connective $E_{\infty}$ ring spectra are equivalent.  

In the mid 1990's, modern closed symmetric monoidal categories of spectra were introduced, allowing one to define a commutative ring spectrum to be a commutative monoid in any such good category of spectra.  The study of 
such rings is now central to stable homotopy theory. Work of several people, especially Schwede and Shipley, shows that, up to zigzags of Quillen equivalences, the resulting categories of commutative ring spectra are all equivalent. In one of these good categories, commutative ring spectra are
equivalent to $E_{\infty}$ ring spectra.  The terms $E_{\infty}$ ring spectra 
and commutative ring spectra have therefore been used as synonyms in recent
years.  A variant notion of $E_{\infty}$ ring spectrum that can be defined
in any such good category of spectra has also been given the same name.

From the point of view of stable homotopy theory, this is perfectly
acceptable, since these notions are tied together by a web of Quillen
equivalences.  From the point of view of homotopy theory as a whole,
including both space and spectrum level structures, it is not acceptable.  
Some of the Quillen equivalences in sight necessarily lose space level
information, and in particular lose the original connection between 
$E_{\infty}$ ring spectra and $E_{\infty}$ ring spaces. Since some modern applications, 
especially those connected with cohomological orientations 
and spectra of units, are best understood in terms of that connection, it 
seems to me that it might be helpful to offer a thorough survey of the structures in this general 
area of mathematics.  

This will raise some
questions.  As we shall see, some new constructions are not at present 
known to be equivalent, in any sense, to older constructions of objects 
with the same name, and one certainly cannot deduce comparisons formally.
It should also correct some misconceptions.  In some cases, an old name 
has been reappropriated for a definitely inequivalent concept. 

The paper divides conceptually into two parts. First, in \S\S 1--10, we 
describe and modernize additive and multiplicative infinite loop space theory.  
Second, in \S\S11--13, we explain how this early 1970's work fits into the 
modern framework of symmetric monoidal categories of spectra.  There will be
two sequels \cite{Sequel, Sequel2}. In the first, we recall how to construct
$E_{\infty}$ ring spaces from bipermutative categories. In the second, we
review some of the early applications of $E_{\infty}$ ring spaces. 

We begin by defining $E_{\infty}$ ring spaces. As we shall see in \S\ref{ringspace},
this is really quite easy.  The hard part is to produce examples, and that problem will
be addressed in \cite{Sequel}.  The definition requires a pair 
$(\sC,\sG)$ of $E_{\infty}$ operads, with $\sG$ acting in a 
suitable way on $\sC$, and $E_{\infty}$ ring spaces might better be called 
$(\sC,\sG)$-spaces.  It is a truism taken for granted since \cite{Geo} 
that all $E_{\infty}$ operads are suitably equivalent.  However, for 
$E_{\infty}$ ring theory, that is quite false.  The precise geometry 
matters, and we must insist that not all $E_{\infty}$ operads are alike. 
The operad $\sC$ is thought of as additive, and the operad $\sG$ is thought 
of as multiplicative.\footnote{As we recall in \S\S\ref{isoop},\ref{machines},
in many applications of additive infinite loop space theory, we must actually start 
with $\sG$, thinking of it as additive, and convert $\sG$-spaces to $\sC$-spaces 
before proceeding.}  

There is a standard canonical multiplicative operad $\sL$,
namely the linear isometries operad.  We recall it and related structures that
were the starting point of this area of mathematics in \S\ref{isoop}.  
In our
original theory, we often replaced $\sL$ by an operad $\sO\times\sL$, and we 
prefer to use a generic letter $\sG$ for an operad thought of as appropriate 
to the multiplicative role in the definition of $E_{\infty}$ ring spaces.  
The original definition of $E_{\infty}$ ring spaces was obscured because 
the canonical additive operad $\sC$ that was needed for a clean description was only 
discovered later, by Steiner \cite{St}.  We recall its definition and its relationship
to $\sL$ in \S\ref{Steiner}.  This gives us the canonical $E_{\infty}$ operad pair $(\sC,\sL)$.  

Actions by operads are equivalent to actions by an associated monad.
As we explain in \S\ref{monad}, that remains true for actions by operad pairs.
That is, just as $E_{\infty}$ spaces can be described as algebras over
a monad, so also $E_{\infty}$ ring spaces can be described as algebras
over a monad. In fact, the monadic version of the definition fits into a beautiful categorical way of thinking about distributivity that was first discovered by Beck \cite{Beck}.  This helps make the definition feel
definitively right.  

As we also explain in \S\ref{monad}, {\em different} monads can
have {\em the same} categories of algebras.  This has been known for years, 
but it is a new observation that this fact can be used to substantially 
simplify the mathematics.  In the sequel \cite{Sequel}, we will use this idea to give 
an elementary construction of $E_{\infty}$ ring spaces from bipermutative categories 
(and more general input data). We elaborate on this categorical observation and
related facts about maps of monads in Appendix A (\S14), which
is written jointly with Michael Shulman.

The early 1970's definition \cite{MQR} of an $E_{\infty}$ ring spectrum was 
also obscure, this time because the notion of ``twisted half-smash
product'' that allows a clean description was only introduced
later, in \cite{LMS}.  The latter notion encapsulates operadically
parametrized internalizations of external smash products.  As we 
recall in \S\ref{ringspec}, $E_{\infty}$ ring spectra are spectra in 
the sense of \cite{LMS, MayOld}, which we shall sometimes call LMS 
spectra for definiteness, with additional structure.  Just as $E_{\infty}$
spaces can be described in several ways as algebras over a monad, so also
$E_{\infty}$ ring spectra can be described in several ways as algebras over
a monad.  We explain this and relate the space and spectrum level monads
in \S\ref{monad2}. 

There is a $0^{th}$ space functor $\OM^{\infty}$ from spectra to 
spaces,\footnote{Unfortunately for current readability, 
in \cite{MQR} the notation 
$\SI^{\infty}$ was used for the suspension {\em prespectrum} functor, the notation $\OM^{\infty}$ was used for the spectrification functor that has 
been denoted by $L$ ever since \cite{LMS}, and the notation 
$Q_{\infty} = \OM^{\infty}\SI^{\infty}$ was used for the current 
$\SI^{\infty}$.}
which is right  adjoint to the suspension spectrum 
functor $\SI^{\infty}$.   A central
feature of the definitions, both conceptually and calculationally, is
that the $0^{th}$ space $R_0$ of an $E_{\infty}$ ring spectrum $R$ is an 
$E_{\infty}$ ring space.  Moreover, the space $GL_1R$ of unit components
in $R_0$ and the component $SL_1R$ of the identity are $E_{\infty}$-spaces,
specifically $\sL$-spaces.\footnote{$GL_1R$ and $SL_1R$ were called $FR$ and
$SFR$ when they were introduced in \cite{MQR}.  These spaces played a major 
role in that book, as we will explain in the second sequel \cite{Sequel2}. 
As we also explain there, $F$ and $GL_1S$ are both tautologically the same and
very different. The currently popular notations follow Waldhausen's later 
introduction \cite{Wald} of the higher analogues $GL_n(R)$.}  We shall say more 
about these spaces in \S\S\ref{spspcompare},\ref{machines},\ref{unitspec}.  

There is also a functor from $E_{\infty}$ ring spaces to $E_{\infty}$ ring
spectra.  This is the point of multiplicative infinite loop space theory 
\cite{MQR, Mult}.  Together with the $0^{th}$ space functor, it gives the 
claimed equivalence between the homotopy categories of ringlike $E_{\infty}$ 
ring spaces and of connective $E_{\infty}$ ring spectra.  We recall this 
in \S\ref{machines}.

The state of the art around 1975 was summarized in \cite{Bull}, and it
may help orient \S\S 1--10 of this paper to reproduce the diagram that survey focused on.
Many of the applications alluded to above are also summarized in \cite{Bull}. 
The abbreviations at the top of the diagram refer to permutative categories and 
bipermutative categories. We will recall and rework how the latter fit into 
multiplicative infinite loop space theory in the sequel \cite{Sequel}.  

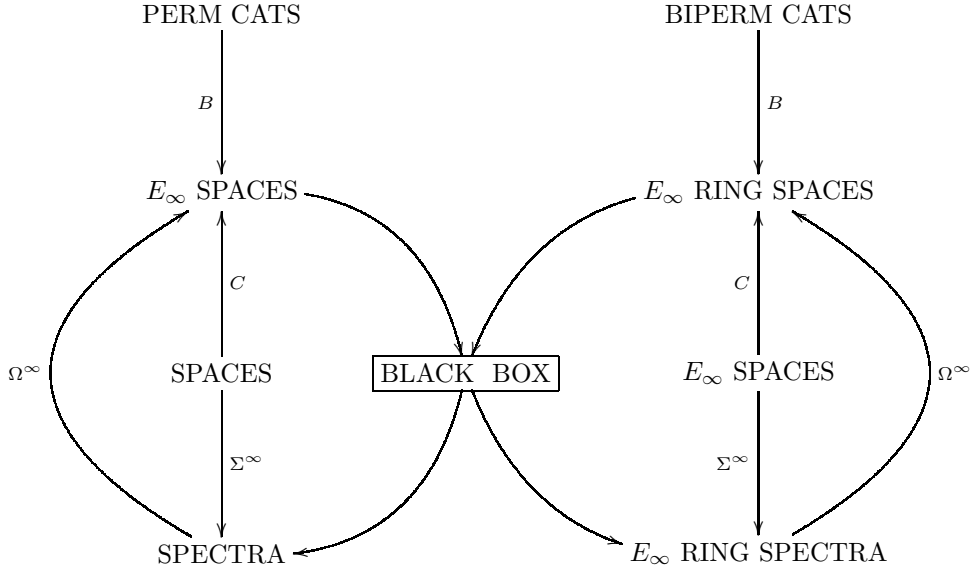
\vspace{3mm}

{

\begin{equation}\label{DIAG}
\xymatrix{
& \text{PERM\ CATS}\ar[dd]_{B} & & \text{BIPERM\ CATS} \ar[dd]^{B} & \\
&&&& \\
& E_{\infty}\ \text{SPACES}\ar@/^2.2pc/[ddr] 
& & E_{\infty} \ \text{RING\ SPACES} \ar@/_2.2pc/[ddl] & \\
&&&&\\
& \text{SPACES} \ar[uu]_{C} \ar[dd]^{\SI^{\infty}}  &  
*+[F]{\text{BLACK \ BOX}} 
\ar@/^2.2pc/[ddl] \ar@/_2.2pc/[ddr]  
& E_{\infty}\ \text{SPACES} \ar[uu]^{C} \ar[dd]_{\SI^{\infty}}\\
&&&&\\
& \text{SPECTRA}\ar@/^5.4pc/[uuuu]^{\OM^{\infty}} & & E_{\infty}\ \text{RING\ SPECTRA} 
\ar@/_5.4pc/[uuuu]_{\OM^{\infty}} & \\}
\end{equation}

}

Passage through the black box is the subject of additive infinite loop
space theory on the left and multiplicative infinite loop space theory
on the right.  These provide functors from $E_{\infty}$ spaces to 
spectra and from $E_{\infty}$ ring spaces to $E_{\infty}$ ring spectra.
We have written a single black box because the multiplicative functor
is an enriched specialization of the additive one.  The black box gives
a recognition principle: it tells us how to recognize spectrum level objects
on the space level. 

We give a modernized description of these functors in \S\ref{machines}.
My early 1970's work was then viewed as ``too categorical'' by older algebraic
topologists.\footnote{Sad to say, nearly all of the older people active 
then are now retired or dead.}  
In retrospect, 
it was not nearly categorical enough for intuitive conceptual understanding.  
In the expectation that I am addressing a more categorically sophisticated 
modern audience, I explain in \S{\ref{phil} how the theory is based on an 
analogy with the Beck monadicity theorem.  One key result, a commutation 
relation between taking loops and applying the additive infinite loop space machine, 
was obscure in my earlier work, and I'll give a new proof in Appendix B (\S15).

The diagram above obscures an essential technical point.  The two entries 
``$E_{\infty}$ spaces'' are different.  The one on the upper left refers to 
spaces with actions by the additive $E_{\infty}$ operad $\sC$, and spaces 
there mean based spaces with basepoint the unit for the additive operadic product.  
The one on the right refers to spaces with actions by the multiplicative $E_{\infty}$ 
operad $\sG$, and spaces there mean spaces with an operadic unit point $1$ and a disjoint 
added basepoint $0$.  The functor $C$ 
is the free $\sC$-space functor, and it takes $\sG$-spaces with $0$ to 
$E_{\infty}$ ring spaces.  This is a key to understanding the various adjunctions 
hidden in the diagram.  The functors labelled $C$ and $\SI^{\infty}$ are left adjoints.

The unit $E_{\infty}$ spaces $GL_1R$ and $SL_1R$ of an $E_{\infty}$ ring spectrum $R$
can be fed into the additive infinite loop space machine to produce associated spectra
$gl_1R$ and $sl_1R$.  There is much current interest in understanding their structure.
As we recall in \S\ref{unitspec}, one can exploit the interrelationship between the 
additive and multiplicative structures to obtain a general theorem that describes the
localizations of $sl_1R$ at sets of primes in terms of purely multiplicative structure.
The calculational force of the result comes from applications to spectra arising from
bipermutative categories, as we recall and illustrate in the second sequel \cite{Sequel2}.
The reader may prefer to skip this section on a first reading, since it is not essential
to the main line of development, but it gives a good illustration of information about
spectra of current interest that only makes sense in terms of $E_{\infty}$ ring spaces.

Turning to the second part, we now jump ahead more than twenty years. In the 1990's, 
several categories of spectra that are symmetric monoidal under their smash product were 
introduced.  This allows the definition of commutative ring spectra as
commutative monoids in a symmetric monoidal category of spectra.   Anybody 
who has read this far knows that the resulting theory of ``stable commutative
topological rings'' has become one of the central areas of study in modern algebraic topology. 
No matter how such a modern category of spectra is constructed, the essential point is that 
there is some kind of external smash product in sight, which is commutative and associative 
in an external sense, 
and the problem that must be resolved is to figure out how to internalize it without 
losing commutativity and associativity.  

Starting from twisted half-smash products, this internalization is carried out 
in EKMM \cite{EKMM}, where the symmetric monoidal category
of $S$-modules is constructed.  We 
summarize some of the relevant
theory in \S\ref{EKMM}. Because the construction there
starts with twisted half-smash products, the resulting commutative
ring spectra are almost the same as $E_{\infty}$ ring spectra.  The
``almost'' is an important caveat. We didn't mention the unit condition
in the previous paragraph, and that plays an important and subtle role 
in \cite{EKMM} and in the comparisons we shall make.  As
Lewis noted \cite{Lewis} and we will rephrase, one cannot have a symmetric 
monoidal category of spectra that is as nicely related to spaces as one 
would ideally like.  The reason this is so stems from an old result of
Moore, which says that a connected commutative topological monoid is
a product of Eilenberg--Mac\,Lane spaces.  

In diagram spectra, in particular symmetric spectra and orthogonal
spectra \cite{HSS, MMSS}, the internalization is entirely different.
Application of the elementary categorical notion of left Kan extension 
replaces the introduction of the twisted half-smash product, and there
is no use of operads. However, there is a series of papers, primarily
due to Schwede and Shipley \cite{MM, MMSS, Schw, SS, Ship}, that lead to
the striking conclusion that all reasonable categories of spectra that 
are symmetric monoidal and have sensible Quillen model structures are 
Quillen equivalent.  Moreover, if one restricts to the commutative monoids, 
alias commutative ring spectra, in these categories, we again obtain Quillen 
equivalent model categories. 

Nevertheless, as we try to make clear 
in \S\ref{MMSS}, these last Quillen equivalences {\em lose essential information}.  
On the diagram spectrum side, one must throw
away any information about $0^{th}$ spaces in order to obtain the Quillen
equivalence with EKMM style commutative ring spectra.  In effect, this 
means that diagram ring spectra do not know about $E_{\infty}$ ring spaces
and cannot be used to recover the original space level results
that were based on implications of that structure.

Philosophically, one conclusion is that fundamentally important homotopical
information can be accessible to one and inaccessible to the other of a pair 
of Quillen equivalent model categories, contrary to current received
wisdom.  The homotopy categories of connective
commutative symmetric ring spectra and of ringlike $E_{\infty}$ ring spaces
are equivalent, but it seems impossible to know this without going through
the homotopy category of $E_{\infty}$ ring spectra, as originally defined.

We hasten to be clear.  It does not follow that $S$-modules are
``better'' than symmetric or orthogonal spectra.  There is by now a huge literature 
manifesting just how convenient, and in some contexts essential, diagram spectra 
are.\footnote{I've contributed to this in collaboration with Mandell, Schwede, Shipley,
and Sigurdsson
\cite{MM, MMSS, MS}.} Rather, it does follow that to have access to the full 
panoply of information and techniques this subject affords, one simply must 
be eclectic. To use either approach alone is to approach modern stable
homotopy theory with blinders on.

A little parenthetically, there is also a quite different alternative notion of a 
``naive $E_{\infty}$ ring spectrum'' (that is meant as a technical term, not a 
pejorative). For that,
one starts with internal iterated smash products and uses tensors
with spaces to define actions by an $E_{\infty}$ operad.  This 
makes sense in any good modern category of spectra, and the geometric 
distinction between different choices of $E_{\infty}$ operad is irrelevant.  
Most such categories of spectra do not know the difference between
symmetric powers $E^{(j)}/\SI_j$ and homotopy symmetric powers
$(E\SI_j)_+\sma_{\SI_j}E^{(j)}$, and naive $E_{\infty}$ ring 
spectra in such a good modern category of spectra are naturally 
equivalent to commutative ring spectra in that category, 
as we explain in \S\ref{naive}.  

This summary raises some important compatibility questions.  For example, 
there is a construction, due to Schlichtkrull \cite{Sch}, of unit spectra 
associated to commutative symmetric ring spectra.  It is based on the use 
of certain diagrams of spaces that are implicit in the structure of symmetric spectra.  
It is unclear that these unit spectra are equivalent to those 
that we obtain from the $0^{th}$ space of an  ``equivalent'' 
$E_{\infty}$ ring spectrum. Thus we now have two constructions, not known 
to be equivalent,\footnote{Since I wrote that, John Lind (at 
Chicago) has obtained an illuminating proof that they are.}  
of objects bearing the same name. Similarly, there is a construction of (naive) $E_{\infty}$ symmetric ring spectra associated to oplax bipermutative categories (which are not equivalent to bipermutative categories as originally defined) that is due 
to Elmendorf and Mandell \cite{EM}.  It is again not known whether or not their construction (at least when specialized to genuine bipermutative categories) gives symmetric ring spectra that are ``equivalent'' to the $E_{\infty}$ ring spectra that are constructed from bipermutative categories via our black box.  Again, we have two constructions that are not known to be equivalent, both thought of as giving the $K$-theory commutative ring spectra associated to bipermutative categories.

Answers to such questions are important if one wants to make consistent use of the alternative constructions, especially since the earlier constructions are 
part of a web of calculations that appear to be inaccessible with the newer constructions. The constructions of \cite{Sch} and \cite{EM} bear no relationship to $E_{\infty}$ ring spaces as they stand and therefore
cannot be used to retrieve the earlier applications or to achieve analogous
future applications.  However, the new constructions have significant advantages as well as significant disadvantages.  Rigorous comparisons are needed. We must be consistent as well as eclectic.  There is work to be done!

For background, Thomason and I proved in \cite{MT}, that any two 
infinite loop space machines (the additive black box in Diagram (\ref{DIAG})) are equivalent.  The proof broke into two quite different steps.  In the first, we compared input data. We showed that Segal's input data (special $\Gamma$-spaces) and 
the operadic input data of Boardman and Vogt and myself ($E_{\infty}$ spaces)
are each equivalent to a more general kind of input data, namely an action of the category of operators $\hat{\sC}$ associated to any chosen $E_{\infty}$ operad $\sC$.  We then showed that any two functors $\bE$ from $\hat{\sC}$-spaces to spectra that satisfy a group completion property 
on the $0^{th}$ space level are equivalent.  This property says that there 
is a natural group completion map $\eta\colon X\rtarr \bE_0X$, and we will
sketch how that property appears in one infinite loop space machine
in \S\ref{machines}.

No such uniqueness result is known for multiplicative infinite loop space
theory.  As we explain in \cite{Sequel}, variant notions of bipermutative
categories give possible choices of input data that are definitely
inequivalent.  There are also equivalent but inequivalent choices of output data, 
as I hope the discussion above makes clear.  We might take as target
any good modern category of commutative ring spectra and then, thinking
purely stably, all choices are equivalent.  However, the essential feature 
of \cite{MT} was the compatibility statement on the $0^{th}$ space level.  
There were no problematic choices since the correct homotopical notion of
spectrum is unambiguous, as is the correct homotopical relationship 
between spectra and their $0^{th}$ spaces (of fibrant approximations model categorically).  
As we have indicated, understanding multiplicative infinite loop space theory on the 
$0^{th}$ space level depends heavily on choosing the appropriate target 
category.\footnote{In fact, this point was already emphasized in the 
introduction of \cite{MT}.}

With the black box that makes sense of Diagram (\ref{DIAG}), 
there are stronger comparisons of input data and $0^{th}$ spaces 
than the axiomatization prescribes.  Modulo the inversion of a
natural homotopy equivalence, the map $\eta$ is a map
of $E_{\infty}$ spaces in the additive theory and a map of
$E_{\infty}$ ring spaces in the multiplicative theory.  This property
is central to all of the applications of \cite{CLM, MQR}.  For example,
it is crucial to the analysis of localizations of spectra of units in 
\S\ref{unitspec}.

This paper contains relatively little that is technically new, although
there are many new perspectives and many improved arguments.  It is  intended 
to give an overview of the global structure of this general area of mathematics, 
explaining the ideas in a context uncluttered by technical details.  It is a 
real pleasure to see how many terrific young mathematicians are interested 
in the theory of structured ring spectra, and my primary purpose is to help 
explain to them what was known in the early stages of the theory and how it 
relates to the current state of the art, in hopes that this might 
help them see connections with things they are working on now. Such a 
retelling of an old story seems especially needed since notations, 
definitions, and emphases have drifted over the years and their are some
current gaps in our understanding.  

Another
reason for writing this is that I plan to rework complete
details of the analogous equivariant story, a tale known decades ago but
never written down.  Without a more up-to-date 
nonequivariant blueprint, that story would likely be quite unreadable.   
The equivariant version will (or, less optimistically, may) give 
full details that will supersede those in the 1970's sources.

\vspace{2mm}

I'd like to thank the organizers of the Banff conference, 
Andy Baker and Birgit Richter, who are entirely to blame 
for the existence of this paper and its sequels.  They scheduled me for an 
open-ended evening closing talk, asking me to talk about the 
early theory.  They then provided the audience with enough 
to drink to help alleviate the resulting boredom.  This work 
began with preparation for that talk.  I'd also like to thank
John Lind and an eagle-eyed anonymous referee for catching numerous misprints
and thereby sparing the reader much possible confusion.

\section{The definition of $E_{\infty}$ ring spaces}\label{ringspace}

We outline the definition of $E_{\infty}$ spaces and $E_{\infty}$ ring
spaces.  We will be careful about basepoints throughout, since that 
is a key tricky point and the understanding here will lead to a streamlined
passage from alternative inputs, such as bipermutative categories, to 
$E_{\infty}$ ring spaces in \cite{Sequel}.  Aside from that, we focus on the intuition and refer the reader to \cite{Geo, MQR, Mult} for the combinatorial details.  Let $\sU$ denote the category of (compactly generated) unbased spaces and $\sT$ denote the category of based spaces.  We tacitly 
assume that based spaces $X$ have nondegenerate basepoints, or are well-based, which means that 
$\ast\rtarr X$ is a cofibration.

We assume that the reader is familiar with the definition of an operad.
The original definition, and our focus here, is on operads in $\sU$, as in 
\cite[p.\,1--3] {Geo}, but the definition applies equally well to define operads in any 
symmetric monoidal category \cite{MayDef, MayDef2}.  As in 
\cite{Geo}, we insist that the operad $\sO$ be reduced, in the sense that $\sO(0)$ is 
a point $\ast$.  This is important 
to the handling of basepoints.  Recall
that there is an element $\id\in \sO(1)$ that corresponds 
to the identity operation\footnote{The notation $1$ is 
standard in the literature, but that would obscure things for us later.} 
and that the 
$j$-th space $\sO(j)$ has a right action 
of the symmetric group $\SI_j$. There are structure maps 
\[\ga\colon \sO(k)\times \sO(j_1)\times \cdots \times \sO(j_k)\rtarr 
\sO(j_1+\cdots + j_k)\]
that are suitably equivariant, unital, and associative. We say that $\sO$ 
is an $E_{\infty}$ operad if $\sO(j)$ is contractible and $\SI_j$ acts freely.  

The precise details of the definition are dictated by looking at the structure present on the  endomorphism operad $\text{End}_X$ of a based space $X$.  This actually has two variants, $\text{End}_X^{\sT}$ and $\text{End}_X^{\sU}$, depending on whether or not we restrict to based maps.  The default will be 
$\text{End}_X = \text{End}_X^{\sT}$. The $j^{th}$ space $\text{End}_X(j)$ is the space of (based) maps $X^j\rtarr X$ and
\[ \ga(g;f_1,\cdots, f_k) = g\com (f_1\times \cdots \times f_k).\]
We interpret $X^0$ to be a point,\footnote{This is reasonable since the product of the empty set of objects is the terminal object.} and $\text{End}_X(0)$ is the map given by the 
inclusion of the basepoint.  Of course, $\text{End}^{\sU}_X(0) = X$, so
the operad $\text{End}^{\sU}_X$ is not reduced.

An action $\tha$ of $\sO$ on $X$ is a map of operads
$\sO\rtarr \text{End}_X$.  Adjointly, it is given by suitably 
equivariant, unital, and associative action maps
\[ \tha\colon \sO(j)\times X^j\rtarr X. \]
We think of $\sO(j)$ as parametrizing a $j$-fold product operation on $X$.  
The basepoint of $X$ must be $\tha(\ast)$, and there are two ways of thinking
about it.  We can start with an unbased space $X$, and then $\tha(\ast)$ gives
it a basepoint, fixing a point in $\text{End}^{\sU}_X(0)$, or we can think of the basepoint as preassigned, and then $\tha(\ast) = \ast$ is required. With 
$j_1=\cdots=j_k=0$, the compatibility of $\tha$ with the structure maps 
$\ga$ ensures that a map of operads $\sO\rtarr \text{End}_X^{\sU}$ 
necessarily lands in $\text{End}_X = \text{End}_X^{\sT}$. 

Now consider a pair $(\sC,\sG)$ of operads.  Write $\sC(0)=\{0\}$ 
and $\sG(0)=\{1\}$.  An action of $\sG$ on $\sC$ consists of maps
\[ \la\colon \sG(k)\times \sC(j_1)\times \cdots \times \sC(j_k)
\rtarr \sC(j_1\cdots j_k)  \]
for $k\geq 0$ and $j_i\geq 0$ that satisfy certain equivariance, unit, and 
distributivity properties; see \cite[p. 142-144]{MQR}, \cite[p. 8-9]{Mult},
or the sequel \cite[4.2]{Sequel}.    
We will give an alternative perspective in \S\ref{monad} that dictates
the details. To deal with basepoints, we interpret the empty product of 
numbers to be $1$ and, with $k=0$, we require that $\la(1) = \id\in \sC(1)$.
We think of $\sC$ as parametrizing addition and $\sG$ as parametrizing multiplication.  For example, we have an
operad $\sN$ such that $\sN(j) =\ast$ for all $j$. An $\sN$-space is precisely
a commutative monoid.  There is one and only one way $\sN$
can act on itself, and an $(\sN,\sN)$ space is precisely a 
commutative topological semi-ring or ``rig space'', a ring without 
negatives.  We say that $(\sC,\sG)$ is an $E_{\infty}$ operad pair
if $\sC$ and $\sG$ are both $E_{\infty}$ operads.  We give a canonical
example in \S\ref{Steiner}.

Of course, a rig space $X$ must have additive and multiplicative unit
elements $0$ and $1$,
and they must be different for non-triviality.  It is convenient 
to think of $S^0$ as $\{0,1\}$, so that these points are encoded by a map 
$e\colon S^0\rtarr X$.  In \cite{MQR,Mult}, we thought of both of these
points as ``basepoints''.  Here we only think of $0$ as a basepoint. 
This sounds like a trivial distinction, but it leads to a significant change 
of perspective when we pass from operads to monads in \S\ref{monad}.  We let 
$\sT_e$ denote the category of spaces $X$ together with a map 
$e\colon S^0\rtarr X$.  That is, it is the category of spaces under $S^0$. 

One would like to say that we have an endomorphism operad pair such that an action of an operad
pair is a map of operad pairs, but that is not quite how things work.  Rather, an action of $(\sC,\sG)$ on $X$ consists of an action $\tha$ of $\sC$ on $(X,0)$ and an action $\xi$ of $\sG$ on $(X,1)$ for which $0$ is a strict zero, so
that $\xi(g;y) = 0$ if any coordinate of $y$ is $0$, and for which the parametrized version of the left distributivity law holds.  In a 
rig space $X$, for variables $(x_{r,1},\cdots,x_{r,j_r})\in X^{j_r}$, 
$1\leq r\leq k$, we set $z_r = x_{r,1}+\cdots +x_{r,j_r}$ and find that 
\[ z_1\cdots z_k = \sum_{Q} x_{1,q_1}\cdots x_{k,q_k}, \]
where the sum runs over the set of sequences
$Q = (q_1,\cdots,q_k)$ such that $1\leq q_r\leq j_r$, ordered
lexicographically.  The parametrized
version required of a $(\sC,\sG)$-space is obtained by first defining maps
\begin{equation}\label{Act0} \xi\colon \sG(k)\times \sC(j_1)\times X^{j_1}\times 
\cdots \times \sC(j_k)\times X^{j_k} \rtarr \sC(j_1\cdots j_k)\times X^{j_1\cdots j_k}  
\end{equation}
and then requiring the following diagram to commute.
\begin{equation}\label{ActI}
 \xymatrix{
\sG(k)\times\sC(j_1)\times X^{j_1}\times \cdots \times \sC(j_k)\times X^{j_k}
\ar[r]^-{\id\times \tha^k} \ar[d]_{\xi} & \sG(k)\times X^k \ar[d]^{\xi}\\
\sC(j_1\cdots j_k)\times X^{j_1\cdots j_k}\ar[r]_-{\tha} & X\\} 
\end{equation}
The promised map $\xi$ on the left is defined by
\begin{equation}\label{ActII}
 \xi(g; c_1, y_1, \cdots, c_k, y_k) 
 = (\la(g;c_1,\cdots,c_k);\prod_{Q}\xi(g;y_Q))
\end{equation}
where $g\in\sG(k)$, $c_r\in \sC(j_r)$, $y_r = (x_{r,1},\cdots,x_{r,j_r})$,
the product is taken over the lexicographically ordered set of sequences $Q$,
as above, and $y_Q = (x_{1,q_1},\cdots , x_{k,q_k})$.

The following observation is trivial, but it will lead in the sequel \cite{Sequel} to significant technical simplifications of \cite{Mult}. 

\begin{rem}\label{cuteness} 
All basepoint conditions, including the strict zero condition, are 
in fact redundant.  We have
seen that the conditions $\sC(0) = 0$ and $\sG(0) = 1$ imply 
that the additive and multiplicative operad actions specify the points
$0$ and $1$ in $X$.  If any $j_r =0$, then $j_1\cdots j_r= 0$, we have no 
coordinates $x_{r,i_r}$, and we must interpret $\xi$ in (\ref{Act0}) to be 
the unique map to the point $\sC(0)\times X^0$.  Then
(\ref{ActI}) asserts that the right vertical arrow $\xi$ takes the
value $0$. With all but one $j_r = 1$ and the remaining $j_r=0$, this
forces $0$ to be a strict zero for $\xi$. 
\end{rem}

\section{$\sI$-spaces and the linear isometries operad}\label{isoop}

The canonical multiplicative operad is the linear isometries operad $\sL$,
which was introduced by Boardman and Vogt \cite{BV, BV2}; see also \cite[I\S1]{MQR}. 
It is an $E_{\infty}$ operad that enjoys several very special geometric properties. 
In this brief section, we recall its definition and that of related structures that 
give rise to $\sL$-spaces and $\sL$-spectra.  

Let $\sI$ denote the topological category of finite dimensional 
real inner product spaces and linear isometric isomorphisms and 
let $\sI_c$ denote the category of finite or countably infinite dimensional 
real inner product spaces and linear isometries.\footnote{The notations $\sI_{\ast}$
and $\sI$ were used for our $\sI$ and $\sI_c$ in \cite{MQR}.  We are following \cite{MS}.}  For the latter,
we topologize inner product spaces as the colimits of their finite 
dimensional subspaces and use the function space topologies on the
$\sI_c(V,W)$.  These are contractible spaces when $W$ is infinite
dimensional.  

When $V$ is finite dimensional and $Y$ is a based space, we let $S^V$ denote
the one-point compactification of $V$ and let $\OM^VY = F(S^V,Y)$ denote the 
$V$-fold loop space of $Y$.  In general $F(X,Y)$ denotes the space of based
maps $X\rtarr Y$. Let $U = \bR^{\infty}$ with its standard inner product.  
Define $\sL(j) = \sI_c(U^j,U)$, where $U^j$ is the sum of $j$ copies of $U$, 
with $U^0 = \{0\}$. The element $\id\in\sL(1)$ is the identity isometry, the 
left action of $\SI_j$ on $U^j$ induces the right action of $\SI_j$ on $\sL(j)$, 
and the structure maps $\ga$ are defined by 
\[ \ga(g;f_1,\cdots, f_j) = g\com (f_1\oplus\cdots \oplus f_j). \]
Notice that $\sL$ is a suboperad of the endomorphism operad of $U$.
 
For use in the next section and in the second sequel \cite{Sequel2}, we recall some 
related formal notions from \cite[I\S1]{MQR}. These reappeared and were 
given new names and applications in \cite[\S23.3]{MS}, whose notations we follow.
An $\sI$-space is a continuous functor $F\colon \sI\rtarr \sU$.   
An $\sI$-FCP (functor with cartesian product) $F$ is an $\sI$-space that is a 
lax symmetric monoidal functor, where $\sI$ is symmetric monoidal under 
$\oplus$ and $\sU$ is symmetric monoidal under cartesian product.  
This means that there are maps
\[ \om\colon F(V)\times F(W)\rtarr F(V\oplus W)\]
that give a natural transformation $\times \com (F,F)\rtarr F\com\oplus$  
which is associative and commutative up to coherent natural isomorphism.  We require that
$F(0) = \ast$ and that $\om$ be the evident identification when $V=0$ or $W=0$.  
When $F$ takes values in based spaces, we require the maps $FV\rtarr F(V\oplus W)$ that 
send $x$ to $\om(x,\ast)$ to be closed inclusions.  We say that $(F,\om)$ is 
monoid-valued if $F$ takes values in the cartesian monoidal category 
of monoids in $\sU$ and $\om$ is given by maps of monoids.  We then
give the monoids $F(V)$ their unit elements as basepoints and insist 
on the closed inclusion property.  All of the classical groups (and $String$) give examples.
Since the classifying space functor is product preserving, the spaces $BF(V)$ then give an
$\sI$-FCP $BF$.

We can define analogous structures with $\sI$ replaced throughout by $\sI_c$.
Clearly $\sI_c$-FCP's restrict to $\sI$-FCP's.  Conversely, our closed inclusion
requirement allows us to pass to colimits over inclusions $V\subset V'$ of subspaces 
in any given countably infinite dimensional inner product space to obtain $\sI_c$-FCP's
from $\sI$-FCP's.  Formally, we have an equivalence between the category of $\sI$-FCP's
and the category of $\sI_c$-FCP's. Details are given in \cite[I\S1 and VII\S2]{MQR} and 
\cite[\S23.6]{MS}, and we will illustrate the argument by example in the next section.
When we evaluate an $\sI_c$-FCP $F$ on $U$, we obtain an $\sL$-space $F(U)$, often 
abbreviated to $F$ when there is no danger of confusion.  The structure maps
\[ \tha\colon \sL(j)\times F(U)^j\rtarr F(U) \]
are obtained by first using $\om$ to define $F(U)^j\rtarr F(U^j)$ and then using 
the evaluation maps
\[ \sI_c(U^j,U)\times F(U^j)\rtarr F(U) \]
of the functor $F$.  This simple source of $E_{\infty}$-spaces is fundamental to 
the geometric applications, as we recall in \S\ref{unitspec} and the second sequel 
\cite{Sequel2}.  
We can feed these examples into the additive infinite 
loop space machine to obtain spectra.    

There is a closely related notion of an $\sI$-FSP (functor with smash 
product).\footnote{These were called $\sI_*$-prefunctors when they were 
first defined in \cite[IV\S2]{MQR}; the more sensible name FSP was introduced 
later by B\"okstedt \cite{Bok}. For simplicity, we restrict attention
to commutative $\sI$-FSP's in this paper.  In analogy with $\sI$-FCP's, the
definition in \cite[IV\S2]{MQR} required a technically convenient inclusion 
condition, but it is best not to insist on that.}  For this, we again 
start with an $\sI$-space $T\colon \sI\rtarr \sT$, but we now regard $\sT$ as symmetric
monoidal under the smash product rather than the cartesian product.
The sphere functor $S$ is specified by $S(V) = S^V$ and is strong
symmetric monoidal: $S(0) = S^0$ and $S^V\sma S^W\iso S^{V\oplus W}$.
An $\sI$-FSP is a lax symmetric monoidal functor $T$ together with a
unit map $S\rtarr T$.  This structure is given by maps
\[ \om\colon T(V)\sma T(W)\rtarr T(V\oplus W)\]
and $\et\colon S^V\rtarr T(V)$. When $W=0$, we require $\om\com (\text{id}\sma\et)$
to be the obvious identification $T(V)\sma S^0\iso T(V\oplus 0)$.  
%This simple notion is the forerunner of various examples and theories.  
The Thom spaces $TO(V)$ of the universal $O(V)$-bundles give the Thom $\sI$-FSP $TO$, 
and the other classical groups give analogous Thom $\sI$-FSP's.  
A full understanding of the relationship between $\sI$-FCP's and $\sI$-FSP's
requires the notion of parametrized $\sI$-FSP's, as defined and illustrated
by examples in \cite[\S23.2]{MS}, but we shall say no more about that here.

We shall define $E_{\infty}$ ring prespectra, or $\sL$-prespectra, in 
\S\ref{ringspec}.  The definition codifies structure that is implicit in 
the notion of an $\sI$-FSP, so these give examples.  That is, we 
have a functor from $\sI$-FSP's to $\sL$-prespectra.  The simple observation
that the classical Thom prespectra arise in nature from $\sI$-FSP's
is the starting point of $E_{\infty}$ 
ring theory and thus of this whole area of mathematics.  We shall also 
define $E_{\infty}$ ring spectra, or $\sL$-spectra, in \S\ref{ringspec},
and we shall describe a spectrification functor from $\sL$-prespectra to $\sL$-spectra.
Up to language and clarification of details, these constructions date from 1972 and are 
given in \cite{MQR}. It was noticed over twenty-five years later that $\sI$-FSP's are exactly 
equivalent to 
(commutative) orthogonal ring spectra. This gives an alternative way of processing the simple 
input data of $\sI$-FSP's, as we shall explain.  However, we next return to 
$E_{\infty}$ ring spaces and explain the canonical operad pair that acts
on the $0^{th}$ spaces of $\sL$-spectra, such as Thom spectra. 

\section{The canonical $E_{\infty}$ operad pair}\label{Steiner}

The canonical additive $E_{\infty}$ operad is much less obvious than $\sL$.  
We first recall the little cubes operads $\sC_n$, which were also introduced
by Boardman and Vogt \cite{BV, BV2}, and the little discs operads $\sD_V$. 
We then explain why neither is adequate for our purposes.

For an open subspace $X$ of a finite dimensional 
inner product space $V$, define the embeddings operad $\text{Emb}_X$ 
as follows.  Let $E_X$ denote the space of (topological) embeddings 
$X\rtarr X$. Let $\text{Emb}_X(j)\subset E_X^j$ be the space of $j$-tuples 
of embeddings with disjoint images.  Regard such a $j$-tuple as 
an embedding  $^jX\rtarr X$, where $^jX$ denotes the disjoint union
of $j$ copies of $X$ (where $^0X$ is the empty space).  The element $\id\in \text{Emb}_X(1)$ is the 
identity embedding, the group $\SI_j$ acts on $\text{Emb}_X(j)$ by 
permuting embeddings, and the structure maps 
\begin{equation}\label{gamma}
\ga\colon \text{Emb}_X(k)\times \text{Emb}_X(j_1)\times \cdots \times
\text{Emb}_X(j_k)\rtarr \text{Emb}_X(j_1+\cdots + j_k) 
\end{equation}
are defined as follows. Let $g=(g_1,\cdots, g_k)\in  \text{Emb}_X(k)$ and 
$f_r=(f_{r,1}, \cdots, f_{r,j_r})\in \text{Emb}_X(j_r)$, $1\leq r\leq k$.
Then the $r$th block of $j_r$ embeddings in
$\ga(g;f_1,\cdots, f_j)$ is given by the composites 
$g_r\com f_{r,s}$, $1\leq s\leq j_r$.   

Taking $X=(0,1)^n\subset \bR^n$, we obtain a suboperad $\sC_n$ of $\text{Emb}_X$ 
by restricting to the little $n$-cubes, namely those embeddings 
$f\colon X\rtarr X$ such that  $f = \ell_1\times\cdots\times \ell_n$, where 
$\ell_i(t) = a_it+b_i$ for real numbers $a_i>0$ and $b_i\geq 0$. 

For a general $V$, taking $X$ to be the open unit disc $D(V)\subset V$, 
we obtain a suboperad $\sD_V$ of $\text{Emb}_V$ by restricting 
to the little $V$-discs, namely those embeddings $f\colon D(V)\rtarr D(V)$ 
such that $f(v) = av +b$ for some real number $a>0$ and some element 
$b\in D(V)$. 

It is easily checked that these definitions do give well-defined suboperads.   
Let $F(X,j)$ denote the configuration space of $j$-tuples of distinct elements 
of $X$, with its permutation action by $\SI_j$.  These spaces do not fit together 
to form an operad, and $\sC_n$ and $\sD_V$ specify fattened up equivalents that 
do form operads.  By restricting little $n$-cubes or little  $V$-discs to their 
values at the center point of $(0,1)^n$ or $D(V)$,
we obtain $\SI_j$-equivariant deformation retractions 
\[  \sC_n(j)\rtarr F((0,1)^n,j)\iso F(\bR^n,j) \ \ \ \text{and}\ \ \ 
\sD_V(j)\rtarr F(D(V),j)\iso F(V,j).\]  
This gives control over homotopy types.

For a little $n$-cube $f$, the product $f\times \id$ gives a little
$(n+1)$-cube.  Applying this to all little $n$-cubes gives a ``suspension'' 
map of operads $\sC_n\rtarr \sC_{n+1}$.  We can pass to colimits over 
$n$ to construct the infinite little cubes operad $\sC_{\infty}$, and it is
an $E_{\infty}$ operad.  However, little $n$-cubes are clearly too square
to allow an action by the orthogonal group $O(n)$, and we cannot define
an action of $\sL$ on $\sC_{\infty}$.  

For a little $V$-disc $f$ and an element $g\in O(V)$, we obtain another
little $V$-disc $gfg^{-1}$.  Thus the group $O(V)$ acts on the operad
$\sD_V$.  However, for $V\subset W$, so that $W=V\oplus (W-V)$ where 
$W-V$ is the orthogonal complement of $V$ in $W$, little 
$V$-discs $f$ are clearly too round for the product $f\times \id$ 
to be a little $W$-disc.  We can send a little $V$-disc $v\mapsto av+b$
to the little $W$-disc $w\mapsto aw +b$, but that is not compatible with
the decomposition $S^W\iso S^V\sma S^{W-V}$ used to identify $\OM^W Y$ 
with $\OM^{W-V}\OM^V Y$.  Therefore we cannot suspend.   

In \cite{MQR}, the solution was to introduce the little convex bodies
partial operads.  They weren't actually operads because the structure
maps $\ga$ were only defined on subspaces of the expected domain
spaces.  The use of partial operads introduced quite unpleasant complications.
Steiner \cite{St} found a remarkable construction of operads $\sK_V$
which combine all of the good properties of the $\sC_n$ and the $\sD_V$.
His operads, which we call the Steiner operads, are defined in terms of paths 
of embeddings rather than just embeddings.

Define $R_V\subset E_V=\text{Emb}_V(1)$ to be 
the subspace of distance reducing embeddings $f\colon V\rtarr V$.  This
means that  $|f(v)-f(w)|\leq |v-w|$ for all $v,w\in V$.  Define a 
Steiner path to be a map $h\colon I\rtarr R_V$ such that $h(1)=\id$
and let $P_V$ be the space of Steiner paths.  Define 
$\pi\colon P_V\rtarr R_V$ by evaluation at $0$, $\pi(h) = h(0)$.   
Define $\sK_V(j)$ to be the space of $j$-tuples $(h_1,\cdots,h_j)$ of
Steiner paths such that the $\pi(h_r)$ have disjoint images. The element 
$\id\in \sK_V(1)$ is the constant path at the identity embedding,
the group $\SI_j$ acts on $\sK_V(j)$ by permutations, and the structure
maps $\ga$ are defined pointwise in the same way as those of $\text{Emb}_V$.
That is, for $g = (g_1,\dots,g_k)\in\sK_V(k)$ and 
$f_r=(f_{r,1},\dots,f_{r,j_r}) \in \sK(j_r)$, $\ga(g;f_1,\cdots f_j)$ 
is given by the embeddings 
$g_r(t)\com f_{r,s}(t)$, in order. This gives well defined operads, 
and application of $\pi$ to Steiner paths gives a map of operads 
$\pi\colon \sK_V\rtarr \text{Emb}_V$. 

By pullback along $\pi$, any space with an action by $\text{Emb}_V$ inherits 
an action by $\sK_V$.  As in \cite[\S5]{Geo} or \cite[VII\S2]{MQR}, 
$\text{Emb}_V$ acts naturally on $\OM^VY$. A $j$-tuple
of embeddings $V\rtarr V$ with disjoint images determines a map from $S^V$ to
the wedge of $j$ copies of $S^V$ by collapsing points of $S^V$ not in any of
the images to the point at infinity and using the inverses of the given embeddings 
to blow up their images to full size.  A $j$-tuple of based maps
$S^V\rtarr Y$ then gives a map from the wedge of the $S^V$ to $Y$. 
Thus the resulting action $\tha_V$ of $\sK_V$ on $\OM^VY$ is given by composites
\[\xymatrix@1{
\sK_V(j)\times (\OM^V Y)^j\ar[r]^-{\pi\times\text{id}} & 
\text{Emb}_V(j)\times (\OM^V Y)^j \ar[r] & \OM^V(^jS^V)\times (\OM^V Y)^j\ar[r] & \OM^VY. \\} \]
Evaluation of embeddings at $0\in V$ gives maps $\ze\colon \text{Emb}_V(j)\rtarr  F(V,j)$. 
Steiner determines the homotopy types of the $\sK_V(j)$ by proving that the 
composite maps $\ze\com\pi\colon \sK_V(j)\rtarr F(V,j)$ are $\SI_j$-equivariant 
deformation retractions.  

The operads $\sK_V$ have extra geometric structure that make them ideally suited for our purposes.
Rewriting $F(V) = F_V$, we see that $E$, $R$, and $P$ above are monoid-valued
$\sI$-FCP's.  The monoid products are given (or induced pointwise) by composition of 
embeddings, and the maps $\om$ are given by cartesian products of embeddings. For the 
functoriality, if $f\colon V\rtarr V$ is an embedding and 
$g\colon V\rtarr W$ is a linear isometric isomorphism, then we obtain an
embedding $gfg^{-1}\colon W\rtarr W$ which is distance reducing if $f$ is.  
We have an inclusion $R\subset E$ of monoid-valued $\sI$-FCP's, and evaluation 
at $0$ gives a map $\pi\colon P\rtarr R\subset E$ of monoid-valued $\sI$-FCP's.
The operad structure maps of $\text{Emb}_V$ and $\sK_V$ are induced by the monoid 
products, as is clear from the specification of $\ga$ after (\ref{gamma}).  

The essential point is that, in analogy with (\ref{gamma}), we have maps
\begin{equation}\label{lambda}
\la\colon \sI(V_1\oplus\cdots \oplus V_k, W)\times 
\text{Emb}_{V_1}(j_1)\times \cdots \times \text{Emb}_{V_k}(j_k)\rtarr 
\text{Emb}_{W}(j_1 + \cdots + j_k)
\end{equation}
defined as follows. Let $g\colon V_1\oplus\cdots \oplus V_k\rtarr W$ be a 
linear isometric isomorphism and let 
$f_r=(f_{r,1}, \cdots, f_{r,j_r})\in \text{Emb}_{V_r}(j_r)$, $1\leq r\leq k$.
Again consider the set of sequences $Q=(q_1,\cdots,q_k)$, $1\leq q_r\leq j_r$,
ordered lexicographically.  Identifying direct sums with direct products, the 
$Q^{th}$ embedding of $\la(g;f_1,\cdots,f_k)$ is the composite 
$gf_{Q}g^{-1}$, where $f_Q = f_{1,q_1}\times \cdots \times f_{k,q_k}$.  
Restricting to distance reducing embeddings $f_{r,s}$ and applying the result 
pointwise to Steiner paths, there result maps
\begin{equation}\label{lambda2}
\la\colon \sI(V_1\oplus\cdots \oplus V_k, W)\times 
\sK_{V_1}(j_1)\times\cdots \times \sK_{V_k}(j_k)\rtarr 
\sK_{W}(j_1 + \cdots + j_k).
\end{equation}

Passing to colimits over inclusions $V\subset V'$ of subspaces in any given 
countably infinite dimensional inner product space, such as $U$, we obtain structure 
exactly like that just described, but now defined on all of $\sI_c$, rather 
than  just on $\sI$. (Compare \cite[I\S1 and VII\S2]{MQR}  and \cite[\S23.6]{MS}).
For example, suppose that the $V_r$ and $W$ in (\ref{lambda2}) are infinite dimensional.  
Since the spaces $\text{Emb}_{V_r}(1)$ are obtained by passage to colimits over the finite 
dimensional subspaces of the $V_r$, for each of the embeddings $f_{r,s}$, there is a finite 
dimensional subspace $A_{r,s}$ such that $f_{r,s}$ is the identity on the orthogonal
complement $V_r -A_{r,s}$.  Therefore, all of the $f_Q$ are the identity
on $V_1\oplus\cdots \oplus V_k - B$ for a sufficiently large $B$.  On
finite dimensional subspaces $gV\subset W$, we define $\la$ as before,
using the maps $gf_{Q}g^{-1}$.  On the orthogonal complement $W-gV$ for 
$V$ large enough to contain $B$, we can and must define the
$Q^{th}$ embedding to be the identity map for each $Q$.    

Finally, we define the canonical additive $E_{\infty}$ operad, denoted $\sC$, 
to be the Steiner operad $\sK_U$. Taking $V_1=\cdots = V_k = U$, we have the 
required maps 
\[ \la\colon \sL(k)\times \sC(j_1)\times \cdots \times \sC(j_k)
\rtarr \sC(j_1\cdots j_k).\]
They make the (unspecified) distributivity diagrams commute, and we next
explain the significance of those diagrams.

\section{The monadic interpretation of the definitions} \label{monad}

We assume that the reader has seen the definition of a monad.
Fixing a ground category $\sV$, a quick definition is that a monad 
$(C,\mu,\et)$ on $\sV$ is a monoid in the category of functors 
$\sV\rtarr \sV$.  Thus $C\colon \sV\rtarr\sV$ is a functor, 
$\mu\colon CC\rtarr C$ and $\et\colon \text{Id}\rtarr C$ are natural transformations, 
and $\mu$ is associative with $\et$ as a two-sided unit. A $C$-algebra is
an object $X\in\sV$ with a unital and associative action map 
$\xi\colon CX\rtarr X$.   
We let $C[\sV]$ denote the category of
$C$-algebras in $\sV$. 

Similarly, when $\sO$ is an operad in $\sV$, we write 
$\sO[\sV]$ for the category of $\sO$-algebras in $\sV$.
We note an important philosophical distinction.  Monads are very general, and
their algebras in principle depend only on $\sV$, without reference
to any further structure that category might have. In contrast, $\sO$-algebras
depend on a choice of symmetric monoidal structure on $\sV$, and that might vary.  
We sometimes write $\sO[\sV,\ten]$ to emphasize this dependence. 

For an operad $\sO$ of unbased spaces with $\sO(0)=*$, as before, there 
are two monads in sight, both of which have the same algebras 
as the operad $\sO$.\footnote{Further categorical perspective on the material
of this section is given in Appendix A (\S14).}  Viewing operads as acting on unbased spaces, 
we obtain a monad $O_{+}^{\sU}$ on $\sU$ with
\begin{equation}\label{AltI} 
O_{+}^{\sU}X = \coprod_{j\geq 0} \sO(j)\times_{\SI_j} X^j.
\end{equation}
Here $\et(x) = (1,x) \in \sO(1)\times X$, and $\mu$ is obtained by taking
coproducts of maps on orbits induced by the structure maps $\ga$.  If 
$X$ has an action $\tha$ by $\sO$, then $\xi\colon O_{+}^{\sU}X\rtarr X$ 
is given by the action maps 
$\tha\colon \sO(j)\times_{\SI_j}X^j\rtarr X$,
and conversely. The subscript $+$ on the monad is intended to indicate that
it is ``augmented'', rather than ``reduced''.  The superscript $\sU$ is
intended to indicate that the operad from which the monad is constructed
is an operad of unbased spaces.

Viewing operads  as acting on spaces with preassigned basepoints, we 
construct a reduced monad $O=O^{\sU}$ on 
$\sT$ by quotienting $O_{+}^{\sU}X$ by basepoint identifications.  There 
are degeneracy operations $\si_i\colon \sO(j)\rtarr \sO(j-1)$ given by
\[ \si_i(c) = \ga(c;(\id)^{i-1},*,(\id)^{j-i}) \]
for $1\leq i\leq j$, and there are maps $s_i\colon X^{j-1}\rtarr X^j$ 
obtained by inserting the basepoint in the $i^{th}$ position.  We set
\begin{equation}\label{AltII}
OX \equiv O^{\sU} X= O_{+}^{\sU}X/(\sim),
\end{equation} 
where $(c,s_iy)\sim (\si_i c, y)$ for
$c\in \sO(j)$ and $y\in X^{j-1}$.  Observations in \S1 explain why these two operads have the same algebras. The reduced monad $O=O^{\sU}$ is more general than the augmented monad $O_+^{\sU}$ since
the latter can be obtained by applying the former to spaces of the form $X_+$.  That is, for unbased 
spaces $X$,
\begin{equation}\label{Omore1}
O^{\sU}(X_+)\iso O_+^{\sU}X
\end{equation}
as $\sO$-spaces. To keep track of adjunctions later, we
note that the functor $(-)_+$ is left adjoint to the functor
$i\colon \sT\rtarr \sU$ that is obtained by forgetting basepoints.

The reduced monad $O=O^{\sU}$ on $\sT$ is of primary topological interest, but the idea that there is a choice will simplify the multiplicative theory. Here we diverge from the original sources \cite{MQR, Mult}.\footnote{I am indebted to helpful conversations with Bob Thomason and Tyler Lawson, some twenty-five years apart, for the changed perspective.}   Summarizing, we have the following result.

\begin{prop}\label{isocats1} The following categories are isomorphic.
\begin{enumerate}[(i)]
\item The category $\sO[\sU,\times] = \sO[\sT, \times]$ of $\sO$-spaces.
\vspace{1mm}
\item The category $O^{\sU}_+[\sU]$ of $O^{\sU}_+$-algebras in $\sU$.
\vspace{1mm}
\item The category $O[\sT] \equiv O^{\sU}[\sT]$ of $O$-algebras in $\sT$.
\end{enumerate}
\end{prop}

We have another pair of monads associated to an operad $\sO$.
Recall again that operads and operad actions make sense in any symmetric 
monoidal category $\sV$.  Above, in (i), we are thinking of $\sT$
as cartesian monoidal, and we are entitled to use the alternative notations 
$\sO[\sU]$ and $\sO[\sT]$ since $\sO$-algebras in $\sU$ can equally well be regarded as $\sO$-algebras in $\sT$. The algebras in Proposition \ref{isocats1} have parametrized products $X^j\rtarr X$ that are defined on cartesian powers $X^j$.

However, we can change ground categories to the symmetric monoidal 
category $\sT$ under its smash product, with unit $S^0$.  We write 
$X^{(j)}$ for the $j$-fold smash power of a space (or, later,
spectrum) $X$, with $X^{(0)} = S^0$. Remembering that 
$X_+\sma Y_+\iso (X\times Y)_+$, we can adjoin disjoint basepoints to the
spaces $\sO(j)$ to obtain an operad $\sO_{+}$ with spaces $\sO_+(j) = \sO(j)_+$  
in $\sT$; in particular, $\sO_{+}(0) = S^0$.
The actions of $\sO_{+}$ parametrize products $X^{(j)}\rtarr X$, and we have
the category $\sO_{+}[\sT]$ of $\sO_{+}$-spaces in $\sT$.  

Recall that $\sT_e$ denotes the category of spaces $X$ under $S^0$, with given map
$e\colon S^0\rtarr X$. 
In \cite{MQR,Mult}, we defined an $\sO$-space with zero, or $\sO_0$ space, 
to be an $\sO$-space $(X,\xi)$ in $\sT_e$ 
such that $0$ acts as a strict 
zero, so that $\xi(f;x_1,\cdots,x_j)= 0$ if any $x_i=0$.  That is exactly 
the same structure as an $\sO_+$-space in $\sT$.  The only difference is that now we think of $\xi\colon S^0 = \sO_+(0)\rtarr X$ as building in the map $e\colon S^0\rtarr X$, which is no longer preassigned.  
We are entitled to use the alternative notation 
$\sO_+[\sT_e]$ for $\sO_{+}[\sT]$.

We construct an (augmented) monad $O_{+} = O_+^{\sT}$ on $\sT$ such that an 
$\sO_{+}$-space in $\sT$ is the same as an $O_{+}$-algebra in $\sT$ by 
setting
\begin{equation}\label{AltIII} 
O_{+}X \equiv O_{+}^{\sT}X = \bigvee_{j\geq 0} \sO(j)_+\sma_{\SI_j} X^{(j)}. 
\end{equation}
This and (\ref{AltI}) are special cases of a general definition that applies to operads
in any cocomplete symmetric monoidal category, as is
discussed in detail in \cite{MayDef, MayDef2}.  

As a digression, thinking homotopically and letting $E\SI_j$ be any contractible free $\SI_j$-space, one defines the extended $j$-fold
smash power $D_jX$ of a based space $X$ by 
\begin{equation}\label{DJX}  D_jX = (E\SI_j)_+\sma _{\SI_j} X^{(j)}. 
\end{equation}
These spaces have many applications.  Homotopically, when $\sO$ is an
$E_{\infty}$ operad, $O_{+}X$ is a model for the wedge over $j$ of the 
spaces $D_jX$. 

Here we have not viewed the element $1$ of a space under $S^0$ as a basepoint.
However, we can use basepoint identifications to take account of the unit properties of $1$ in an action by $\sO_+$.  We then obtain a reduced monad 
$O^{\sT}$ on $\sT_e$ with the same algebras as the monad $O_+^{\sT}$ on $\sT$. It is again more general than the monad $O_+$.  For a based space $X$, 
$S^0\wed X$ is the space under $S^0$ obtained from the based space $X$ by adjoining a point $1$ (not regarded as a basepoint).  This gives
the left adjoint to the inclusion $i\colon \sT_{e}\rtarr \sT$ that is obtained
by forgetting the point $1$, and
\begin{equation}\label{Omore2}
O^{\sT}(S^0\wed X)\iso O^{\sT}_+(X)
\end{equation}
as $\sO_+$-spaces. We summarize again.

\begin{prop}\label{isocats2} The following categories are isomorphic.
\begin{enumerate}[(i)]
\item The category $\sO_+[\sT,\sma] = \sO_+[\sT_e,\sma]$ of $\sO_+$-spaces.
\vspace{1mm}
\item The category $O_+[\sT]\equiv O^{\sT}_+[\sT] $ of $O_+$-algebras in $\sT$.
\vspace{1mm}
\item The category $O^{\sT}[\sT_e]$ of $O^{\sT}$-algebras in $\sT_e$.
\end{enumerate}
\end{prop}

In \cite{MQR, Mult}, we viewed the multiplicative structure of $E_{\infty}$ 
ring spaces as defined on the base category $\sT_e$, and we used the monad 
$G^{\sT}$ on $\sT_e$ instead of the monad $G_+$ on $\sT$.  That unnecessarily 
complicated the details in \cite{Mult}, where different kinds of input data 
to multiplicative infinite loop space theory were compared, and we now prefer
to use $G_+$.  That is convenient both here and in the sequel \cite{Sequel}.

With this as preamble, consider an operad pair $(\sC,\sG)$, such as 
the canonical one $(\sC,\sL)$ from the previous section.  We have several 
monads in sight whose algebras are the $(\sC,\sG)$-spaces $X$.  We single
out the one most convenient for the comparison of $E_{\infty}$ ring spaces
and $E_{\infty}$ ring spectra by focusing on the ``additive monad'' $C$ 
on $\sT$, where the basepoint is denoted $0$ and is the unit for the 
operadic product, and the ``multiplicative monad'' $G_+$, which is also 
defined on $\sT$.   A different choice will be more convenient in the
sequel \cite{Sequel}. 

The diagrams that we omitted in our outline definition of an action of $\sG$ 
on $\sC$ are exactly those that are needed to make the following result true.

\begin{prop}\label{CmonSpace} 
The monad $C$ on $\sT$ restricts to a monad on the category 
$\sG_+[\sT]$ of $\sG_{+}$-spaces in $\sT$. A $(\sC,\sG)$-space is 
the same structure as a $C$-algebra in $\sG_{+}[\sT]$.
\end{prop}
\begin{proof}[Sketch proof] The details of the definition of an action
of $\sG$ on $\sC$ are designed to ensure that, for a $\sG_{+}$-space 
$(X,\xi)$, the maps $\xi$ of (\ref{ActII}) induce maps 
\[ \xi\colon \sG(k)_+\sma (CX)^{(k)}\rtarr CX \]
that give $CX$ a structure of $\sG_{+}$-space in $\sT$ such that 
$\mu\colon CCX\rtarr CX$ and $\et\colon X\rtarr CX$ are maps of
$\sG_+$-spaces in $\sT$.  Then the diagram (\ref{ActI}) asserts
that a $(\sC,\sG)$-space is the same as a $C$-algebra in $\sG_{+}[\sT]$.   
Details are in \cite[VI\S1]{MQR}. 
\end{proof}

We have the two monads $(C,\mu_{\oplus},\et_{\oplus})$ and 
$(G_{+},\mu_{\otimes},\et_{\otimes})$, both defined on $\sT$,
such that $C$ restricts to a monad on $\sG_{+}[\sT]$.   This 
puts things in a general categorical context that was studied 
by Beck \cite{Beck}.  A summary of his results is given 
in \cite[5.6]{Mult} and in \cite[App.B (\S15)]{Sequel}.  He gives several 
equivalent characterizations
of a pair $(C,G_{+})$ of monads related in this fashion.
One is that the composite functor $CG_{+}$ is itself
a monad with certain properties.  Another is that there is a 
natural interchange transformation $G_{+}C\rtarr CG_{+}$ such that 
appropriate diagrams commute. Category theorists know well that 
this is definitively the right context in which to study generalized 
distributivity phenomena.  While I defined $E_{\infty}$ ring
spaces before I knew of Beck's work, his context makes it clear
that this is a very natural and conceptual definition.

\section{The definition of $E_{\infty}$ ring prespectra and $E_{\infty}$ ring spectra}\label{ringspec}

We first recall the categories of (LMS) prespectra and spectra from \cite{LMS}.
As before, we let $U = \bR^{\infty}$ with its standard inner product. Define an indexing space to be a finite dimensional subspace of $U$ with the induced inner product. A (coordinate free) prespectrum $T$ consists of based spaces $TV$ and based maps
$\tilde{\si}\colon TV\rtarr \OM^{W-V}TW$ for $V\subset W$, where $W-V$ is the orthogonal complement of $V$ and $S^{W-V}$ is its one point compactification; 
$\tilde{\si}$  must be the identity when $V=W$, and the obvious transitivity condition must hold when $V\subset W\subset Z$.  A spectrum is a prespectrum 
such that each map $\tilde{\si}$ is a homeomorphism; we then usually write $E$ rather than $T$.  

We let $\sP$ and $\sS$
denote the categories of prespectra and spectra.  Then $\sS$ is a full 
subcategory of $\sP$, with inclusion $\ell\colon \sS\rtarr \sP$, and there
is a left adjoint spectrification functor $L\colon \sP\rtarr \sS$.   When
$T$ is an inclusion prespectrum, meaning that each $\tilde{\si}$ is an
inclusion, $(LT)(V) = \colim_{V\subset W} \OM^{W-V} TW$.

We may restrict attention to any cofinal set of indexing spaces ${\sV}$ in $U$;
we require $0$ to be in $\sV$ and we require the union of the $V$ in $\sV$ 
to be all of $U$. Up to isomorphism, the category $\sS$ is independent of the
choice of $\sV$.  The default is $\sV = \sA\ell\ell$.  We can define 
prespectra and spectra in the same way in any countably infinite dimensional real inner product space $U$, and we write $\sP(U)$ when we wish to remember 
the universe.  The default is $U=\bR^{\infty}$.

For $X\in \sT$, we define $QX$ to be $\colim \OM^V\SI^V X$, and we let
$\et\colon X\rtarr QX$ be the natural inclusion.  We define 
$(\SI^{\infty}X)(V) = Q\SI^V X$.  Since $S^W\iso S^V\sma S^{W-V}$, we have 
identifications $Q\SI^V X\iso \OM^{W-V} Q\SI^W X$, so that $\SI^{\infty}X$
is a spectrum.  For a spectrum $E$, we define $\OM^{\infty}E = E(0)$; we often write it as $E_0$.  The functors $\SI^{\infty}$ 
and $\OM^{\infty}$ are adjoint, $QX$ is $\OM^{\infty}\SI^{\infty}X$, and $\et$
is the unit of the adjunction.  We let 
$\epz\colon \SI^{\infty}\OM^{\infty}E\rtarr E$ be the counit of 
the adjunction.  

As holds for any adjoint pair of functors, we have a monad 
$(Q,\mu,\et)$ on $\sT$, where 
$\mu\colon QQ\rtarr Q$ is
$\OM^{\infty}\epz\SI^{\infty}$, and the functor $\OM^{\infty}$ 
takes values in $Q$-algebras via the action map
$\OM^{\infty}\epz\colon 
\OM^{\infty}\SI^{\infty}\OM^{\infty}E \rtarr \OM^{\infty}E$.
These observations are categorical trivialities, but they are central
to the theory and will be exploited heavily.  We will see later that this adjunction restricts
to an adjunction in $\sL_+[\sT]$. That fact will lead to the proof that the
$0^{th}$ space of an $E_{\infty}$ ring spectrum is an $E_{\infty}$ 
ring space.

As already said, the starting point of this area of mathematics was the observation
that the classical Thom prespectra, such as $TU$, appear in nature as $\sI$-FSP's
and therefore as $E_{\infty}$ ring prespectra in the sense we are about to describe.
To distinguish, we will write $MU$ for the corresponding spectrum.
Quite recently, Strickland \cite[App]{Strick} has observed similarly
that the periodic Thom spectrum, $PMU = MU[x,x^{-1}]$, also arises in 
nature as the spectrum associated to an $E_{\infty}$ ring prespectrum.
To define this concept, we must consider smash products and change of universe.

We have an external smash product 
$T\barwedge T'$ that takes a pair of prespectra indexed on $\sA\!\ell\ell$ 
in $U$ to a prespectrum indexed on the indexing spaces 
$V\oplus V'$ in $U\oplus U$.
It is specified by 
\[ (T\barwedge T')(V,V') = TV\sma T'V' \]
with evident structure maps induced by those of $T$ and $T'$.  This 
product is commutative and associative in an evident sense; for the
commutativity, we must take into account the interchange isomorphism
$\ta\colon U\oplus U\rtarr U\oplus U$.  If we think of spaces as 
prespectra indexed on $0$, then the space $S^0$ is a unit object.
Formally, taking the disjoint union over $j\geq 0$ of the categories 
$\sP(U^j)$, we obtain a perfectly good symmetric monoidal category.  
This was understood in the 1960's.  A well-structured treatment 
of spectra from this external point of view was given by Elmendorf \cite{Elm}. 

For a linear isometry $f\colon U\rtarr U'$, we have an obvious change of universe
functor $f^*\colon \sP(U')\rtarr \sP(U)$ specified by $(f^*T')(V) = T'(fV)$,
with evident structure maps.  It restricts to a change of universe functor
$f^*\colon \sS(U')\rtarr \sS(U)$.  These functors have left adjoints $f_*$.
When $f$ is an isomorphism, $f_* = (f^{-1})^*$.  For a general $f$, it is 
not hard to work out an explicit definition; see \cite[II\S1]{LMS}.  The 
left adjoint on the spectrum level is $L f_*\ell$.
This is the way one constructs a spectrum level left adjoint from any
prespectrum level left adjoint.  The external smash product of spectra
is defined similarly, $E\barwedge E' = L(\ell E\barwedge \ell E')$. 

The reader should have in mind the Thom spaces $TO(V)$ or, using
complex inner product spaces, $TU(V)$ associated to well chosen 
universal $V$-plane bundles.  In contrast to the original definitions
of  \cite{MQR, LMS}, we restrict attention to the linear isometries 
operad $\sL$. There seems to be no useful purpose in
allowing more general operads in this exposition. 

We agree to write $T^{[j]}$ for external $j$-fold smash powers.
An $E_{\infty}$ prespectrum, or $\sL$-prespectrum, $T$ has an action 
of $\sL$ that is
specified by maps of prespectra
\[ \xi_j(f)\colon T^{[j]}\rtarr f^*T \]
or, equivalently, $f_*T^{[j]}\rtarr T$,
for all $f\in \sL(j)$ that are suitably continuous and are 
compatible with the operad structure on $\sL$.  The compatibility 
conditions are that  
$\xi_j(f\ta) = \xi_j(f)\com \ta_*$ for $\ta\in \SI_j$ (where
$\ta$ is thought of as a linear isomorphism $U^j\rtarr U^j$), 
$\xi_1(\id) = \id$, and
\[ \xi_{j_1+\cdots + j_k}(\ga(g;f_1,\cdots, f_k)) = \xi_k(g)
\com (\xi_{j_1}(f_1)\barwedge \cdots \barwedge \xi_{j_k}(f_k))\]
for $g\in\sL(k)$ and $f_r\in\sL(j_r)$.  

For the continuity condition,
let $V=V_1\oplus\cdots\oplus V_j$ and let $A(V,W)\subset \sL(j)$ be the 
subspace of linear isometries $f$ such that $f(V)\subset W$, where
the $V_r$ and $W$ are indexing spaces. We have a map 
$A(V,W)\times V\rtarr  A(V,W)\times W$ of bundles
over $A(V,W)$ that sends $(f,v)$ to $(f, f(v))$.  Its image is a 
subbundle, and we let $T(V,W)$ be the Thom
complex of its complementary bundle.  Define a function
\[ \ze=\ze(V,W)\colon T(V,W)\sma TV_1\sma\cdots\sma TV_j\rtarr TW \]
by 
\[ \ze((f, w),y_1,\cdots, y_j)) = \si(\xi_j(f)(y_1\sma\cdots\sma y_j)\sma w) \]
for $f\in A(V,W)$, $w\in W-f(V)$, and $y_r\in TV_r$; on the right, $\si$ is 
the structure map $T(fV)\sma S^{W-fV}\rtarr TW$.  The functions $\ze(V,W)$
must all be continuous. 

This is a very simple notion. As already noted in \S\ref{isoop}, it is easy to see that 
$\sI$-FSP's and thus Thom prespectra give examples.  However, the continuity condition 
requires a more conceptual description. We want to think of the maps $\xi_j(f)$ 
as $j$-fold products parametrized by $\sL(j)$, and we want to collect
them together into a single global map.  The intuition is that we should 
be able to construct a ``twisted half-smash product prespectrum'' 
$\sL(j)\thp T^{[j]}$ indexed on $U$ by setting
\[ (\sL(j)\thp T^{[j]})(W) = T(V,W)\sma TV_1\sma\cdots \sma TV_j. \]
This doesn't quite make sense because of the various possible choices 
for the $V_r$, but it does work in exactly this way if we choose appropriate 
cofinal sequences of indexing spaces in $U^j$ and $U$.  The resulting
construction $L (-) \ell$ on the spectrum level is independent
of choices.   Another intuition is that we are suitably topologizing 
the disjoint union over $f\in\sL(j)$ of the prespectra $f_*T^{[j]}$ indexed on $U$.

These intuitions are made precise in \cite[VI]{LMS} and, more conceptually
but perhaps less intuitively, \cite[App]{EKMM}.  The construction of 
twisted half-smash products of spectra is the starting point of the EKMM approach to the stable homotopy category \cite{EKMM}, but for now we are 
focusing on earlier work.  With this construction in place, we have an equivalent definition of an 
$\sL$-prespectrum in terms of action maps
\[\xi_j\colon \sL(j)\thp T^{[j]}\rtarr T \]
such that equivariance, unit, and associativity diagrams commute. The
diagrams are exactly like those in the original definition of an action 
of an operad on a space.  

In more detail, and focusing on the spectrum level, the construction of
twisted half-smash products actually gives functors 
\begin{equation}
A\thp E_1\barwedge \cdots \barwedge E_j
\end{equation}
for any spectra $E_r$ indexed on $U$ and any map $A\rtarr \sL(j)$.  There
are many relations among these functors as $j$ and $A$ vary. In
particular there are canonical maps
\begin{equation}\label{butyes} 
\xymatrix{
\sL(k)\thp (\sL(j_1)\thp E^{[j_1]}\barwedge \cdots \barwedge
\sL(j_k)\thp E^{[j_k]}) \ar[d]^{\iso} \\
(\sL(k)\times \sL(j_1)\times \cdots \times \sL(j_k))\thp E^{[j]}
\ar[d]^{\ga\thp\id} \\
\sL(j)\thp E^{[j]},\\}
\end{equation}
where $j= j_1 +\cdots + j_k$.  Using such maps we can make sense out
of the definition of an $E_{\infty}$ ring spectrum in terms of an
action by the operad $\sL$. 

\begin{defn} An $E_{\infty}$ ring spectrum, or $\sL$-spectrum, is a 
spectrum $R$ with an
action of $\sL$ given by an equivariant, unital, and associative system
of maps
\begin{equation}\label{Eringspec}
\sL(j)\thp R^{[j]}\rtarr R.
\end{equation}
\end{defn}

\begin{lem}
If $T$ is an $\sL$-prespectrum, then $LT$ is an $\sL$-spectrum.
\end{lem}

This holds since the spectrification functor $L\colon \sP\rtarr \sS$ satisfies
\[ L(\sL(j)\thp T^{[j]})\iso \sL(j)\thp (LT)^{[j]}. \]

\section{The monadic interpretation of $E_{\infty}$ ring spectra}\label{monad2}

At this point, we face an unfortunate clash of notations, and for the moment the reader is asked to forget all about prespectra and the spectrification functor $L\colon \sP\rtarr \sS$.  We focus solely on spectra in this section. 

In analogy with \S\ref{monad}, we explain that there are two monads in sight,
both of whose algebras coincide with $E_{\infty}$ ring spectra.\footnote{Again, further categorical perspective is supplied in Appendix A (\S14).} One is in
\cite{LMS}, but the one more relevant to our current explanations is new.

In thinking about external smash products, we took spectra indexed on
zero to be spaces.   Since $\sL(0)$ is the inclusion $0\rtarr U$, the
change of universe functor $\sL(0)\thp (-)$ can and must be interpreted
as the functor $\SI^{\infty}\colon \sT\rtarr \sS$.  Similarly, the zero fold 
external smash power $E^{[0]}$ should be interpreted as the space $S^0$.  
Since $\SI^{\infty}S^0$ is the sphere spectrum $S$, the $0^{th}$ 
structure map in (\ref{Eringspec}) is a map $e\colon S\rtarr R$.  We have 
the same dichotomy as in \S\ref{monad}.  We can either think of the map $e$ 
as preassigned, in which case we think of our ground category as the
category $\sS_e$ of spectra under $S$, 
or we can think of $e =\xi_0$ as part of the structure of an $E_{\infty}$
ring spectrum, in which case we think of our ground category as $\sS$.

In analogy with the space level monad $L_+$ we define a monad $L_+$
on the category of spectra by letting
\[  L_+ E = \bigvee_{j\geq 0} \sL(j)\thp_{\SI_j} E^{[j]}. \]
The $0^{th}$ term is $S$, and $\et\colon S\rtarr L_+E$ is the inclusion.
The product $\mu$ is induced by passage to orbits and wedges from the canonical maps (\ref{butyes}).  

We also have a reduced monad $L$ defined on the category $\sS_e$.  
The monad $L$ on $\sT_e$ is obtained from the
monad $L_+$ on $\sT$ by basepoint identifications.
The construction can be formalized in terms of coequalizer diagrams.
We obtain the analogous monad $L$ on $\sS_e$ by ``base sphere'' 
identifications that are formalized by precisely analogous coequalizer 
diagrams \cite[VII\S3]{LMS}.  In analogy with (\ref{Omore2}), the spectrum level
monads $L$ and $L_+$ are related by a natural 
isomorphism
\begin{equation}\label{Omore3}
L(S\wed E)\iso L_+E.
\end{equation} 
This isomorphism, like (\ref{Omore2}), is monadic.  This means that the 
isomorphisms are compatible with the structure maps of the two monads, 
as is made precise in Definition \ref{moniso}.  
The algebras of $L$ are the same as those of $L_+$, and we have the
following analogue of Propositions \ref{isocats1} and \ref{isocats2}. 

\begin{prop}\label{isocats3} The following categories are isomorphic.
\begin{enumerate}[(i)]
\item The category $\sL[\sS] = \sL[\sS_e]$ of $\sL$-spectra.
\vspace{1mm}
\item The category $L_+[\sS]$ of $L_+$-algebras in $\sS$.
\vspace{1mm}
\item The category $L[\sS_e]$ of $L$-algebras in $\sS_e$.
\end{enumerate}
\end{prop}

A central feature of twisted half-smash products is that there is a natural untwisting isomorphism
\begin{equation}\label{untwist}
A\thp (\SI^{\infty} X_1\barwedge \cdots \barwedge \SI^{\infty} X_j)
\iso  
\SI^{\infty}(A_+\sma X_1\sma\cdots X_j).
\end{equation}
for any space $A$ over $\sL(j)$.  Using this, we obtain a monadic natural isomorphism
\begin{equation}\label{newmon}
L_+\SI^{\infty}X \iso \SI^{\infty} L_+X
\end{equation}
relating the space and spectrum level monads $L_+$.

It has become usual to compress notation by defining 
\begin{equation}\label{sigmaplus}
 \SI^{\infty}_+ X = \SI^{\infty}(X_+)
\end{equation}
for a space $X$, ignoring its given basepoint if it has one. 
When $X$ has a (nondegenerate) basepoint, the equivalence 
$\SI(S^0\wed X)\htp S^1\wed \SI X$ implies an 
equivalence\footnote{In contrast with the isomorphisms appearing
in this discussion, this equivalence plays no role in our
formal theory; we only recall it for use in a digression that we are about to insert.}
\begin{equation}\label{homowed}
 \SI^{\infty}_+X\htp S\wed \SI^{\infty}X
\end{equation}
under $S$. With the notation of (\ref{sigmaplus}), the relationship between the space and spectrum level monads $L$ is given by a monadic natural isomorphism
\begin{equation}\label{oldmon}
L\SI^{\infty}_+ X \iso \SI^{\infty}_+ LX.
\end{equation}
(See \cite[VII.3.5]{LMS}). Remarkably, as we shall show in Appendix A (\S14), 
the commutation relations (\ref{Omore3}), (\ref{newmon}), and (\ref{oldmon}) 
are formal, independent of calculational understanding of the monads in question.

We digress to recall a calculationally important 
splitting theorem that is implicit in these formalities.  
Using nothing but (\ref{Omore3}) and (\ref{untwist}) -- (\ref{oldmon}), we find
\begin{eqnarray*}
S\wed \SI^{\infty}LX & \htp &  \SI^{\infty}_+ LX \\
& \iso &  L\SI^{\infty}_+ X\\
& \htp & L(S\wed \SI^{\infty}X)\\
& \iso & L_+\SI^{\infty}X \\
& = &\bigvee_{j\geq 0} \sL(j)\thp_{\SI_j} (\SI^{\infty}X)^{[j]} \\
&\iso & S \wed \bigvee_{j\geq 1} \SI^{\infty}(\sL(j)_+\sma_{\SI_j} X^{(j)}).
\end{eqnarray*}
Quotienting out $S$ and recalling (\ref{DJX}), we obtain a natural equivalence
\begin{equation}\label{splittoo}
\SI^{\infty} LX \htp \bigvee_{j\geq 1} \SI^{\infty}D_jX.
\end{equation}
We may replace $LX$ by $CX$ for any other $E_{\infty}$ operad $\sC$, such
as the Steiner operad, and then $CX\htp QX$ when $X$ is connected.  Thus,
if $X$ is connected,
\begin{equation}
\SI^{\infty} QX\htp \bigvee_{j\geq 1} \SI^{\infty}D_jX.
\end{equation}
This beautiful argument was discovered by Ralph Cohen \cite{RCohen}. 
As explained in \cite[VII\S5]{LMS}, we can exploit the projection
$\sO\times \sL\rtarr \sL$ to obtain a splitting theorem for $OX$ 
analogous to (\ref{splittoo}), where $\sO$ is any operad, not necessarily 
$E_{\infty}$. 

\section{The relationship between $E_{\infty}$ ring spaces and
$E_{\infty}$ ring spectra}\label{spspcompare}

We show that the $0^{th}$ space of an $E_{\infty}$ ring spectrum
is an $E_{\infty}$ ring space.  This observation is at the conceptual 
heart of what we want to convey.  It was central to the 1970's 
applications, but it seems to have dropped off the radar screen,
and this has led to some confusion in the modern literature.  
One reason may be that the only proof ever published is in the original 
source \cite{MQR}, which preceded the definition of twisted half-smash 
products and the full understanding of the category of spectra. Since 
this is the part of \cite{MQR} that is perhaps most obscured by later 
changes of notations and definitions, we give a cleaner updated treatment
that gives more explicit information.

Recall that $\sL_+[\sT]\iso L_+[\sT]$ and $\sL[\sS]\iso L_+[\sS]$ 
denote the categories of $\sL_+$-spaces, or $\sL$-spaces with zero, 
and of $\sL$-spectra, thought of as identified with the categories 
of $L_+$-algebras in $\sT$ and of $L_+$-algebras in $\sS$.
To distinguish, we write $X$ for based spaces, $Z$ for $\sL_+$-spaces, 
$E$ for spectra, and $R$ for $\sL$-spectra. 

\begin{prop}\label{formalQ} The (topological) adjunction
\[  \sS(\SI^{\infty}X,E)\iso \sT(X,\OM^{\infty}E) \]
induces a (topological) adjunction
\[  L_+[\sS](\SI^{\infty}Z,R) \iso L_+[\sT](Z,\OM^{\infty}R). \]
Therefore the monad $Q$ on $\sT$ restricts to a monad $Q$ on 
$L_+[\sT]$ and, when restricted to $\sL$-spectra, the 
functor $\OM^{\infty}$ takes values in algebras over $L_+[\sT]$.
\end{prop}  
\begin{proof} This is a formal consequence of the fact that the
isomorphism (\ref{newmon}) is monadic, as is explained in
general categorical terms in Proposition \ref{Appcute}.
If $(Z,\xi)$ is an $L_+$-algebra, 
then $\SI^{\infty}Z$ is an $L_+$-algebra with structure map
\[ \xymatrix@1{ L_+\SI^{\infty}Z \ar[r]^-{\iso} & 
\SI^{\infty}L_+Z \ar[r]^-{\SI^{\infty}\xi}
& \SI^{\infty}Z.\\} \]
The isomorphism (\ref{newmon}) and the 
adjunction give a natural composite $\de$: 
\[\xymatrix@1{  
L_+\OM^{\infty}E \ar[r]^-{\et} & 
\OM^{\infty}\SI^{\infty}L_+\OM^{\infty}E \ar[r]^-{\iso} 
& \OM^{\infty}L_+\SI^{\infty}\OM^{\infty}E \ar[r]^-{\OM^{\infty}L_+\epz}
& \OM^{\infty}L_+E.\\} \]
If $(R,\xi)$ is an $L_+$-algebra, then 
$\OM^{\infty}R$ is an $L_+$-algebra with structure map
\[ \xymatrix@1{
L_+\OM^{\infty}R \ar[r]^-{\de} & \OM^{\infty}L_+ R\ar[r]^-{\OM^{\infty}\xi} 
& \OM^{\infty}R.\\} \]
Diagram chases show that the unit $\et$ and counit $\epz$ of the
adjunction are maps of $L_+$-algebras when $Z$ and $R$ are $L_+$-algebras.
The last statement is another instance of an already cited categorical triviality
about the monad associated to an adjunction,  
and we shall return to the relevant categorical observations
in the next section.
\end{proof}

In the notation of algebras over operads, this has the following consequence. 

\begin{cor}\label{themguys} The adjunction of Proposition \ref{formalQ} induces an adjunction
\[ \sL[\sS](\SI^{\infty}_+Y,R)\iso \sL[\sU](Y,\OM^{\infty}R) \]
between the category of $\sL$-spaces and the category of $\sL$-spectra.
\end{cor}
\begin{proof} Recall that the functor $i\colon \sT\rtarr \sU$ given by
forgetting the basepoint induces an isomorphism $\sL[\sT] \iso \sL[\sU]$ 
since maps of $\sL$-spaces must preserve the basepoints given by the operad 
action.  Also, since adjoining a disjoint basepoint $0$ to an $\sL$-space 
$Y$ gives an $\sL$-space with zero, or $\sL_+$-space in $\sT$, while forgetting
the basepoint $0$ of an $\sL_+$-space in $\sT$ gives an $\sL$-space in $\sU$, 
the evident adjunction
\[  \sT(U_+,X) \iso \sU(U,iX) \]
for based spaces $X$ and unbased spaces $U$ induces an adjunction
\[ \sL_+[\sT](Y_+,Z)\iso \sL[\sU](Y,Z) \]
for $\sL$-spaces $Y$ and $\sL_+$-spaces $Z$.  Taking $Z=\OM^{\infty}R$,
the result follows by composing with the adjunction of Proposition \ref{formalQ}.
\end{proof}

Now let us bring the Steiner operads into play.  For an indexing space 
$V\subset U$,
$\sK_V$ acts naturally on $V$-fold loop spaces $\OM^V Y$.  These actions
are compatible for $V\subset W$, and by passage to colimits we obtain
a natural action $\tha$ of the Steiner operad $\sC = \sK_U$ on $\OM^{\infty}E$
for all spectra $E$.  For spaces $X$, we define $\al\colon CX\rtarr QX$ to be
the composite
\begin{equation}\label{alpha}
\xymatrix@1{
CX \ar[r]^-{C\eta} &  C\OM^{\infty}\SI^{\infty}X \ar[r]^-{\tha} &
\OM^{\infty}\SI^{\infty}X = QX.}\\  
\end{equation}
Another purely formal argument shows that $\al$ is a map of monads in $\sT$ 
\cite[5.2]{Geo}.  This observation is central to the entire theory.

We have seen in Propositions \ref{CmonSpace} and \ref{formalQ} that $C$ and $Q$
also define monads on $\sL_+[\sT]$, and we have the following crucial compatibility.  

\begin{prop}\label{CmonSpace2} The map $\al\colon C\rtarr Q$ of 
monads on $\sT$ restricts to a map of monads on $\sL_+[\sT]$.
\end{prop}
\begin{proof}[Sketch proof] 
We must show that $\al\colon CX\rtarr QX$ is a map of $\sL_+$-spaces
when $X$ is an $\sL_+$-space.  Since it is clear by naturality that
$C\et\colon CX\rtarr CQX$ is a map of $\sL_+$-spaces, it suffices to
show that $C\tha\colon CQX\rtarr QX$ is a map of $\sL_+$-spaces.  We
may use embeddings operads rather than Steiner operads since the 
action of $\sK_V$ on $\OM^VY$ is obtained by pullback of 
the action of $\text{Emb}_V$.  The argument given for the 
little convex bodies operad in \cite[VII.2.4 (p. 179)]{MQR} applies
verbatim here.  It is a passage to colimits argument similar to that
sketched at the end of \S\ref{Steiner}. 
\end{proof}

Recall that $(\sC,\sL)$-spaces are the same as $C$-algebras in $\sL_+[\sT]$,
and these are our $E_{\infty}$ ring spaces.  Similarly, our $E_{\infty}$ ring
spectra are the same as $\sL$-spectra, and the $0^{th}$ space functor takes
$\sL$-spectra to $Q$-algebras in $\sL_+[\sT]$.  By pullback along $\al$,
this gives the following promised conclusion.

\begin{cor} 
The $0^{th}$ spaces of $E_{\infty}$ ring spectra are
naturally $E_{\infty}$ ring spaces.
\end{cor}

In particular, the $0^{th}$ space of an $E_{\infty}$ ring spectrum is both
a $\sC$-space and an $\sL$-space with $0$.  The interplay between the Dyer-Lashof homology operations constructed from the two operad actions is essential to the caculational applications, and for that we must 
use all of the components.  However, to apply infinite loop space theory
using the multiplicative operad $\sL$, we must at least delete the 
component of $0$, and it makes sense to also delete the other non-unit 
components.

\begin{defn}\label{GLE} The $0^{th}$ space $R_0$ of an (up to homotopy) ring spectrum $R$ is an (up to homotopy) ring space, and $\pi_0(\OM^{\infty}R)$ is a ring.  Define $GL_1R$ to be the subspace of $R_0$ that 
consists of the components of the unit elements.  Define $SL_1R$ to be the 
component of the identity 
element.\footnote{To repeat, these spaces were introduced in \cite{MQR}, where they were denoted 
$FR$ and $SFR$.}
For a space $X$, $[X_+,GL_1R]$ is the group of units in the (unreduced) cohomology 
ring $R^0(X)$.
\end{defn}

\begin{cor}\label{GLR} If $R$ is an $E_{\infty}$ ring spectrum, then 
the unit spaces $GL_1R$ and $SL_1R$ are $\sL$-spaces.
\end{cor}

Again, we emphasize how simply and naturally these definitions fit together.

\section{A categorical overview of the recognition principle}\label{phil}

The passage from space level to spectrum level information through the
black box of (\ref{DIAG}) admits a simple conceptual outline. We explain 
it here.\footnote{I am indebted to Saunders Mac\,Lane, Gaunce Lewis, and 
Matt Ando for ideas I learned from them some thirty-five years apart.}

We consider two categories, $\sV$ and $\sW$, with an adjoint pair
of functors $(\XI, \LA)$ between them.\footnote{$\XI$ is the capital
Greek letter
Xi; $\XI$ and $\LA$ are meant to look a little like $\SI$ and $\OM$.}  We write 
\[ \et\colon \text{Id}\rtarr \LA\XI \ \ \text{and}\ \ \epz\colon 
\XI\LA \rtarr \text{Id} \]
for the unit and counit of the adjunction.  The reader should be
thinking of $(\SI^n,\OM^n)$, where $\sV = \sT$ and $\sW$ is the
category of $n$-fold loop sequences $\{\OM^iY|0\leq i\leq n\}$ and 
maps $\{\OM^if|0\leq i\leq n\}$.  This is a copy of $\sT$, but we
want to remember how it encodes $n$-fold loop spaces.  It is analogous
and more relevant to the present theory to think of 
$(\SI^{\infty},\OM^{\infty})$, where $\sV= \sT$ and $\sW = \sS$.  

As we have already noted several times, we have the monad
\[  (\LA\XI, \mu , \et) \]
on $\sV$, where $\mu = \LA\epz\XI$.  We also have a right action of the 
monad $\LA\XI$ on the 
functor $\XI$ and a left action of $\LA\XI$ on $\LA$.  These are given by
the natural maps 
\[\epz \XI\colon \XI\LA\XI\rtarr \XI  \ \ \text{and}\ \  
\LA\epz\colon \LA\XI\LA\rtarr \LA.  \]
Actions of monads on functors satisfy unit and associativity diagrams just 
like those that define the action of a monad on an object.  If we think of
an object $X\in\sV$ as a functor from the trivial one object category 
$\ast$ to $\sV$, then an action of a monad on the object $X$ is the same as
a left action of the monad on the functor $X$.

Now suppose that we also have some monad $C$ on $\sV$ and a map of monads
$\al\colon C\rtarr \LA\XI$.  By pullback, we inherit a right action 
of $C$ on $\XI$ and a left action of $C$ on $\LA$, which we denote by
$\rh$ and $\la$.  Thus 
\[ \rh = \epz\XI\com \XI\al\colon \XI C\rtarr \XI \ \ \text{and} \ \ 
 \la = \LA\epz\com \al\LA\colon C\LA\rtarr \LA.\]
For a $C$-algebra $X$ in $\sV$, we seek an object $\bE X$ in $\sW$ such that
$X$ is weakly equivalent to $\LA \bE X$ as a $C$-algebra.  This is the general
situation addressed by our black box, and we first remind ourselves of how 
we would approach the problem if were looking for a categorical analogue.
We will reprove a special case of Beck's monadicity theorem \cite[VI\S7]{Mac}, 
but in a way that suggests how to proceed homotopically. 

Assume that $\sW$ is cocomplete.  For any right $C$-functor 
$\PS\colon \sV\rtarr \sW$ with right action $\rh$ and left $C$-functor $\PH\colon \sU\rtarr \sV$ with left action $\la$, where $\sU$ is any other category, we have the monadic tensor 
product $\PS\otimes_C \PH\colon \sU\rtarr \sW$ that is defined on objects
$U\in \sU$ as the coequalizer displayed in the diagram
\begin{equation}\label{coequal1} \xymatrix@1{ (\PS C \PH)(U)
\ar[r]<.5ex>^-{\rh\PH} \ar[r]<-.5ex>_-{\PS\la}
& (\PS\PH)(U)\ar[r] & (\PS\otimes_C \PH)(U).\\} 
\end{equation}
We are interested primarily in the case when $\sU = \ast$ and  $\PH = X$ for
a $C$-algebra $(X,\xi)$ in $\sV$, and we then write $\PS\otimes_C X$. 
Specializing to the case $\PS = \XI$, this is the coequalizer in $\sW$
displayed in the diagram
\begin{equation}\label{coequal}
 \xymatrix@1{\XI C X\ar[r]<.5ex>^-{\rh} \ar[r]<-.5ex>_-{\XI\xi}
& \XI X\ar[r] & \XI\otimes_CX.\\} 
\end{equation}
For comparison, we have the canonical split coequalizer
\begin{equation}\label{coequal2}
 \xymatrix@1{CC X\ar[r]<.5ex>^-{\mu} \ar[r]<-.5ex>_-{C\xi}
& C X\ar[r]^{\xi} & X\\} 
\end{equation}
in $\sV$, which is split by $\et_{CX}\colon CX\rtarr CCX$ and 
$\et_X\colon X\rtarr CX$. 

Let $C[\sV]$ denote the category of $C$-algebras in $\sV$. Then 
$\XI\otimes_C(-)$ is a functor $C[\sV]\rtarr \sW$, and our original
adjunction restricts to an adjunction
\begin{equation}\label{cutead}
 \sW(\XI\otimes_CX, Y) \iso C[\sV](X,\LA Y).
\end{equation}
Indeed, consider a map $f\colon X\rtarr \LA Y$ of $C$-algebras, so that
$f\com \xi = \la\com Cf$.  Taking its adjoint $\tilde f\colon \XI X\rtarr Y$,
we see by a little diagram chase that it equalizes the pair of maps
$\XI \xi$ and $\rh$ and therefore induces the required adjoint map 
$\XI\otimes_CX\rtarr Y$.  

This construction applies in particular to the monad $\LA\XI$, and for
the moment we take $C = \LA\XI$.  The Beck monadicity theorem says that
the adjunction (\ref{cutead}) is then an adjoint equivalence under appropriate
hypotheses, which we now explain. 
 
Consider those parallel pairs of arrows 
$(f,g)$ in the diagram 
\[ \xymatrix@1{X\ar[r]<.5ex>^-{f} \ar[r]<-.5ex>_-{g}
& Y\ar@{-->}[r]^-{h} & Z\\} \]
in $\sW$ such that there exists a split coequalizer diagram
\[ \xymatrix@1{\LA X\ar[r]<.5ex>^-{\LA f} \ar[r]<-.5ex>_-{\LA g}
& \LA Y\ar[r]^-{j} & V\\} \]
in $\sV$.  Assume that $\LA$ preserves and reflects coequalizers of such
pairs $(f,g)$ of parallel arrows. Preservation means that if $h$ is a
coequalizer of $f$ and $g$, then $\LA h\colon \LA Y\rtarr \LA Z$ is 
a coequalizer of $\LA f$ and $\LA g$. It follows that there is a unique
isomorphism $i\colon V\rtarr \LA Z$ such that $i\com j = \LA h$. Reflection means 
that there is a coequalizer $h$ of $f$ and $g$ and an isomorphism 
$i\colon V\rtarr \LA Z$ such that $i\com j = \LA h$.

By preservation, if $(X,\xi)$ is a $C$-algebra, then 
$\et\colon X\rtarr \LA(\XI\otimes_{C} X)$ must be an isomorphism because application 
of $\LA$ to the arrows $\rh = \epz \XI$ and $\XI\xi$ of 
(\ref{coequal}) gives the pair of maps that have the split coequalizer 
(\ref{coequal2}). By reflection, 
if $Y$ is an object of $\sW$, then $\epz\colon \XI\otimes_{C} \LA Y\rtarr Y$ 
must be an isomorphism since if we apply $\LA$ to the coequalizer 
(\ref{coequal}) with $X=\LA Y$ we obtain the split coequalizer (\ref{coequal2}) 
for the $C$-algebra $\LA Y$.  
This proves that the adjunction (\ref{cutead}) is an adjoint equivalence.

Returning to our original map of monads $\al\colon C\rtarr \LA\XI$, but 
thinking homotopically, one might hope that $\LA(\XI\otimes_C X)$ is 
equivalent to $X$ under reasonable hypotheses.  However, since coequalizers usually behave badly homotopically, we need a homotopical approximation.  Here thinking model categorically seems more of a hindrance than a help, and we instead use the two-sided monadic bar construction of \cite{Geo}.  It can be defined in the
generality of (\ref{coequal1}) as a functor
\[ B(\PS,C,\THA)\colon \sU\rtarr \sW,\]
but we restrict attention to the case of a $C$-algebra $X$, where it is
\[ B(\PS,C,X). \]

We have a simplicial object $B_*(\PS,C,X)$ in $\sW$ whose object of 
$q$-simplices is 
\[ B_q(\PS,C,X) = \PS C^q X.\]
The right action $\PS C\rtarr \PS$ induces the $0^{th}$ face map, the
map $C^{i-1}\mu$, $1\leq i <q$, induces the $i^{th}$ face map, and the action
$CX\rtarr X$ induces the $q^{th}$ face map.  The maps $C^i\et$, $0\leq i\leq q$
induce the degeneracy maps.  We need to realize simplicial objects $Y_*$ in
$\sV$ and $\sW$ to objects of $\sV$ and $\sW$, and we need to do so compatibly. For that, we need a covariant standard simplex functor 
$\DE^*\colon \DE\rtarr \sV$, where $\DE$ is the standard category of finite 
sets $n$ and monotonic maps.\footnote{Equivalently, $n$ is the ordered set
$0<1<\cdots <n$, and maps preserve order.}  We compose with $\XI$ to obtain a standard
simplex functor $\DE^*\colon \DE\rtarr \sW$.  We define
\[ |X_*|_{\sV} = X_*\otimes_{\DE} \DE^* \] 
for simplicial objects $X_*$ in $\sV$, and similarly for $\sW$.  In the 
cases of interest, realization is a left adjoint. We define
\[ B(\PS,C,X) = |B_*(\PS,C,X)|. \]

Commuting the left adjoint $C$ past realization, we find that 
\begin{equation}
|CX_*|_{\sV}\iso C|X_*|_{\sV}
\end{equation}
and conclude that the realization of a simplicial 
$C$-algebra is a $C$-algebra.  The iterated action map 
$\xi^{q+1}\colon C^{q+1}X\rtarr X$ gives a map $\epz_*$ from $B_*(C,C,X)$ to 
the constant simplicial object at $X$.  Passing to realizations, 
$\epz_*$ gives a natural map of $C$-algebras $\epz\colon B(C,C,X)\rtarr X$.
Forgetting the $C$-action, $\epz_*$ is a simplicial homotopy equivalence in the category of simplicial objects 
in $\sW$, in the combinatorial sense that is defined for simplicial objects in any category.  In reasonable situations, for example categories tensored over spaces, passage to realization converts this to a homotopy equivalence in 
$\sW$.  Commuting coequalizers past realization, we find
\begin{equation}
B(\PS,C,X)\iso \PS\otimes_{C} B(C,C,X).
\end{equation}

This has the same flavor as applying cofibrant approximation to $X$ and then
applying $\PS\otimes_C (-)$, but the two-sided bar construction has considerable
advantages. For example, in our topological situations, it is a continuous 
functor, whereas cofibrant approximations generally are not.   Also, starting from a general object $X$, one could not expect something as strong as an underlying homotopy equivalence from a cofibrant approximation $X'\rtarr X$.  More fundamentally, the functoriality in all three variables is central to 
the arguments.  Presumably, for good model categories $\sV$, if one restricts 
$X$ to be cofibrant in $\sV$, then $B(C,C,X)$ is cofibrant in $\sV[\sT]$, at least
up to homotopy equivalence, and thus can be viewed as an especially nice substitute
for cofibrant approximation, but I've never gone back to work out such model categorical 
details.\footnote{This exercise has recently been carried out in the five author paper \cite{five}.
It is to be emphasized that nothing in the outline we are giving simplifies in the slightest.  Rather, 
the details sketched here are reinterpreted model theoretically.  As anticipated, the essential point is
to observe that $B(C,C,X)$ is of the homotopy type of a cofibrant object when $X$ is cofibrant.}

Now the black box works as follows.  We take $\sV$ and $\sW$ to be categories
with a reasonable homotopy theory, such as model categories.  Our candidate 
for $\bE X$ is $B(\XI,C,X)$.  There are three steps that are needed to obtain an
equivalence between a suitable $C$-algebra $X$ and $\LA \bE X$.   The fundamental one is to prove an approximation theorem to the effect that $C$ is a homotopical approximation to $\LA\XI$. This step has nothing to do with monads, 
depending only on the homotopical properties of the comparison map $\al$.

\begin{step}\label{step1}
For suitable objects $X\in \sV$, $\al\colon CX\rtarr \LA \XI X$ is a weak
equivalence.
\end{step}

The second step is a general homotopical property of realization 
that also has nothing to do with the monadic framework. It implies 
that the good homotopical behavior of $\al$ is preserved by various
maps between bar constructions.

\begin{step}\label{step2} For suitable simplicial objects $Y_*$ and $Y'_*$ in
$\sW$, if $f_*\colon Y_*\rtarr Y'_*$ is a map such that each $f_q$ 
is a weak equivalence, then $|f_*|$ is a weak equivalence in $\sW$,
and similarly for $\sV$.
\end{step}

In the space level cases of interest, the weaker property of being a group completion will generalize being a weak equivalence in Steps \ref{step1} and 
\ref{step2}, but we defer discussion of that until the next section. 

The third step is an analogue of commuting $\LA$ past coequalizers
in the categorical argument we are mimicking. It has two parts, one
homotopical and the other monadic.

\begin{step}\label{step3}
For suitable simplicial objects $Y_*$ in $\sW$, the canonical natural map
\[ \ze\colon |\LA Y_*|_{\sV}\rtarr \LA | Y_*|_{\sW} \] 
is both a weak equivalence and a map of $C$-algebras.
\end{step}

Here $\LA Y_*$ is obtained by applying $\LA$ levelwise to simplicial
objects. To construct the canonical map $\ze$, we first obtain
a natural isomorphism
\begin{equation}
\XI |X_*|_{\sV}\iso |\XI X_*|_{\sW}
\end{equation}
by commuting left adjoints, where $X_*$ is a simplicial object in $\sV$.  
Applying this with  $X_*$ replaced by $\LA Y_*$ we obtain
\[ |\epz|_{\sW}\colon \XI |\LA Y_*|_{\sV}\iso |\XI \LA Y_*|_{\sW}
\rtarr |Y_*|_{\sW}. \]
Its adjoint is the required natural map 
$\ze\colon |\LA Y_*|_{\sV}\rtarr \LA|Y_*|_{\sW}$.

Assuming that these steps have been taken, the black box works as follows 
to relate the homotopy categories of $C[\sV]$ and $\sW$.
For a $C$-algebra $X$ in $\sV$, we have the diagram of maps of $C$-algebras
\[  \xymatrix@1{ 
X & B(C,C,X) \ar[l]_-{\epz} \ar[rr]^{B(\al,\id,\id)} 
& & B(\LA\XI,C,X) \ar[r]^-{\ze} & \LA B(\XI,C,X) = \LA \bE X\\} \]
in which all maps are weak equivalences (or group completions).  
On the left, we have a map, but not a $C$-map, $\et\colon X\rtarr B(C,C,X)$ 
which is an inverse homotopy equivalence to $\epz$.  We 
also write $\et$ for the resulting composite $X\rtarr \LA \bE X$.
This is the analogue of the map $\et\colon X\rtarr \LA(\XI\otimes_{C} X)$ 
in our categorical sketch.  

For an object $Y$ in $\sW$, observe that 
\[ B_q(\XI,\LA\XI,\LA Y) = (\XI\LA)^{q+1}Y \]
and the maps $(\XI\LA)^{q+1}\rtarr Y$ obtained by iterating $\epz$
give a map of simplicial objects from $B_*(\XI,\LA\XI,\LA Y)$ to 
the constant simplicial object at $Y$.  On passage to realization,
we obtain a composite natural map
\[ \xymatrix@1{
\bE\LA Y = B(\XI,C,\LA Y) \ar[rr]^-{B(\id,\al,\id)}
& & B(\XI,\LA\XI,\LA Y)\ar[r]^-{\epz} & Y,\\}\]
which we also write as $\epz$. 
This is the analogue of $\epz\colon \XI\otimes_{C} \LA Y\rtarr Y$ in our
categorical sketch.  We have the following commutative diagram in which 
all maps except the top two are weak equivalences, hence they are too. 
\[  \xymatrix{ 
\LA \bE \LA Y = \LA B(\XI,C,\LA Y) \ar[rr]^-{\LA B(\id,\al,\id)}
& & \LA B(\XI,\LA\XI,\LA Y) \ar[r]^-{\LA \epz} & \LA Y\\
B(\LA \XI,C,\LA Y) \ar[rr]_-{B(\id,\al,\id)} \ar[u]^{\ze}
& & B(\LA \XI,\LA\XI,\LA Y) \ar[u]_{\ze} \ar[ur]_{\epz} & \\} \]
We do not expect $\LA$ to reflect weak equivalences in general, so we
do not expect $\bE\LA Y\htp Y$ in general, but we do expect this on suitably
restricted $Y$.  We conclude that the adjunction (\ref{cutead}) induces 
an adjoint equivalence of suitably restricted homotopy categories.

\section{The additive and multiplicative infinite loop space 
machine}\label{machines}

The original use of this approach in \cite{Geo} took $(\XI,\LA)$ to be 
$(\SI^n,\OM^n)$ and $C$ to be the monad associated to the little 
$n$-cubes operad $\sC_n$.  It took $\al$ to be the map of monads 
given by the composites $\tha\com C_n\et\colon C_nX\rtarr \OM^n\SI^n X$.
For connected $\sC_n$-spaces $X$, details of all steps may be found in 
\cite{Geo}. 

For the non-connected case, we say that an $H$-monoid\footnote{This is
just a convenient abbreviated way of writing homotopy associative $H$-space.} is
grouplike if $\pi_0(X)$ is a group under the induced multiplication. We
say that a map $f\colon X\rtarr Y$ between homotopy commutative $H$-monoids
is a group completion if $Y$ is grouplike and two things hold.  First, 
$\pi_0(Y)$ is the group completion of $\pi_0(X)$ in the sense that any map 
of monoids from $\pi_0(X)$ to a group $G$ factors uniquely through a group 
homomorphism $\pi_0(Y)\rtarr G$. Second, for any commutative ring of coefficients, 
or equivalently any field of coefficients, the homomorphism 
$f_*\colon H_*(X)\rtarr H_*(Y)$ of graded commutative rings is a localization 
(in the classical algebraic sense) at the submonoid $\pi_0(X)$ of $H_0(X)$.  
That is, $H_*(Y)$ is $H_*(X)[\pi_0(X)^{-1}]$.  

For $n\geq 2$,  $\al_n$ is a group 
completion for all $\sC_n$-spaces $X$ by calculations of Fred Cohen
\cite{CLM} or by an argument due to Segal \cite{Seg1}.  Thus $\al_n$ is a 
weak equivalence for all grouplike $X$.  This gives Step \ref{step1}, and 
Steps \ref{step2} and \ref{step3} are dealt with in \cite{MayPer}.  

As Gaunce Lewis pointed out to me many years ago and Matt Ando, et al, rediscovered
\cite{five}, in the stable case $n=\infty$ we can compare spaces and spectra directly. 
We take $\sV = \sT$ and $\sW = \sS$, we take $(\XI,\LA)$ to be 
$(\SI^{\infty},\OM^{\infty})$, and we take $C$ to be the monad associated
to the Steiner operad (for $U = \bR^{\infty}$); in the additive theory, we 
could equally well use 
the infinite little cubes operad $\sC_{\infty}$.  As recalled in
(\ref{alpha}), we have a map of monads $\al\colon C\rtarr Q$.  The 
calculations needed to prove that $\al$ is a group completion preceded 
those in the deeper case of finite $n$ \cite{CLM}, and we have Step \ref{step1}.
Step \ref{step2} is given for spaces in \cite[Ch. 11]{Geo} and 
\cite[App.]{MayPer} and for spectra in \cite[Ch X]{EKMM}; see 
\cite[X.1.3 and X.2.4]{EKMM}. 

We need to say a little about Step \ref{step3}.  For simplicial spaces
$X_*$ such that each $X_q$ is connected (which holds when we apply
this step) and which satisfy the usual (Reedy) cofibrancy condition
(which follows in our examples from the assumed nondegeneracy of 
basepoints), the map
$\ze\colon |\OM X_*|\rtarr \OM |X_*|$ is a weak equivalence by 
\cite[12.3]{Geo}.  That is the hardest thing in \cite{Geo}.  Moreover, the 
$n$-fold iterate $\ze^n\colon |\OM^n X_*|\rtarr \OM^n |X_*|$ is a map of 
$\sC_n$-spaces by \cite[12.4]{Geo}. The latter argument works equally well 
with $\sC_n$ replaced by the Steiner operad $\sK_{\bR^n}$, so we have
Step \ref{step3} for each $\OM^n$.  For sufficiently nice simplicial spectra 
$E_*$, such as those relevant here, the canonical map 
\begin{equation}\label{zeta1}
\ze\colon |\OM^{\infty} E_*|_{\sT} \rtarr \OM^{\infty} |E_*|_{\sS}
\end{equation}
can be identified with the colimit of the iterated canonical maps
\begin{equation}\label{zeta2}
 \ze^n\colon |\OM^n {(E_n)}_*|_{\sT} \rtarr \OM^{n} |(E_n)_*|_{\sT},
\end{equation}
and Step \ref{step3} for $\OM^{\infty}$ 
follows by passage to colimits from Step \ref{step3} for the $\OM^n$, applied
to simplicial $(n-1)$-connected spaces. 
Here, for simplicity of comparison with \cite{Geo}, we have indexed spectra 
sequentially, that is on the cofinal sequence $\bR^n$ in $U$.  The $n^{th}$ 
spaces $(E_n)_*$ of the simplicial spectrum $E_*$ give a simplicial space and 
the $\OM^n(E_n)_*$ are compatibly isomorphic to $(E_0)_*$. Thus, on the left side, 
the colimit is $|\OM^{\infty} E_*|_{\sT}$.
When $E_* = LT_*$ for a simplicial inclusion prespectrum $T_*$,
the right side can be computed as $L|T_*|_{\sP}$, where the prespectrum 
level realization is defined levelwise.  One checks that $|T_*|_{\sP}$ is 
again an inclusion prespectrum, and the identification of the colimit 
on the right with $\OM^{\infty}|E_*|_{\sS}$ follows. 
 
Granting these details, we have the following additive infinite loop space
machine.  Recall that a spectrum is connective if its negative homotopy groups are zero and that a map $f$ of connective spectra is a weak equivalence if and only if $\OM^{\infty}f$ is a weak equivalence.

\begin{thm}\label{add}  For a $\sC$-space $X$, define
$\bE X = B(\SI^{\infty},C,X)$. Then $\bE X$ is connective and there is a natural diagram of maps of $C$-spaces
\[  \xymatrix@1{ 
X & B(C,C,X) \ar[l]_-{\epz} \ar[rr]^{B(\al,\id,\id)} 
& & B(Q,C,X) \ar[r]^-{\ze} & \OM^{\infty} \bE X\\} \]
in which $\epz$ is a homotopy equivalence with natural homotopy inverse $\et$,
$\ze$ is a weak equivalence, and $B(\al,\id,\id)$ is a group completion. 
Therefore the composite $\et\colon X\rtarr \OM^{\infty} \bE X$ is a group
completion and thus a weak equivalence if $X$ is grouplike.  For a 
spectrum $Y$, there is a composite natural map of spectra
\[ \xymatrix@1{
\epz\colon \bE\OM^{\infty} Y \ar[rr]^-{B(\id,\al,\id)}
& & B(\SI^{\infty},Q,\OM^{\infty} Y)\ar[r]^-{\epz} & Y,\\}\]
and the induced maps of $\sC$-spaces
\[  \xymatrix@1{ 
\OM^{\infty}\epz\colon 
\OM^{\infty} \bE \OM^{\infty} Y  \ar[rr]^-{\OM^{\infty} B(\id,\al,\id)}
& & \OM^{\infty} B(\SI^{\infty},Q,\OM^{\infty} Y) \ar[r]^-{\OM^{\infty}\epz} & \OM^{\infty} Y\\}\]
are weak equivalences. Therefore $\bE$ and $\OM^{\infty}$
induce an adjoint equivalence between the homotopy category of grouplike
$E_{\infty}$ spaces and the homotopy category of connective spectra.
\end{thm}

The previous theorem refers only to the Steiner operad $\sC$, for canonicity,
but we can apply it equally well to any other $E_{\infty}$ operad $\sO$. We
can form the product operad $\sP = \sC\times \sO$, and the $j^{th}$ levels of
its projections $\pi_1\colon \sP\rtarr \sC$ and $\pi_2\colon \sP\rtarr \sO$
are $\SI_j$-equivariant homotopy equivalences. While the monad $P$ associated 
to $\sP$ is not a product, the induced projections of monads $P\rtarr C$ and 
$P\rtarr O$ are natural weak equivalences.  This allows us to replace $C$ by 
$P$ in the previous theorem. If $X$ is an $O$-space, then it is a $P$-space by 
pullback along $\pi_2$, and $\SI^{\infty}$ is a right $P$-functor by pullback along $\pi_1$.  

There is another way to think about this trick that I now find preferable.  
Instead of repeating Theorem \ref{add} with $C$ replaced by $P$, one can first 
change input data and then apply Theorem \ref{add} as it stands.  Here we again 
use the two-sided bar construction.  For $\sO$-spaces $X$ regarded by pullback
as $\sP$-spaces, we have a pair of natural weak equivalences of $\sP$-spaces
\begin{equation}
\xymatrix@1{ 
X & B(P,P,X) \ar[l]_-{\epz} \ar[rr]^-{B(\pi_1,\id,\id)} & &  B(C,P,X),\\}
\end{equation}
where $B(C,P,X)$ is a $C$-space regarded as a $\sP$-space by pullback 
along $\pi_1$.  The same maps show that if $X$ is a $\sP$-space, then 
it is weakly equivalent as a $\sP$-space to $B(C,P,X)$.  Thus the 
categories of $\sO$-spaces and $\sP$-spaces can be used interchangeably.
Reversing the roles of $\sC$ and $\sO$, the categories of $\sC$-spaces
and $\sP$-spaces can also be used interchangeably.  We conclude that $\sO$-spaces
for any $E_{\infty}$ operad $\sO$ can be used as input to the additive
infinite loop space machine. We have the following conclusion.

\begin{cor}\label{homocat} For any $E_{\infty}$ operad $\sO$, the additive infinite loop 
space machine $\bE$ and the $0^{th}$ space functor $\OM^{\infty}$ induce an adjoint equivalence 
between the homotopy category of grouplike $\sO$-spaces and the homotopy category of 
connective spectra.
\end{cor}

In particular, many of the interesting examples are $\sL$-spaces.
We can apply the additive infinite loop space machine to them, 
ignoring the special role of $\sL$ in the multiplicative theory. 
As we recall in the second sequel \cite{Sequel2}, examples include 
various stable classifying spaces and homogeneous spaces of geometric 
interest.  By Corollary \ref{GLR}, they also include the unit spaces $GL_1R$ and $SL_1R$
of an $E_{\infty}$ ring spectrum $R$.  The importance of these spaces in geometric 
topology is explained in \cite{Sequel2}.  The following definitions and results 
highlight their importance in stable homotopy theory and play a significant role 
in \cite{five}.\footnote{That paper reads in part like a sequel to this one.  
However, aside from a very brief remark that merely acknowledges their existence,
$E_{\infty}$ ring spaces are deliberately avoided there.}  We start with a 
reinterpretation of the adjunction of Corollary \ref{themguys} 
for grouplike $\sL$-spaces $Y$.

\begin{lem}\label{gplikeR} If $Y$ is a grouplike $\sL$-space, then
\[ \sL[\sS](\SI^{\infty}_+Y,R)\iso \sL[\sU](Y,GL_1R) \]
\end{lem}
\begin{proof}
A map of $\sL$-spaces $Y\rtarr \OM^{\infty}R$ must take values in $GL_1R$ since
the group $\pi_0Y$ must map to the group of units of the ring $\pi_0 \OM^{\infty} R$.
\end{proof}

The notations of the following definition have recently become standard, although
the definition itself dates back to \cite{MQR}. 

\begin{defn}\label{gl} Let $R$ be an $\sL$-spectrum.  Using the operad $\sL$ 
in the additive infinite loop space machine, define $gl_1R$ and $sl_1R$ to be
the spectra obtained from the $\sL$-spaces $GL_1R$ and $SL_1R$, so that
$\OM^{\infty}gl_1R \simeq GL_1R$ and $\OM^{\infty}sl_1R \simeq SL_1R$.
\end{defn}

\begin{cor} On homotopy categories, the functor $gl_1$ from 
$E_{\infty}$ ring spectra to spectra is right adjoint to the functor 
$\SI^{\infty}_+\OM^{\infty}$ from spectra to $E_{\infty}$ ring spectra.
\end{cor}
\begin{proof}  Here we implicitly replace the $C$-space $\OM^{\infty}R$ by a weakly
equivalent $\sL$-space, as above. Using this replacement, we can view the functor 
$\SI^{\infty}_{+}\OM^{\infty}$ as taking values in $\sL$-spectra.   
Now the conclusion is obtained by composing the equivalence of Corollary 
\ref{homocat} with the adjunction of
Lemma \ref{gplikeR} and using that, in the homotopy category, maps from a spectrum $E$
to a connective spectrum $F$, such as $gl_1R$, are the same as maps from the 
connective cover $cE$ of $E$ into $F$, while $\OM^{\infty} cE\htp  \OM^{\infty} E$. 
\end{proof}

Note that we replaced $\sL$ by $\sC$ to define $GL_1R$ as a $\sC$-space valued functor
before applying $\bE$,
and we replaced $\sC$ by $\sL$ to define $\OM^{\infty}$ as an $\sL$-space valued functor.
Another important example of an $E_{\infty}$ operad should also be mentioned.

\begin{rem}\label{BEop} There is a categorical operad, denoted $\sD$, 
that is obtained by applying the classifying space functor to the translation 
categories of the groups $\SI_j$. This operad acts on the classifying spaces 
of permutative categories, as we recall from \cite[\S4]{MayPer} in 
\cite{Sequel}.  Another construction of the same $E_{\infty}$ operad is 
obtained by applying a certain product-preserving functor from spaces to 
contractible spaces to the operad $\sM$ that defines monoids; 
see \cite[p. 161]{Geo} and \cite[4.8]{MayPer}.  The second 
construction shows that $\sM$ is a suboperad, so that a $\sD$-space 
has a canonical product that makes it a topological monoid.  
The simplicial version of $\sD$ is called the Barratt-Eccles operad 
in view of their use of it in \cite{BE1, BE2, BE3}.
\end{rem}

\begin{rem}\label{OME}  In the applications, one often uses the 
consistency statement that, for an $E_{\infty}$ space $X$, there is a natural 
map of spectra $\SI \bE\OM X\rtarr \bE X$ that is an equivalence if $X$ is connected.  
This is proven in \cite[\S14]{Geo}, with improvements in \cite[\S3]{MayPer} and
\cite[VI.3.4]{MQR}.  The result is considerably less obvious than it seems, and the cited 
proofs are rather impenetrable, even to me.   I found a considerably simpler conceptual 
proof while writing this paper.  Since this is irrelevant to our multiplicative 
story, I'll avoid interrupting the flow here by deferring the new proof to Appendix B (\S15).
\end{rem}

Now we add in the multiplicative structure, and we find that it is
startlingly easy to do so.
Let us say that a rig space $X$ is ringlike if it is grouplike under its
additive $H$-monoid structure.  A map of rig spaces $f\colon X\rtarr Y$ is 
a ring completion if $Y$ is ringlike and $f$ is a group completion of the
additive structure.  Replacing $\sT$ and $\sS$ by $L_+[\sT]$ and 
$L_+[\sS]$ and using that $\al\colon C\rtarr Q$ is a map of monads
on $L_+[\sT]$, the formal structure of the previous section still
applies verbatim, and the homotopy properties depend only on the additive structure.  The only point that needs mentioning is that, for the monadic part of Step \ref{step3}, we now identify the map $\zeta$ of (\ref{zeta1}) with the colimit over $V$ of the maps 
\begin{equation}\label{zeta3}
 \ze^V\colon |\OM^V {(EV)}_*|_{\sT} \rtarr \OM^{V} |(EV)_*|_{\sT}
\end{equation}
and see that $\ze$ is a map of $\sL_+$-spaces because of the naturality
of the colimit system (\ref{zeta3}) with respect to linear isometries.
Therefore the additive infinite loop space machine specializes 
to a multiplicative infinite loop space machine.   

\begin{thm}\label{mult}  For a $(\sC,\sL)$-space $X$, define
$\bE X = B(\SI^{\infty},C,X)$. Then $\bE X$ is a connective $\sL$-spectrum
and all maps in the diagram 
\[  \xymatrix@1{ 
X & B(C,C,X) \ar[l]_-{\epz} \ar[rr]^{B(\al,\id,\id)} 
& & B(Q,C,X) \ar[r]^-{\ze} & \OM^{\infty} \bE X\\} \]
of the additive infinite loop space machine are maps of
$(\sC,\sL)$-spaces. Therefore the composite 
$\et\colon X\rtarr \OM^{\infty} \bE X$ is a 
ring completion.  For an $\sL$-spectrum $R$, the maps
\[ \xymatrix@1{
\epz\colon \bE\OM^{\infty} R \ar[rr]^-{B(\id,\al,\id)}
& & B(\SI^{\infty},Q,\OM^{\infty} R)\ar[r]^-{\epz} & R\\}\]
are maps of $\sL$-spectra and the maps
\[  \xymatrix@1{ 
\OM^{\infty}\epz\colon \OM^{\infty} \bE \OM^{\infty} R  \ar[rr]^-{\OM^{\infty} B(\id,\al,\id)}
& & \OM^{\infty} B(\SI^{\infty},Q,\OM^{\infty} R) \ar[r]^-{\OM^{\infty}\epz} & \OM^{\infty} R\\}\]
are maps of $(\sC,\sL)$-spaces.  Therefore $\bE$ and $\OM^{\infty}$
induce an adjoint equivalence between the homotopy category of ringlike
$E_{\infty}$ ring spaces and the homotopy category of connective 
$E_{\infty}$ ring spectra.
\end{thm}

Again, we emphasize how simply and naturally these structures fit together.
However, here we face an embarrassment.  We would like to apply this machine
to construct new $E_{\infty}$ ring spectra, and the problem is that the
only operad pairs we have in sight are $(\sC,\sL)$ and $(\sN,\sN)$. We 
could apply the product of operads trick to operad pairs if we only had 
examples to which to apply it.  We return to this point in the sequel
\cite{Sequel}, where we show how to convert such naturally occurring data
as bipermutative categories to $E_{\infty}$ ring spaces, but the theory in 
this paper is independent of that problem.

\section{Localizations of the special unit 
spectrum $sl_1R$}\label{unitspec}

The Barratt-Priddy-Quillen theorem tells us how to construct the sphere
spectrum from symmetric groups.  This result is built into the additive 
infinite loop space machine.  A multiplicative elaboration is also built 
into the infinite loop space machine, as we explain here.  
For a $\sC$-space or $(\sC,\sL)$-space $X$, we abbreviate notation by writing 
$\GA X = \OM^{\infty}\bE X$, using notations like $\GA_1 X$ to indicate
components.  We write $\et\colon X\rtarr \GA X$ for the group completion
map of Theorem \ref{add}.\footnote{The letter $\GA$ is chosen as a reminder of
the group completion property.}  Since it is the composite of $E_{\infty}$ maps or,
multiplicatively, $E_{\infty}$ ring maps and the homotopy inverse of such 
a map, we may think of it
as an $E_{\infty}$ or $E_{\infty}$ ring map. 

\begin{thm}  For a based space $Y$, $\al$ and the left map $\et$ are 
group completions and $\GA \al$ and the right map $\et$ are 
equivalences in the commutative diagram
\[\xymatrix{
CY \ar[r]^{\al} \ar[d]_{\et}  &  QY \ar[d]^{\et} \\
\GA CY \ar[r]_-{\GA \al} & \GA QY. \\}\]
If $Y$ is an $\sL_+$-space, then this is a diagram of $E_{\infty}$ ring spaces.
\end{thm}

Replacing $Y$ by $Y_+$, we see by inspection that $C(Y_+)$ is the disjoint 
union over $j\geq 0$ of the spaces $\sC(j)\times_{\SI_j} Y^j$, and of course 
$\sC(j)$ is a model for $E\SI_j$.  When $Y = BG$ for a topological group $G$, 
these  are classifying spaces $B(\SI_j\int G)$.  When $Y = \ast$ and thus $Y_+ = S^0$, 
they are classifying spaces  $B\SI_j$, and we see that the $0^{th}$ space of the sphere spectrum is the
group completion of the $H$-monoid $\amalg_{j\geq 0} B\SI_j$.  This is one version 
of the Barratt-Quillen theorem.  

For an $E_{\infty}$ space $X$ with a map $S^0\rtarr X$, there is a natural map of 
monoids from the additive monoid $\bZ_{\geq 0}$ of nonnegative integers to the monoid 
$\pi_0(X)$. It is obtained by passage to $\pi_0$ from the composite $CS^0\rtarr CX\rtarr X$.  
We assume that it is a monomorphism, as holds in the interesting cases. 
Write $X_m$ for the $m^{th}$ component.  Translation by an element in $X_n$ 
(using the $H$-space structure induced by the operad action) induces a map 
$n\colon X_m\rtarr X_{m+n}$.  We have the homotopy commutative ladder
\[ \xymatrix{
X_0\ar[r]^-{1}\ar[d]_{\et}  & X_1 \ar[r] \ar[d]^{\et}  & \cdots  \ar[r] 
& X_{n-1} \ar[d]^{\et} \ar[r]^-{1} & X_n \ar[d]^{\et} \ar[r] & \cdots \\
\GA_0X\ar[r]_-{1} & \GA_1X \ar[r]  & \cdots  \ar[r] & \GA_{n-1}X 
\ar[r]_-{1} & \GA_{n}X \ar[r] & \cdots \\} \]
Write $\bar{X}$ for the telescope of the top row.  The maps on the 
bottom row are homotopy equivalences, so the ladder induces a map 
$\bar{\et}\colon \bar{X}\rtarr \GA_0 X$.  Since $\et$ is a group completion, 
$\bar{\et}$ induces an isomorphism on homology.  Taking $X=CS^0$, it follows
that $\bar{\et}\colon B\SI_{\infty} \rtarr Q_0S^0$ is a homology isomorphism
and therefore that $Q_0S^0$ is the plus construction on $B\SI_{\infty}$. 
This is another version of the Barratt-Quillen theorem. 

We describe a multiplicative analogue of this argument and result, due to Tornehave
in the case of $QS^0$ and generalized in \cite[VII.5.3]{MQR}, where full
details may be found.  Recall that we write $gl_1R$ and $sl_1R$ for the spectra 
$\bE GL_1R$ and $\bE SL_1R$ that the black box associates to the $E_{\infty}$ spaces 
$SL_1R\subset GL_1R$, where $R$ is an $E_{\infty}$ ring spectrum.  
The map $sl_1R\rtarr gl_1R$ is a connected cover.  It is usually not an easy matter to identify 
$sl_1R$ explicitly. The cited result gives a general step in this direction. The point to 
emphasize is that the result intrinsically concerns
the relationship between the additive and multiplicative $E_{\infty}$ space
structures on an $E_{\infty}$ ring space.  Even if one's focus is solely on
understanding the spectrum $sl_1R$ associated to the $E_{\infty}$ ring spectra $R$, 
one cannot see a result like this without introducing $E_{\infty}$ ring spaces. 

As in \cite[VII\S5]{MQR}, we start with an $E_{\infty}$ ring 
space $X$ and we assume that the canonical map of rigs from the rig
$\bZ_{\geq 0}$ of nonnegative integers to $\pi_0(X)$ is a monomorphism.  
When $X = \OM^{\infty} R$, 
$\bE X$ is equivalent as an $E_{\infty}$ ring spectrum to the connective 
cover of $R$ and $GL_1\bE(X)$ is equivalent as an $E_{\infty}$ space to $GL_1R$. The general case is especially interesting when $X$ is the classifying space 
$B\sA$ of a bipermutative category $\sA$ (as defined in \cite[VI\S3]{MQR}; 
see the sequel \cite{Sequel}). 

Let $M$ be a multiplicative submonoid of 
$\bZ_{\geq 0}$ that does not contain zero.  For example, $M$ might be $\{p^i\}$ 
for a prime $p$, or it might be the set of positive integers prime to $p$.
Let $\bZ_M = \bZ[M^{-1}]$ denote the localization of $\bZ$ at $M$; thus
$\bZ_M =\bZ[p^{-1}]$ in our first example and $\bZ_M = \bZ_{(p)}$ in the
second.  Let $X_M$ denote the disjoint union of the components $X_m$ with 
$m\in M$. Often, especially when $X=B\sA$, we have a good understanding 
of $X_M$.

Clearly $X_M$ is a sub $\sL$-space of $X$.  Converting it to a $\sC$-space 
and applying the additive infinite loop space machine or, equivalently, 
applying the additive infinite loop space machine constructed starting with
$\sC\times\sL$, we obtain a connective spectrum $\bE(X_M) = \bE(X_M,\xi)$.  
The alternative notation highlights that the spectrum comes 
from the multiplicative operad action on $X$.  This gives us the infinite
loop space $\GA_1(X_M,\xi) = \OM^{\infty}_1\bE(X_M,\xi)$, which depends only
on $X_M$. 

We shall relate this to $SL_1\bE(X)=\OM^{\infty}_1\bE(X,\tha)$. The alternative
notation highlights that 
$\bE(X,\tha)$ is constructed from the additive operad action on $X$ and has 
multiplicative structure inherited from the multiplicative operad action.  

A key example to have in mind is $X=B\sG\sL(R)$, where $\sG\sL(R)$ is the general linear 
bipermutative category of a commutative ring $R$; its objects are the $n\geq 0$ and 
its morphisms are the general linear groups $GL(n,R)$.  In that case $\bE X = KR$ is
the algebraic $K$-theory $E_{\infty}$ ring spectrum of $R$.  The 
construction still makes sense when $R$ is a topological ring.  We can take $R=\bR$ or
$R = \bC$, and we can restrict to orthogonal or unitary matrices without changing the
homotopy type.  Then $K\bR = kO$ and $K\bC=kU$ are the real and complex connective
topological $K$-theory spectra with ``special linear'' spaces $BO_{\otimes}$ 
and $BU_{\otimes}$.  

To establish the desired relation between $SL_1\bE(X)$ and $\GA_1(X_M,\xi)$, we 
need a mild homological hypothesis on $X$, namely that $X$ is convergent at
$M$. It always holds when $X$ is ringlike, when $X = CY$ for an
$\sL_+$-space $Y$, and when $X = B\sA$ for the usual bipermutative
categories $\sA$; see \cite[VII.5.2]{MQR}.  We specify it during 
the sketch proof of the following result.  

\begin{thm}\label{unitsp}  
If $X$ is convergent at $M$, then, as an $E_{\infty}$ space,
the localization of $\OM^{\infty}_1\bE(X,\tha)$ at $M$ is 
equivalent to the basepoint component $\OM^{\infty}_1\bE(X_M,\xi)$. 
\end{thm}

Thus, although $\bE(X,\tha)$ is constructed using the
$\sC$-space structure on $X$, the localizations of 
$sl_1\bE(X)$ depend only on $X$ as an $\sL$-space. When
$X=\OM^{\infty}R$, $\bE(X,\tha)$ is equivalent to the 
connective cover of $R$ and $sl_1\bE X$ is equivalent to $sl_1R$. 

\begin{cor} For an $E_{\infty}$ ring spectrum $R$, the localization  
$(sl_1R)_M$ is equivalent to the connected cover of $\bE((\OM^{\infty}R)_M,\xi)$.
\end{cor}

\begin{proof}[Sketch proof of Theorem \ref{unitsp}] 
We repeat the key diagram from \cite[p. 196]{MQR}.   
Remember that $\GA X = \OM^{\infty}\bE X$; $\GA_1X$ and $\GA_MX$ 
denote the component of $1$ and the disjoint union of the 
components $\GA_mX$ for $m\in M$.  We have distinguished applications 
of the infinite loop space machine with respect to actions $\tha$
of $\sC$ and $\xi$ of $\sL$, and we write $\et_{\oplus}$ and $\et_{\otimes}$
for the corresponding group completions.  The letter $i$ always denotes an inclusion.  The following is a 
diagram of $\sL$-spaces and homotopy inverses of equivalences that are maps
of $\sL$-spaces.
\[\xymatrix{
& & & \GA(X_1,\xi) \ar[rr]^-{\GA i}\ar[dd]_<<<<<<<{\GA \et_{\oplus}} 
& & \GA(X_M,\xi) \ar[dd]^{\GA \et_{\oplus}}_{\htp} \\
 & & & & & \\
X_1 \ar[rr]^-{i} \ar[dd]_{\et_{\oplus}} \ar[uurrr]^{\et_{\otimes}}
& & X_M\ar[dd]^>>>>>>{\et_{\oplus}}  \ar[uurrr]_>>>>>>>>>>>>>>>>>{\et_{\otimes}}
& \GA(\GA_1(X,\tha),\xi) \ar[rr]^{\GA i} & & \GA(\GA_M(X,\tha),\xi)   \\
& & & & & \\
\GA_1(X,\tha) \ar[rr]_-{i} \ar[uurrr]^<<<<<<<<<<<<<<{\et_{\otimes}}
& & \GA_M(X,\tha)\ar[uurrr]_{\et_{\otimes}} & & &\\} \]
Note that the spaces on the left face are connected, so that their
images on the right face lie in the respective components of $1$. 
There are two steps.
\begin{enumerate}[(i)]
\item  If $X$ is ringlike, then the composite 
$\GA i\com \et_{\otimes}\colon X_1\rtarr \GA_1(X_M,\xi)$
on the top face is a localization of $X_1$ at $M$.
\item  If $X$ is convergent, then the vertical map $\GA\et_{\oplus}$ labelled
$\htp$ at the top right is a weak equivalence.
\end{enumerate}
Applying (i) to the bottom face, as we may, we see that the (zigzag) composite from
$\GA_1(X,\tha)$ at the bottom left to $\GA_1(X_M,\xi)$ 
at the top right is a localization at $M$.  

To prove (i) and (ii), write
the elements of $M$ in increasing order, $\{1, m_1, m_2, \cdots\}$, 
let $n_i = m_1\cdots m_i$, and consider the following homotopy 
commutative ladder.

{\small
\[ \xymatrix{
X_1\ar[r]^{m_1}\ar[d]_{\et_{\otimes}}  & X_{n_1} \ar[r] \ar[d]^{\et_{\otimes}}  & \cdots  \ar[r] & X_{n_{i-1}}  \ar[r]^-{m_i}\ar[d]^{\et_{\otimes}} 
& X_{n_i}\ar[d]^{\et_{\otimes}} \ar[r] &  \cdots \\
\GA_1(X_M,\xi)\ar[r]_{m_1} & \GA_{n_1}(X_M,\xi) \ar[r]  & \cdots  \ar[r] 
& \GA_{n_{i-1}}(X_M,\xi)
\ar[r]_-{m_i} & \GA_{n_i}(X_M,\xi) \ar[r] & \cdots \\} \]
}

\noindent
Here the translations by $m_i$ mean multiplication (using the $H$-space 
structure induced by the action $\xi$) by an element of $X_{m_i}$. 

Let $\bar{X}_M$ be the telescope of the top row.  Take homology with coefficients in a commutative ring.  The $m_i$ in the bottom row are equivalences since $\GA(X_M,\xi)$ is grouplike, so the diagram gives a map
$\bar{X}\rtarr \GA_1(X_M,\xi)$. The homological definition of a group completion, applied to $\et_{\otimes}$, implies that this map is a homology isomorphism. 
Note that the leftmost arrow $\et_{\otimes}$ factors through 
$\GA(X_1,\xi) = \GA_1(X_1,\xi)$. Exploiting formulas that relate 
the additive and multiplicative Pontryagin products on $H_*(X)$, we can
check that $\tilde{H}_*(\bar{X};\bF_p)=0$ if $p$ divides an element $m\in M$.
The point is that $m$ is the sum of $m$ copies of $1$, and 
there is a distributivity formula for $x\cdot m$ in terms of the 
additive $H$-space structure $\ast$.  This implies that the space 
$\GA_1(X_M,\xi)$ is $M$-local. 

For (i), write $\ast( -n)\colon X_n\rtarr X_0$ for the equivalence given by
using $\ast$ to send $x$ to $x\ast y_{-n}$ for a 
point $y_{-n}\in X_{-n}$.  Formulas in the definition of a 
$(\sC,\sL)$-space imply that the following ladder is homotopy commutative.
\[ \xymatrix{
X_1\ar[r]^{m_1}\ar[d]_{\ast(- 1)}  & X_{n_1} \ar[r] \ar[d]^{\ast(-n_1)}  & \cdots  \ar[r] & X_{n_{i-1}}  \ar[r]^-{m_i}\ar[d]^{\ast( - n_{i-1})} 
& X_{n_i}\ar[d]^{\ast(-n_i)} \ar[r] &  \cdots \\
X_0 \ar[r]_{m_1} & X_0 \ar[r]  & \cdots  \ar[r] 
& X_0 \ar[r]_-{m_i} & X_0 \ar[r] & \cdots \\} \]
A standard construction of localizations of $H$-spaces gives that the 
telescope of the bottom row is a localization $X_0\rtarr (X_0)[M^{-1}]$, 
hence so is the telescope $X_1\rtarr \bar{X}_M$ of the top row, hence
so is its composite with $\bar{X}_M\rtarr \GA_1(X_M,\xi)$. This proves (i).

For (ii), we consider the additive analogue of our first ladder:
\[ \xymatrix{
X_1\ar[r]^{m_1}\ar[d]_{\et_{\oplus}}  & X_{n_1} \ar[r] \ar[d]^{\et_{\oplus}}  & \cdots  \ar[r] & X_{n_{i-1}}  \ar[r]^-{m_i}\ar[d]^{\et_{\oplus}} 
& X_{n_i}\ar[d]^{\et_{\oplus}} \ar[r] &  \cdots \\
\GA_1(X,\tha)\ar[r]_{m_1} & \GA_{n_1}(X,\tha) \ar[r]  & \cdots  \ar[r] 
& \GA_{n_{i-1}}(X,\tha)
\ar[r]_-{m_i} & \GA_{n_i}(X,\tha) \ar[r] & \cdots \\} \]
We say that $X$ is convergent at $M$ if, for each prime $p$ which does not divide any element of $M$, there is an eventually increasing sequence 
$n_i(p)$ such that 
\[ (\et_{\oplus})_*\colon H_j(X_i;\bF_p) \rtarr H_j(\GA_i(X,\tha);\bF_p) \]
is an isomorphism for all $j\leq n_i(p)$.  With this condition, the induced
map of telescopes is a mod $p$ homology isomorphism for such primes $p$. This 
implies the same statement for the map
\[ \GA_1\et_{\oplus}\colon \GA_1(X_M,\xi)\rtarr \GA_1(\GA_M(X,\tha),\xi).\]  
Since this is a map between $M$-local spaces, it is an equivalence.  This
proves (ii) on components of $1$, and it follows on other components.
\end{proof}

For a general example, consider $CY$ for an $L_+$-space $Y$.  

\begin{cor}\label{Tornehave}  There is a natural commutative diagram
of $E_{\infty}$ spaces
\[ \xymatrix{
\GA_1(CY,\tha)\ar[d] \ar[r]^-{\GA_1\al} 
& \GA_1(QY,\tha) \ar[d]& Q_1Y \ar[l]_-{\et_{\oplus}} \ar[d]\\
\GA_1(C_MY,\xi)\ar[r]_-{\GA_1\al} 
& \GA_1(Q_MY,\xi) & \GA_1(Q_MY,\xi) \ar@{=}[l] \\} \]
in which the horizontal arrows are weak equivalences
and the vertical arrows are localizations at $M$.
\end{cor}

Now specialize to the case $Y=S^0$.  Then $Q_1S^0$, the
unit component of the $0^{th}$ space of the sphere spectrum,
is the space $SL_1S = SF$ of degree $1$ stable homotopy equivalences 
of spheres.  We see that
its localization at $M$ is the infinite loop space constructed from the
$\sL$-space $C_MS^0$.
The latter space is the disjoint union of the Eilenberg-Mac\,Lane spaces
$\sC(m)/\SI_m = K(\SI_m,1)$, given an $E_{\infty}$ space 
structure that realizes the products $\SI_m\times \SI_n\rtarr \SI_{mn}$ 
determined by lexicographically ordering the products of sets of $m$ and $n$ 
elements for $m,n\in M$.  Thus the localizations of $SF$ can be recovered from
symmetric groups in a way that captures their infinite loop structures.

\section{$E_{\infty}$ ring spectra and commutative $S$-algebras}\label{EKMM}

Jumping ahead over twenty years, we here review the basic definitions of EKMM \cite{EKMM}, leaving
all details to that source. However, to establish context, let
us first recall the following result of Gaunce Lewis \cite{Lewis}. 

\begin{thm}\label{Lewis} Let $\sS$ be any category that is enriched in based topological
spaces and satisfies the following three properties.
\begin{enumerate}[(i)]
\item $\sS$ is closed symmetric monoidal under continuous smash product 
and function spectra functors $\sma$ and $F$ that satisfy the topological
adjunction 
\[ \sS(E\sma E', E'') \iso \sS(E,F(E',E'')). \]
\item There are continuous functors $\SI^{\infty}$ and
$\OM^{\infty}$ between spaces and spectra that satisfy the topological adjunction 
\[ \sS(\SI^{\infty}X,E)\iso \sT(X,\OM^{\infty}E). \]
\item The unit for the smash product in $\sS$ is $S\equiv \SI^{\infty}S^0$.
\end{enumerate}
Then, for any commutative monoid $R$ in $\sS$, such as $S$ itself, 
the component $SL_1(R)$ of the identity element in $\OM^{\infty} R$ is a product of 
Eilenberg-MacLane spaces.
\end{thm}
\begin{proof} The enrichment of the adjunctions means that the 
displayed isomorphisms are homeomorphisms. By \cite[3.4]{Lewis},  
the hypotheses imply that $\sS$ is tensored over $\sT$. In turn, 
by \cite[3.2]{Lewis}, this implies that the functor $\OM^{\infty}$
is lax symmetric monoidal with respect to the unit 
$\et\colon S^0\rtarr \OM^{\infty}\SI^{\infty}S^0 = \OM^{\infty}S$
of the adjunction and a natural transformation
\[ \ph\colon \OM^{\infty}D\sma \OM^{\infty}E\rtarr \OM^{\infty}(D\sma E).\]
Now let $D=E=R$ with product $\mu$ and unit $\et\colon S\rtarr R$. The adjoint
of $\et$ is a map $S^0\rtarr \OM^{\infty}R$, and we let $1\in \OM^{\infty}R$ 
be the image of $1\in S^0$.  The composite
\[\xymatrix@1{
\OM^{\infty}E\times \OM^{\infty}R \ar[r] &
\OM^{\infty}R\sma \OM^{\infty}R
\ar[r]^-{\ph} &\OM^{\infty}(R\sma R) \ar[r]^-{\OM^{\infty}\mu} & \OM^{\infty}R\\} \]
gives $\OM^{\infty}R$ a structure of commutative topological 
monoid with unit $1$. Restricting to the component $SL_1R$ of $1$, 
we have a connected commutative topological monoid,
and Moore's theorem (e.g. \cite[3.6]{Geo}) gives the conclusion.
\end{proof}

As Lewis goes on to say, if $\OM^{\infty}\SI^{\infty}X$ is homeomorphic
under $X$ to $QX$, as we have seen holds for the category $\sS$ of (LMS) spectra, and if (i)--(iii) hold, then we can conclude in particular that $SF=SL_1S$ is a product of  Eilenberg--Mac\,Lane spaces, which is false.   
We interpolate a model theoretic variant of this contradiction.

\begin{rem}\label{Lew2}
The sphere spectrum $S$ is a commutative ring spectrum in
any symmetric monoidal category of spectra $\sS$ with unit object $S$. 
Suppose that $S$ is cofibrant in some model structure on $\sS$ whose
homotopy category is equivalent to the stable homotopy category and
whose fibrant objects are $\OM$-prespectra.  More precisely, 
we require an underlying prespectrum functor $U\colon \sS\rtarr \sP$
such that $UE$ is an $\OM$-prespectrum if $E$ is fibrant, and we also
require the resulting $0^{th}$ space functor $U_0$ to be lax symmetric
monoidal. Then we cannot 
construct a model category of commutative ring spectra by letting the 
weak equivalences and fibrations be the maps that are weak equivalences 
and fibrations in $\sS$.  If we could, a fibrant approximation of $S$ 
as a commutative ring spectrum would be an $\OM$-spectrum whose $0^{th}$ 
space is equivalent to $QS^0$.  Its $1$-component would be a connected commutative monoid equivalent to $SF$. 
\end{rem}

All good modern categories of spectra satisfy (i) and (iii) (or their simplicially enriched analogues) and therefore cannot satisfy (ii).  
However, as our summary so far should make clear, one must not let go 
of (ii) lightly.  One needs something like it to avoid severing the 
relationship between spectrum and space level homotopy theory.  Our 
summaries of modern definitions will focus on the relationship between 
spectra and spaces. Since we are now switching towards a focus on stable homotopy theory, we start to keep track of model structures.  
Returning to our fixed category $\sS$ of (LMS) spectra, we shall describe a 
sequence of Quillen equivalences, in which the right adjoints labelled $\ell$ are both inclusions of subcategories.

\[ \xymatrix{
*+[F-:<9pt>]{\sP} \ar[d]<.5ex>^-{L} \\
*+[F-:<9pt>]{\sS} \ar[d]<.5ex>^-{\bL} \ar[u]<.5ex>^-{\ell}\\
*+[F-:<9pt>]{\bL[\sS]} \ar[d]<.5ex>^-{S\sma_{\sL}(-)} \ar[u]<.5ex>^-{\ell}\\
 *+[F-:<9pt>]{\sM_S} \ar[u]<.5ex>^-{F_{\sL}(S,-)} \\} \]
 
 \vspace{2mm}

%\[ \xymatrix@1{
%*+[F-:<9pt>]{\sP} \ar[rr]<.5ex>^-{L} 
%& & *+[F-:<9pt>]{\sS} \ar[rr]<.5ex>^-{\bL} \ar[ll]<.5ex>^-{\ell}
%& & *+[F-:<9pt>]{\bL[\sS]} \ar[rr]<.5ex>^-{S\sma_{\sL}(-)} \ar[ll]<.5ex>^-{\ell}
%& & *+[F-:<9pt>]{\sM_S} \ar[ll]<.5ex>^-{F_{\sL}(S,-)} \\}\]
The category $\sP$ of prespectra has a level model structure whose weak equivalences and fibrations are defined levelwise, and it has a stable model structure whose weak equivalences are the maps that induce
isomorphisms of (stabilized) homotopy groups and whose cofibrations are the level cofibrations; its fibrant objects are the $\OM$-prespectra. The category 
$\sS$ of spectra is a model category whose level model structure and stable model structures coincide.  That is, the weak equivalences and fibrations are 
defined levelwise, and these are already the correct stable weak equivalences
because the colimits that define the homotopy groups of a spectrum run over a 
system of isomorphisms.  The spectrification functor $L$ and inclusion $\ell$
give a Quillen equivalence between $\sP$ and $\sS$. 

Of course, $\sS$ satisfies (ii) but not (i) and (iii).
We take the main step towards the latter properties by introducing 
the category $\bL[\sS]$ of $\bL$-spectra. 

The space $\sL(1)$ is a monoid under composition, and we have the notion 
of an action of $\sL(1)$ on a spectrum $E$.  It is given by a map 
\[ \xi\colon \bL E = \sL(1)\thp E\rtarr E \]
that is unital and associative in the evident sense. 
Since $\sL(1)$ is contractible, the unit condition implies that 
$\xi$ must be a weak equivalence. Moreover, $\bL E$ is an $\bL$-spectrum
for any spectrum $E$, and the action map $\xi\colon \bL E\rtarr E$ is a map 
of $\bL$-spectra 
for any $\bL$-spectrum $E$. The inclusion $\ell\colon \bL[\sS]\rtarr \sS$ forgets the action maps. It is right adjoint to the free $\bL$-spectrum functor $\bL\colon \sS\rtarr \bL[\sS]$.  Define the weak equivalences and fibrations of $\bL$-spectra to be the maps 
$f$ such that $\ell f$ is a weak equivalence or fibration.  Then $\bL[\sS]$ 
is a model category and $(\bL,\ell)$ is a Quillen equivalence between $\sS$ and
$\bL[\sS]$.   Indeed, the unit $\et\colon E\rtarr \ell\bL E$ and counit 
$\xi\colon \bL\ell E\rtarr E$  of the adjunction are weak equivalences, 
and every object in both categories is fibrant, a very convenient property. 
 
Using the untwisting isomorphism (\ref{untwist}) and the projection 
$\sL(1)_+ \rtarr S^0$, we see that the spectra $\SI^{\infty}X$ are 
naturally $\bL$-spectra.  However $S =\SI^{\infty} S^0$, which is cofibrant 
in $\sS$, is not cofibrant in $\bL[\sS]$; rather, $\bL S$ is a cofibrant approximation.

We have a commutative and associative\footnote{This crucial property is a consequence of a remarkable motivating observation, due to Mike Hopkins, about special properties of the 
structure maps of the linear isometries operad.}  smash product 
$E\sma_\sL E'$ in $\bL[\sS]$; we write $\ta$ for the commutativity
isomorphism $E\sma_{\sL} E'\rtarr E'\sma_{\sL} E$.  The smash product is defined as a
coequalizer $\sL(2)\thp_{\sL(1)\times \sL(1)} E\sma E'$, but we refer the
reader to \cite{EKMM} for details.  There is a natural unit 
map $\la\colon S\sma_{\sL} E\rtarr E$.  It is a weak equivalence for all
$\bL$-spectra $E$, but it is not in general an isomorphism.  That is, 
$S$ is only a weak unit.  

Moreover, there is a natural isomorphism of $\bL$-spectra
\[ \SI^{\infty}(X\sma Y)\iso \SI^{\infty} X\sma_{\sL} \SI^{\infty} Y.\]
Note, however, that the $(\SI^{\infty},\OM^{\infty})$ adjunction must now change.  We may reasonably
continue to write $\OM^{\infty}$ for the $0^{th}$ space functor 
$\OM^{\infty}\com \ell$, but its left adjoint is now the composite
$\bL\com\SI^{\infty}$.  

While Lewis's contradictory desiderata do not hold, we are not too 
far off since we still have a sensible $0^{th}$ space functor.  We 
are also very close to a description of $E_{\infty}$ ring spectra as 
commutative monoids in a symmetric monoidal category.

\begin{defn}\label{Lcommon}
A commutative monoid in $\bL[\sS]$ is an $\bL$-spectrum
$R$ with a unit map $\et\colon S\rtarr R$ and a commutative
and associative product $\ph\colon R\sma_{\sL} R\rtarr R$ such that
the following unit diagram is commutative 
\[ \xymatrix{
S\sma_{\sL} R\ar[r]^-{\et\sma\id} \ar[dr]_{\la} 
& R\sma_{\sL} R \ar[d]^{\ph} 
& R\sma_{\sL} S \ar[l]_-{\id\sma\et}
\ar[dl]^{\la\ta} \\
& R & }\]
\end{defn}

The only difference from an honest commutative monoid is that
the diagonal unit arrows are weak equivalences rather than 
isomorphisms.  Thinking of the unit maps $e\colon S\rtarr R$
of $E_{\infty}$ ring spectra as preassigned, we can specify a
product $\ast$ in the category $\bL[\sS]_e$ of $\bL$-spectra under $S$
that gives that category a symmetric monoidal structure, and then 
an $E_{\infty}$ ring spectrum is an honest commutative monoid in
that category \cite[XIII.1.16]{EKMM}.  We prefer to keep to the 
perspective of Definition \ref{Lcommon}, and \cite[II.4.6]{EKMM} gives the 
following result. 

\begin{thm}  The category of commutative monoids in $\bL[\sS]$ is 
isomorphic to the category of $E_{\infty}$ ring spectra.
\end{thm}

Obviously, we have not lost the connection with $E_{\infty}$ ring spaces.
Since we are used to working in symmetric monoidal categories and want to 
work in a category of spectra rather than of spectra under $S$, we take 
one further step. The unit map $\la\colon S\sma_{\sL}E\rtarr E$ is often
an isomorphism.  This holds when $E=\SI^{\infty}X$ and when 
$E = S\sma_{\sL} E'$ for another $\bL$-spectrum $E'$.  We define an 
$S$-module to be an $\bL$-spectrum $E$ for which $\la$ is an isomorphism, 
and we let $\sM_S$ be the category of $S$-modules.  It is symmetric monoidal with unit $S$ under the smash product $E\sma_S E' = E\sma_{\sL} E'$ that is
inherited from $\bL[\sS]$, and it also inherits a natural isomorphism of $S$-modules
\[ \SI^{\infty}(X\sma Y)\iso \SI^{\infty} X\sma_S \SI^{\infty} Y.\]

Commutative monoids in $\sM_S$ are called commutative $S$-algebras.
They are those commutative monoids in $\bL[\sS]$ whose unit maps are isomorphisms.  Thus they are especially nice $E_{\infty}$ ring spectra. 
For any $E_{\infty}$ ring spectrum $R$, $S\sma_{\sL} R$ is a commutative
$S$-algebra and the unit equivalence  $S\sma_{\sL} R\rtarr R$ is a map of 
$E_{\infty}$ ring spectra.  Thus there is no real loss of generality in 
restricting attention to the commutative $S$-algebras.  Their $0^{th}$ 
spaces are still $E_{\infty}$ ring spaces.   

However the $0^{th}$ space functor $\OM^{\infty}\colon \sM_S\rtarr \sT$ is not
a right adjoint.  The functor 
$S\sma_{\sL}(-)\colon \bL[\sS]\rtarr \sM_S$ is {\em right} adjoint
to the inclusion $\ell\colon \sM_S\rtarr \bL[\sS]$, and it has a right adjoint
$F_{\sL}(S,-)$.  Thus, for $D\in\bL[\sS]$ and $E\in \sM_S$,  we have
\[ \sM_S(E, S\sma_{\sL} D)\iso \bL[\sS](\ell E, D)\]
and
\[ \sM_S(S\sma_{\sL}D, E)\iso  \bL[\sS](D, F_{\sL}(S,E)). \]
Letting the weak equivalences and fibrations in $\sM_S$ be created by the functor $F_{\sL}(S,-)$, the second adjunction gives a Quillen equivalence between $\bL[\sS]$ and $\sM_S$.  Since there is a natural weak equivalence 
$\tilde {\la}\colon \ell E\rtarr F_{\sL}(S,E)$, the weak equivalences, but 
not the fibrations, are also created by $\ell$.  On the $0^{th}$ space level, 
$\tilde{\la}$ induces a natural weak equivalence 
\[  \OM^{\infty} \ell E\htp \OM^{\infty} F_{\sL}(S,E). \]
We conclude that we have lost no $0^{th}$ space information beyond that which
would lead to a contradiction to Theorem \ref{Lewis} in our passage from $\sS$ to 
$\sM_S$. 

As explained in \cite[II\S2]{EKMM}, there is actually a ``mirror image'' category $\sM^S$ that is equivalent
to $\sM_S$ and whose $0^{th}$ space functor is equivalent, rather than just
weakly equivalent, to the right adjoint of a functor $\sS\rtarr \sM^S$.
It is the subcategory of objects in $\bL[\sS]$ whose counit maps 
$\tilde{\la}\colon E\rtarr F_{\sL}(S,E)$ are isomorphisms. It has adjunctions
that are mirror image to those of $\sM_S$, switching left and right.
Writing $r\colon \sM^S\rtarr \bL[\sS]$ for the inclusion and taking
$D\in\bL[\sS]$ and $E\in \sM^S$, we have 
 \[ \sM^S(F_{\sL}(S,D),E) \iso \bL[\sS](D,rE)\]
 and
\[ \sM^S(E, F_{\sL}(S,D))\iso  \bL[\sS](\ell (S\sma_{\sL} rE), D). \]

The following theorem from \cite[VII\S4]{EKMM} is more central to our story
and should be compared with Remark \ref{Lew2}. 

\begin{thm}\label{QuillE}  
The category of $E_{\infty}$ ring spectra is a Quillen
model category with fibrations and weak equivalences created by the
forgetful functor to $\bL[\sS]$.  The category of commutative $S$-algebras 
is a Quillen model category with fibrations and weak equivalences created 
by the forgetful functor to $\sM_S$.
\end{thm}

\section{The comparison with commutative 
diagram ring spectra}\label{MMSS}

For purposes of comparison and to give some completeness to this survey,
we copy the following schematic diagram of Quillen 
equivalences\footnote{There is a caveat in that $\sF\sT$ only models 
connective spectra.} from \cite{MMSS}.

{ \small 

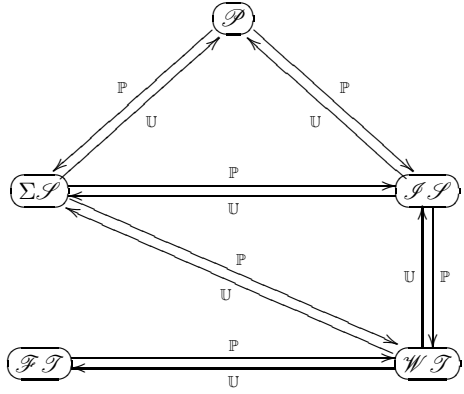
\[ \xymatrix{
& & *+[F-:<5pt>]{\sP} \ar[ddrr]<.5ex>^{\bP} \ar[ddll]<-.5ex>_{\bP} & & & &\\
& & & & & &  \\
*+[F-:<5pt>]{\SI\sS}  \ar[0,4]<.5ex>^{\bP} 
\ar[2,4]<.5ex>^{\bP} \ar[uurr]<-.5ex>_{\bU}
 & & & & *+[F-:<5pt>]{\sI \sS} \ar[0,-4]<.5ex>^{\bU} 
\ar[2,0]<.5ex>^{\bP} \ar[uull]<.5ex>^{\bU}\\
 & & & & & &  \\ 
 *+[F-:<5pt>]{\sF\sT}  \ar[0,4]<.5ex>^{\bP}  & & & &
 *+[F-:<5pt>]{\sW\sT}  \ar[-2,0]<.5ex>^{\bU} 
\ar[0,-4]<.5ex>^{\bU} \ar[-2,-4]<.5ex>^{\bU} \\} \]

}

We have a lexicon:

\begin{enumerate}[(i)]
\item  $\sP$ is the category of $\sN$-spectra, or (coordinatized) prespectra. 
\vspace{.5mm}
\item  $\SI  \sS$ is the category of $\SI$-spectra, or symmetric spectra.
\vspace{.5mm}
\item  $\sI \sS$ is the category of $\sI$-spectra, or orthogonal spectra.
\vspace{.5mm}
\item  $\sF\sT$ is the category of $\sF$-spaces, or $\GA$-spaces.
\vspace{.5mm}
\item  $\sW\sT$ is the category of $\sW$-spaces. 
\end{enumerate}

These categories all start with some small (topological)
category $\sD$ and the category $\sD\sT$ of (continuous) covariant functors $\sD\rtarr \sT$,
which are called $\sD$-spaces.  The domain categories have inclusions 
among them, as indicated in the following diagram of domain categories $\sD$.

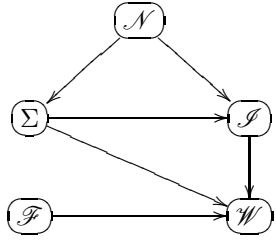
\[ \xymatrix{
& *+[F-:<5pt>]{\sN} \ar[dr] \ar[dl] &  \\
*+[F-:<5pt>]{\SI}  \ar[rr] \ar[drr]
[ar][dr]<.0ex> 
 & & *+[F-:<5pt>]{\sI}  \ar[d] [ar][2,0] \\
 *+[F-:<5pt>]{\sF}  \ar[rr]  & &
*+[F-:<5pt>]{\sW}\\}  \]

To go from $\sD$-spaces to $\sD$-spectra, one starts with a sphere
space functor $S\colon \sD \rtarr \sT$ with smash products.  It makes sense
to define a module over $S$, and the $S$-modules are the $\sD$-spectra.
Alternatively but equivalently, one can 
use $S$ to build a new (topological) domain category $\sD_{S}$ such that 
a $\sD_S$-space is a $\sD$-spectrum.  Either way, we obtain the category
$\sD\sS$ of $\sD$-spectra.   When $\sD=\sF$ or $\sD=\sW$, there is no distinction between $\sD$-spaces and $\sD$-spectra and $\sD\sT=\sD\sS$.

In the previous diagram, $\sN$ is the category of non-negative integers,
$\SI$ is the category of symmetric groups, $\sI$ is the category of 
linear isometric isomorphisms as before, $\sF$ is the category of finite 
based sets, which is the opposite category of Segal's category $\GA$, and 
$\sW$ is the category of based spaces that are homeomorphic to finite CW complexes.  The functors $\bU$ in the first diagram are forgetful functors associated to these inclusions of domain categories, and the functors $\bP$ are prolongation functors left adjoint to the $\bU$. All of these categories except $\sP$ are symmetric monoidal. The reason is that the functor $S\colon \sD\rtarr \sT$ is symmetric monoidal in the other cases, but not in the case of $\sN$.  The functors $\bU$ between symmetric monoidal categories are lax symmetric monoidal, the functors 
$\bP$ between symmetric monoidal categories are strong symmetric monoidal, and the functors $\bP$ and $\bU$ restrict to adjoint pairs relating the various categories of rings, commutative rings, and modules over rings. 

We are working with spaces but, except that orthogonal spectra should be
omitted, we have an analogous diagram of categories of spectra 
that are based on simplicial sets \cite{BF, HSS, Lyd1, Lyd2, Sch1}.  That
diagram compares to ours via the usual adjunction between simplicial sets
and topological spaces.  Each of these categories of spectra has intrinsic interest, and they have various advantages and disadvantages.  We focus implicitly on symmetric and orthogonal spectra in what follows; up to a 
point, $\sW$-spaces and $\sF$-spaces work similarly. Full details are in
\cite{MMSS} and the references just cited.

We emphasize that no non-trivial symmetric (or orthogonal) spectrum $E$ can 
also be an LMS spectrum.  If it were, its $0^{th}$ space $E_0$, with trivial 
$\SI_2$-action, would be homeomorphic as a $\SI_2$-space to the non-trivial
$\SI_2$-space $\OM^2 E_2$. 

We recall briefly how smash products are defined in 
diagram categories.  There are two equivalent ways. Fix a symmetric
monoidal domain category $\sD$ with product denoted $\oplus$. For
$\sD$-spaces $T$ and $T'$, there is an external smash product
$T\barwedge T'$, which is a $\sD\times \sD$-space.  It is specified by
\[ (T\barwedge T')(d,e) = Td\sma T'e. \]
Applying left Kan extension along $\oplus$, one obtains a $\sD$-space
$T \sma T'$.  This construction is characterized by an adjunction
\[  ((\sD\times\sD)\sT)(T\barwedge T', V\com \oplus)\iso \sD\sT(T\sma T',V) \]
for $\sD$-spaces $V$.  When $T$ and $T'$ are $S$-modules, one can construct a 
``tensor product'' $T\sma_S T'$ by mimicking the coequalizer description of the 
tensor product of modules over a commutative ring.  That gives the required
internal smash product of $S$-modules.  Alternatively and equivalently, one
can observe that $\sD_S$ is a symmetric monoidal category when 
$S\colon \sD\rtarr \sT$ is a symmetric monoidal functor, and one can
then apply left Kan extension directly, with $\sD$ replaced by $\sD_S$.
Either way, $\sD\sS$ becomes a symmetric monoidal category with unit $S$. 

In view of the use of left Kan extension, monoids $R$ in $\sD\sS$ have an
external equivalent defined in terms of maps $R(d)\sma R(e)\rtarr R(d\oplus e)$.
These are called $\sD$-FSP's.  As we have already recalled, Thom spectra give 
naturally occurring examples of $\sI$-FSP's.

We also recall briefly how the model structures are defined.  We begin
with the evident level model structure. Its weak equivalences and fibrations 
are defined levelwise. We then define stable weak equivalences and use them
and the cofibrations of the level model structure to construct the stable model structure.  The resulting fibrant objects are the $\OM$-spectra.  In all of these categories except that of symmetric spectra, the stable weak equivalences are the maps whose underlying maps of prespectra induce isomorphisms of stabilized homotopy groups.\footnote{The ``underlying prespectrum'' of an 
$\sF$-space is obtained by first prolonging it to a $\sW$-space and then taking the underlying prespectrum of that, and we are suppressing some details.}  
It turns out that a map $f$ of symmetric spectra is a stable weak equivalence 
if and only if $\bP f$ is a stable weak equivalence of orthogonal spectra in 
the sense just defined.  There are other model structures here, as we shall see. The thing to notice is that, in the model structures just specified, the sphere spectra $S$ are cofibrant.  Compare Remark \ref{Lew2} and Theorem \ref{QuillE}.   

These model structures are compared in \cite{MMSS}.  Later
work of Schwede and Shipley \cite{Schw, SS, Ship} gives $\SI \sS$ a 
privileged role.  Given any other sufficiently good stable model 
category whose homotopy category is correct, in the sense that it is equivalent 
to Boardman's original stable homotopy category, there is a left Quillen equivalence from $\SI \sS$  to that category.  

However, this is not always the best way to compare two 
models for the stable homotopy category.  If one has models $\sS_1$ and 
$\sS_2$ and compares both to $\SI \sS$, then $\sS_1$ and $\sS_2$ are 
compared by a zigzag of Quillen equivalences.  It is preferable to avoid 
composing left and right Quillen adjoints, since such
composites do not preserve structure.  For example, using a necessary modification 
of the model structure on $\SI\sS$ to be explained shortly, Schwede gives a left 
Quillen equivalence $\SI\sS \rtarr \sM_S$ \cite{Schw}, 
and \cite{MMSS} shows that the prolongation functor $\bP\colon \SI\sS\rtarr \sI\sS$ 
is a left Quillen equivalence.  This gives a zigzag of Quillen equivalences between 
$\sM_S$ and $\sI\sS$.   These categories are both defined using $\sI$, albeit in quite 
different ways, and it is more natural and useful to construct a left Quillen equivalence 
$\bN\colon \sI\sS\rtarr \sM_S$.  Using a similar necessary modification of the
model structure on $\sI\sS$, this is done in Mandell and May \cite[Ch. I]{MM}; 
Schwede's left Quillen equivalence is the composite $\bN\com\bP$.

In any case, there is a web of explicit Quillen equivalences relating
all good known models for the stable homotopy category, and these
equivalences even preserve the symmetric monoidal structure and so 
preserve rings, modules, and algebras \cite{MM, MMSS, Schw, Ship}.  
Thus, as long as 
one focuses on stable homotopy theory, any convenient model 
can be used, and information can easily be transferred from one to another.
More precisely, if one focuses on criteria (i) and (iii)  of Theorem \ref{Lewis}, 
one encounters no problems. However, our focus is on (ii), the relationship 
between spectra and spaces, and here there are significant problems.  

For a start, it is clear that we cannot have a symmetric monoidal
Quillen left adjoint from $\SI\sS$ or $\sI\sS$ to $\sM_S$ with the model structures
that we have specified since the sphere spectra in $\SI\sS$ and $\sI\sS$ are cofibrant 
and the sphere spectrum in $\sM_S$ is not.   For the comparison, one must 
use different model structures on $\SI\sS$ and $\sI\sS$, namely the positive stable model 
structures.  These are obtained just as above but starting with the level model structures whose 
weak equivalences and fibrations are defined using only the positive levels, not the $0^{th}$ 
space level.  This does not change the stable weak equivalences, and the resulting positive 
stable model structures are Quillen 
equivalent to the original stable model structures.  

However, the 
fibrant spectra are now the positive $\OM$-spectra, for which the structure
maps $\tilde{\si}\colon T_n\rtarr \OM T_{n+1}$ of the underlying 
prespectrum are weak equivalences only for $n>0$.  This in principle
throws away all information about the $0^{th}$ space, even after
fibrant approximation. The analogue of Theorem \ref{QuillE} reads as follows.
Actually, a significant technical improvement of the positive stable
model structure has been obtained by Shipley \cite{Ship2}, but her
improvement does not effect the discussion here: one still must use the
positive variant. 

\begin{thm}\label{QuillE2}
The categories of commutative symmetric ring spectra and commutative
orthogonal ring spectra have Quillen model structures whose weak
equivalences and fibrations are created by the forgetful functors
to the categories of symmetric spectra and orthogonal spectra with
their positive stable model structures.
\end{thm}

Parenthetically, as far as I know it is unclear whether or not 
there is an analogue of this result for $\sW$-spaces.  The results of
\cite{MM, MMSS, Schw, Ship} already referred to give the following 
comparisons, provided that we use the positive model structures on the
diagram spectrum level. 

\begin{thm}  There are Quillen equivalences from the category of
commutative symmetric ring spectra to the category of commutative
orthogonal ring spectra and from the latter to the category of
commutative $S$-algebras.
\end{thm}

Thus we have comparison functors
\[  \xymatrix{
\text{Commutative symmetric ring spectra}\ar[d]^{\bP} \\
\text{Commutative orthogonal ring spectra}\ar[d]^{\bN} \\
\text{Commutative $S$-algebras} \\
E_{\infty}\ \, \text{ring spectra} 
\ar[u]_{S\sma_{\sL}(-)} \ar[d]^{\OM^{\infty}}\\
E_{\infty}\ \, \text{ring spaces}.\\} \]
The functors $\bP$, $\bN$, and $S\sma_{\sL}(-)$ are {\em left}\, 
Quillen equivalences.  The functor $\OM^{\infty}$ is a
{\em right}\, adjoint.  The composite is not homotopically 
meaningful since, after fibrant approximation, commutative symmetric
ring spectra do not have meaningful $0^{th}$ spaces; in fact, their
$0^{th}$ spaces are then just $S^0$.  If one only uses diagram spectra,
the original $E_{\infty}$ ring theory relating spaces and spectra
is lost.

\section{Naive $E_{\infty}$ ring spectra}\label{naive}

Again recall that we can define operads and operad actions in any symmetric 
monoidal category.  If a symmetric monoidal category $\sW$ is tensored over a
symmetric monoidal category $\sV$, then we can just as well define actions
of operads in $\sV$ on objects of $\sW$.  All good modern categories of 
spectra are tensored over based spaces (or simplicial sets).  We can
therefore define an action of an operad $\sO_+$ in $\sT$ on a spectrum in any
such category.  Continuing to write $\sma$ for the tensor of a space and a spectrum, 
such an action on a spectrum $R$ is given by maps of spectra
\[  \sO(j)_+ \sma R^{(j)}\rtarr R. \]  
Taking $\sO$ to be an $E_{\infty}$ operad, we call such $\sO$-spectra
naive $E_{\infty}$ ring spectra.\footnote{In recent e-mails, Tyler Lawson 
has jokingly called these MIT $E_{\infty}$ ring spectra, to contrast them
with the original Chicago variety.  He galerted me to the fact that some
people working in the area may be unaware of or indifferent to the distinction.}
They are defined in terms of the already constructed internal smash product
and thus have nothing to do with the internalization of an external smash product
that is intrinsic to the original definition of $E_{\infty}$ ring spectra.

They are of interest because some natural constructions land in naive
$E_{\infty}$ spectra (of one kind or another).  In some cases, such as 
$\sW$-spaces, where we do not know of a model structure on commutative ring spectra, 
naive $E_{\infty}$ ring spectra provide an adequate stopgap.  In other cases, 
including symmetric spectra, orthogonal spectra, and $S$-modules, we
can convert naive $E_{\infty}$ ring spectra to equivalent commutative
ring spectra, as we noted without proof in \cite[0.14]{MMSS}. The
reason is the following result, which deserves considerable emphasis.
See \cite[III.5.5]{EKMM}, \cite[15.5]{MMSS}, and, more recently and 
efficiently, \cite[3.3]{Ship2} for the proof.

\begin{prop}\label{rational} For a positive cofibrant symmetric or 
orthogonal spectrum or for a cofibrant $S$-module $E$, the projection
\[\pi\colon (E\SI_j)_+\sma_{\SI_j} E^{(j)} \rtarr E^{(j)}/\SI_j  \]
(induced by $E\SI_j\rtarr *$) is a weak equivalence.
\end{prop}  

This is analogous to something that is only true in characteristic
zero in the setting of differential graded modules over a field.   
In that context, it implies that $E_{\infty}$ DGA's can be 
approximated functorially by quasi-isomorphic commutative 
DGA's \cite[II.1.5]{KM}.  The following result is precisely
analogous to the cited result and can be proven in much the 
same way.  That way, it is another exercise in the use of the two-sided 
monadic bar construction.  The cofibrancy issues can be handled 
with the methods of \cite{Ship}, and I have no doubt that in all cases
the following result can be upgraded to a Quillen equivalence. The more 
general result \cite[1.4]{EM} shows this to be true for
simplicial symmetric spectra and suitable simplicial operads, so 
I will be purposefully vague and leave details to the interested reader.  
Let $\sO$ be an $E_{\infty}$ operad and work in one of the categories 
of spectra cited in Proposition \ref{rational}.  

\begin{prop}\label{rational2} There is a functor that assigns a weakly 
equivalent commutative ring spectrum to a (suitably cofibrant) naive 
$\sO$-spectrum. The homotopy categories of naive $\sO$-spectra and 
commutative ring spectra are equivalent.
\end{prop}

It is immediately clear from this result and the discussion in the 
previous section that naive $E_{\infty}$ ring spectra in $\SI\sS$ and
$\sI\sS$ have nothing to do with $E_{\infty}$ ring spaces, whereas 
the $0^{th}$ spaces of naive $E_{\infty}$ ring spectra in $\sM_S$ are
weakly equivalent to $E_{\infty}$ ring spaces.

\section{Appendix A. Monadicity of functors and comparisons of monads}

Change of monad results are well-known to category theorists, but perhaps
not as readily accessible in the categorical literature as they might be, 
so we give some elementary details here.\footnote{This appendix is written
jointly with Michael Shulman.}  We first make precise two notions of a 
map relating monads $(C,\mu,\io)$ and $(D,\nu,\ze)$ in different 
categories $\sV$ and $\sW$.  We have used both,
relying on context to determine which one is intended.  

\begin{defn}\label{moniso} Let $(C,\mu,\io)$ and $(D,\nu,\ze)$ be monads on 
categories $\sV$ and $\sW$.  
An op-lax map $(F,\al)$ from $C$ to $D$ is a functor $F\colon \sV\rtarr \sW$
and a natural transformation $\al\colon FC\rtarr DF$ such that the following
diagrams commute.
\[ \xymatrix{
FCC \ar[d]_{F\mu} \ar[r]^-{\al C} & DFC \ar[r]^-{ D\al} & DDF \ar[d]^{\nu F}\\
FC \ar[rr]_-{\al} & & DF \\ } 
\ \  \ \text{and} \ \ \ 
 \xymatrix{
& F \ar[dl]_{F\io} \ar[dr]^{\ze F} & \\
FC \ar[rr]_-{\al} & & DF \\  } \]
A lax map $(F,\be)$ from $C$ to $D$ is a functor $F\colon \sV\rtarr \sW$
and a natural transformation $\be\colon DF\rtarr FC$ such that the following
diagrams commute.
\[ \xymatrix{
DDF \ar[d]_{\nu F} \ar[r]^-{D \be} & DFC \ar[r]^-{\be D} & FCC \ar[d]^{F\mu}\\
DF \ar[rr]_-{\be} & & FC \\ } 
\ \  \ \text{and} \ \ \ 
 \xymatrix{
& F \ar[dl]_{\ze F} \ar[dr]^{F\io } & \\
DF \ar[rr]_-{\be} & & FC \\  } \]
If $\al\colon FC\rtarr DF$ is a natural isomorphism, then $(F,\al)$ is an 
op-lax map $C\rtarr D$ if and only if $(F,\al^{-1})$ is a lax map $D\rtarr C$.  When this holds, we say that $\al$ and $\al^{-1}$ are monadic natural isomorphisms.
\end{defn}

These notions are most familiar when $\sV=\sW$ and $F = \text{Id}$.  
In this case, a lax map $D\rtarr C$ coincides with an op-lax map 
$C\rtarr D$, and this is the usual notion of a map of monads from
$C$ to $D$. The map $\al\colon C\rtarr Q$ used in the approximation 
theorem and the recognition principle is an example. As we have used extensively, maps $(\text{Id},\al)$ lead to pullback of action functors.

\begin{lem} If $(\text{Id},\al)$ is a map of monads $C\rtarr D$ on
a category $\sV$, then a left or right action of $D$ on a functor induces 
a left or right action of $C$ by pullback of the action along $\al$.
In particular, if $(Y,\ch)$ is a $D$-algebra in $\sV$, then $(Y,\ch\com \al)$ 
is a $C$-algebra in $\sV$. 
\end{lem}

We have also used pushforward actions when $\sV$ and $\sW$ vary, and for 
that we need lax maps.

\begin{lem} If $(F,\be)$ is a lax map from a monad $C$ on $\sV$ to a monad $D$ on $\sW$, then a left or right action of $C$ on a functor induces 
a left or right action of $D$ by pushforward of the action along $(F,\be)$. 
In particular, if $(X,\xi)$ is a $C$-algebra in $\sV$, then 
$(FX,F\xi\com \be)$ is a $D$-algebra in $\sW$.
\end{lem}

Now let $F$ have a right adjoint $U$. Let $\et\colon \text{Id}\rtarr UF$ and 
$\epz\colon FU\rtarr \text{Id}$ be the unit and counit of the adjunction.  
We have encountered several examples of monadic natural isomorphisms $(F,\be)$ 
relating a monad $C$ in $\sV$ to a monad $D$ in $\sW$, where $F$ has a left
adjoint $U$.  Thus $\be$ is a natural isomorphism $DF\rtarr FC$.  In this 
situation, we have a natural map $\de\colon CU\rtarr UD$, namely the composite
\[ \xymatrix@1{
CU\ar[r]^-{\et CU} & UFCU \ar[r]^-{U\be^{-1}U} & UDFU \ar[r]^{UD\epz} & UD.\\ }
\]
It is usually not an isomorphism, and in particular is not an isomorphism
in our examples.
Implicitly or explicitly, we have several times used the following result.

\begin{prop}\label{Appcute} The pair $(U,\de)$ is a lax map from 
the monad $D$ in 
$\sW$ to the monad $C$ in $\sV$.  Via pushforward along $(F,\be)$ and 
$(U,\de)$, the adjoint pair $(F,U)$ induces an adjoint pair of functors
between the categories $C[\sV]$ and $D[\sW]$ of $C$-algebras in $\sV$ and 
$D$-algebras in $\sW$:
\[  D[\sW](FX, Y)\iso C[\sV](X,UY). \]
\end{prop}
\begin{proof}[Sketch proof] The arguments are straightforward
diagram chases. The essential point is that, for a $C$-algebra $(X,\xi)$
and a $D$-algebra $(Y,\ch)$, the map $\et\colon X\rtarr UFX$ is a map
of $C$-algebras and the map $\epz\colon FUY\rtarr Y$ is a map of 
$D$-algebras.  
\end{proof}

These observations are closely related to the categorical study of 
monadicity.  A functor $U\colon \sW\rtarr \sV$ is said to be 
{\em monadic} if it has a left adjoint $F$ such that $U$ induces
an equivalence from $\sV$ to the category of algebras over the 
monad $UF$.  This is a property of the functor $U$.  If $U$ is monadic,
then its left adjoint $F$ and the induced monad $UF$ such that $\sV$ is equivalent to the category of $UF$-algebras are uniquely determined by 
$\sV$, $\sW$, and the functor $U$.

This discussion illuminates the comparisons of monads in \S\ref{monad}
and \S\ref{monad2}.  For the first it is helpful to consider the following 
diagram of forgetful functors. 
\[\xymatrix{\sT_e\ar[d] &
  \sO\text{-spaces with zero} \ar[l]\ar[dl] \\
  \sT \ar[d] & \sO\text{-spaces}\ar[dl] \ar[l]\\
  \sU}
\]
Note that the operadic unit point is $1$ in $\sO\text{-spaces with zero}$
but $0$ in $\sO$-spaces; the diagonal arrows are obtained by 
forgetting the respective operadic unit points.  These forgetful functors are all monadic.  
If we abuse
notation by using the same name for each left adjoint and for the monad
induced by the corresponding adjunction, then we have the following diagram of left adjoints.
\[\xymatrix{\sT_e \ar[r]^<>(.5){O^\sT} &
  \sO\text{-spaces with zero} \\
  \sT \ar[u]^{S^0\vee -} \ar[r]_<>(.5){O^\sU} \ar[ur]^<>(.3){O^\sT_+} &
  \sO\text{-spaces} \\
  \sU \ar[u]^{(-)_+} \ar[ur]_{O^\sU_+}}
\]
Since the original diagram of forgetful functors obviously commutes,
so does the corresponding diagram of left adjoints. This formally implies 
the relations
\[ O^\sU(X_+) \iso O_+^\sU(X) \ \ \ \text{and}\ \ \ 
  O^\sT(S^0\vee X) \iso O^\sT_+(X) \]
of (\ref{Omore1}) and (\ref{Omore2}). 
The explicit descriptions of the four monads $O^\sU_+$, $O^\sU$,
$O^\sT_+$, and $O^\sT$ are, of course, necessary to the applications,
but it is helpful conceptually to remember that their definitions are
forced on us by knowledge of the corresponding forgetful functors.
As an incidental point, it is also important to remember that, unlike 
the case of adjunctions, the composite of two monadic functors need 
not be monadic, although it is in many examples, such as those above.

Similarly, for \S\ref{monad2}, it is helpful to consider
two commutative diagrams of forgetful functors. In both, 
all functors other than the $\OM^{\infty}$ are monadic. 
The first is 
\[\xymatrix{
\sS_e\ar[d] & \sL\text{-spectra under}\  S \ar[l] \ar[d]^{\OM^{\infty}} \ar[dl]\\
\sS \ar[d]_{\OM^{\infty}} &  \sL\text{-spaces with zero}\ar[dl]\\
\sT. & \\ }
\]
The lower diagonal arrow forgets the action of $\sL$ and remembers
the basepoint $0$.  The corresponding diagram of left adjoints is
\[\xymatrix{
\sS_e  \ar[r]^-{L} & \sL\text{-spectra under}\  S \\
\sS \ar[u]^{S\wed -} \ar[ru]^(0.3){L_+} 
& \sL\text{-spaces with zero}\ar[u]_{\SI^{\infty}}\\
\sT \ar[u]^{\SI^{\infty}} \ar[ur]_{L_+}. &  \\  }
\]
Its commutativity implies the relations
\[ L(S\wed E) \iso L_+(E) \ \ \ \text{and}\ \ \ 
 L_+\SI^{\infty}X\iso \SI^{\infty}L_+ X \]
of (\ref{Omore3}) and (\ref{newmon}).  The second diagram of forgetful 
functors is
\[\xymatrix{
\sS_e\ar[d]_{\OM^{\infty}} & \sL\text{-spectra under}\  S \ar[l] \ar[d]^{\OM^{\infty}}\\
\sT_e \ar[d] &  \sL\text{-spaces with zero}\ar[d]\\
\sT_1 &  \sL\text{-spaces}. \ar[l] \\ }
\]
Here $\sT_1$ denotes the category of based spaces with basepoint $1$.
The lower two vertical arrows forget the basepoint $0$ and remember 
the operadic unit $1$ as basepoint.
The corresponding diagram of left adjoints is
\[\xymatrix{
\sS_e \ar[r]^-{L} & \sL\text{-spectra under}\  S \\
\sT_e \ar[u]^{\SI^{\infty}} &  \sL\text{-spaces with zero}\ar[u]_{\SI^{\infty}}\\
\sT_1 \ar[u]^{(-)_+} \ar[r]_-{L} &  \sL\text{-spaces}. \ar[u]_{(-)_+} \\ }
\]
This implies the relation $L\SI^{\infty}(X_+) \iso \SI^{\infty} (LX)_+$ of (\ref{oldmon}),
which came as a computational surprise when it was first discovered.

\section{Appendix B. Loop spaces of $E_{\infty}$ spaces and the recognition principle}

Let $X$ be an $\sO$-space, where $\sO$ is an $E_{\infty}$ operad.  Either replacing $X$ by
an equivalent $\sC$-space or using the additive infinite loop space machine on 
$\sC\times \sO$-spaces, we construct a spectrum $\bE X$ as in Theorem \ref{add}.  For definiteness, 
we use notations corresponding to the first choice.  As promised in Remark \ref{OME}, we shall 
reprove the following result.  The proof will give more precise information than the statement,
and we will recall a consequence that will be relevant to our discussion of orientation theory in
the second sequel \cite{Sequel2} after giving the proof.

\begin{thm}\label{OME2} The space $\OM X$ is an $\sO$-space and there is a natural map of
spectra $\om\colon \SI \bE \OM X\rtarr \bE X$ that is a weak equivalence if $X$ is connected. 
Therefore its adjoint $\tilde{\om}\colon \bE\OM X\rtarr \OM\bE X$ is also a weak
equivalence when $X$ is connected.
\end{thm}

We begin with a general result on monads, but stated with notations that suggest our
application.  It is an elaboration of \cite[5.3]{Geo}. The proof is easy diagram chasing.

\begin{lem}\label{Oldup} Let $\sT$ be any category, let $C$ be a monad on $\sT$, and let
$(\SI,\OM)$ be an adjoint pair of endofunctors on $\sT$.  Let $\xi\colon \SI C\rtarr C\SI$
be a monadic natural isomorphism, so that the following diagrams commute.
\[
\xymatrix{
\SI C C \ar[d]_{\SI \mu} \ar[rr]^-{C\xi\com \xi C} & & CC\SI \ar[d]^{\mu\SI} \\
\SI C \ar[rr]_-{\xi} & & C\SI\\}
 \ \ \ \text{and} \ \ \ 
\xymatrix{
& \SI  \ar[dl]_{\SI\et} \ar[dr]^{\et \SI} & \\
\SI C \ar[rr]_{\xi} & & C\SI \\}
\]
\begin{enumerate}[(i)]
\item The functor $\OM C \SI$ is a monad on $\sT$ with 
unit and product the composites 
\[ \xymatrix@1{  \text{Id} \ar[r]^-{\et} &  \OM \SI \ar[r]^-{\OM\et\SI} & \OM C \SI \\} \]
and
\[ \xymatrix@1{ \OM C \SI \OM C \SI  \ar[rr]^-{\OM C\epz C\SI} & &
\OM C C \SI  \ar[rr]^-{\OM \mu \SI} & & \OM C \SI.\\} \]
Moreover, the adjoint $\tilde{\xi}\colon C\rtarr \OM C \SI$ of $\xi$ is a map of monads on $\sT$. 
\item If $(X,\tha)$ is a $C$-algebra, then $\OM X$ is an $\OM C\SI$-algebra with action map
\[ \xymatrix@1{  \OM C\SI \OM X \ar[rr]^-{\OM C\epz} & & \OM C X \ar[rr]^-{\OM \tha} & & \OM X,\\} \]
hence $\OM X$ is a $C$-algebra by pull back along $\tilde{\xi}$.
\item If $(F,\nu)$ is a $C$-functor ($F\colon \sT\rtarr \sV$ for some category $\sV$), then $F\SI$ is an $\OM C\SI$-functor with action transformation
\[\xymatrix@1{
F\SI \OM C \SI \ar[rr]^-{F\epz C\SI} & & F C\SI \ar[rr]^{\nu\SI} & & F\SI, \\} \]
hence $F\SI$ is a $C$-functor by pull back along $\tilde{\xi}$.
\end{enumerate}
If $\al\colon C\rtarr C'$ is a map of monads on $\sT$, then so is $\OM\al \SI\colon \OM C\SI\rtarr
\OM C'\SI$. 
\end{lem}

The relevant examples start with the loop suspension adjunction $(\SI,\OM)$ on $\sT$.

\begin{lem}\label{Newup}  For any (reduced) operad $\sC$ in $\sU$ with associated monad $C$ 
on $\sT$,
there is a monadic natural transformation $\xi\colon \SI C\rtarr C\SI$.  There is also
a monadic natural transformation $\rh\colon \SI Q\rtarr Q\SI$ such that the following
diagram commutes, where $\sC$ is the Steiner operad (or its product with any other operad).
\[  \xymatrix{
\SI C \ar[r]^-{\xi} \ar[d]_{\SI\al} & C\SI \ar[d]^{\al\SI} \\
\SI Q \ar[r]_{\rh} & Q\SI \\} \]
\end{lem}
\begin{proof}  For $c\in\sC(j)$, $x_i\in X$, and $t\in I$, we define 
\[ \xi((c;x_1,\cdots,x_j)\sma t) = (c;x_1\sma t,\cdots, x_j\sma t) \]
and check monadicity by diagram chases.   A point $f\in QX$ can be
represented by a map $f\colon S^n\rtarr X\sma S^n$ for $n$ sufficiently 
large and a point of $Q\SI X$ can be represented by a map $g\colon S^n\rtarr X\sma S^1 \sma S^n$.  
We define
\[  \rh(f\sma t)(y) = x\sma t\sma z, \]
where $y\in S^n$ and $f(y) = x\sma z\in X\sma S^n$ and check monadicity by somewhat laborious
diagram chases.  For the diagram, recall that $\al$ is the composite
\[ \xymatrix@1{  CX \ar[r]^{C\et} & CQX \ar[r]^{\tha} & QX \\} \]
and expand the required diagram accordingly to get
\[  \xymatrix{
\SI C \ar[rr]^-{\xi} \ar[d]_{\SI C\et } & & C\SI \ar[d]^{C\et \SI} \ar[dl]_{C\SI\et}\\
\SI CQX \ar[r]^-{\xi Q} \ar[d]_{\SI \tha} & C\SI QX \ar[r]^-{C\rh} & CQ\SI X \ar[d]^{\tha \SI} \\
\SI Q \ar[rr]_-{\rh} & &  Q\SI. \\} \]
The top left trapezoid is a naturality diagram and the top right triangle is easily seen
to commute by checking before application of $C$.  The bottom rectangle requires going 
back to the definition of the action $\tha$, but it is easily checked from that.
\end{proof} 

For any operad $C$, the action $\tilde{\tha}$ of $C$ on $\OM X$ induced via Lemma \ref{Oldup}(ii)
from an action $\tha$ of $C$ on $X$ is given by the obvious pointwise formula
\[   \tilde{\tha} (c;f_1,\cdots, f_j)(t) = \tha(c;f_1(t),\cdots, f_j(t)) \]
for $c\in\sC(j)$ and $f_i\in \OM X$.  The conceptual description leads to the
following proof.

\begin{proof}[Proof of Theorem \ref{OME2}]  For a spectrum $E$, $(\OM E)_0 = \OM(E_0)$, and 
it follows that we have a natural isomorphism of adjoints
$\chi\colon \SI \SI^{\infty} \rtarr \SI^{\infty}\SI$.  We claim that this is an
isomorphism of $C$-functors, where the action of $C$ on the right is given by
Lemma \ref{Oldup}(iii).  To see this, we recall that the action of $C$ on $\SI^{\infty}$ is
the composite
\[ \xymatrix@1{
\SI^{\infty}C \ar[r]^-{\SI^{\infty}\al} & \SI^{\infty} Q = 
\SI^{\infty}\OM^{\infty}\SI^{\infty}  \ar[r]^-{\epz\SI^{\infty}} 
& \SI^{\infty}\\} \]
and check that the following diagram commutes.
\[\xymatrix{
\SI\SI^{\infty} C \ar[r]^-{\ch C} \ar[d]_{\SI\SI^{\infty}\al} &
\SI^{\infty}\SI C \ar[d]_{\SI^{\infty}\SI\al} \ar[r]^{\SI^{\infty}\xi} &
\SI^{\infty} C\SI \ar[d]^{\SI^{\infty}\al\SI} \\
\SI\SI^{\infty} Q \ar[r]^-{\ch Q} \ar[d]_{\SI\epz} &
\SI^{\infty}\SI Q \ar[r]^{\SI^{\infty}\rh} &
\SI^{\infty} Q\SI \ar[d]^{\epz\SI^{\infty}\SI} \\
\SI\SI^{\infty} \ar[rr]_-{\ch} & & \SI^{\infty}\SI\\} \]
The top left square is a naturality diagram and the top right square is $\SI^{\infty}$
applied to the diagram of Lemma \ref{Newup}.  The bottom rectangle is another chase. 
The functor $\SI$ on spectra commutes with geometric realization, and there results
an identification
\begin{equation}\label{yes1} \SI \bE \OM X = \SI B(\SI^{\infty}, C, \OM X)
\iso B(\SI^{\infty}\SI , C, \OM X). \end{equation}
The action of $C$ on $\OM X$ is given by Lemma \ref{Oldup}(ii), and we have a map
\begin{equation}\label{yes2} B(\id,\tilde{\xi}, \id)\colon B(\SI^{\infty}\SI , C, \OM X)
\rtarr B(\SI^{\infty}\SI , \OM C\SI, \OM X). \end{equation}
The target is the geometric realization of a simplicial spectrum with $q$-simplices
\[\SI^{\infty}\SI(\OM C\SI)^q\OM X = \SI^{\infty}(\SI \OM C)^q\SI\OM X. \]
Applying $\epz\colon \SI\OM \rtarr \text{Id}$ in the $q+1$ positions, we obtain maps
\[ \SI^{\infty}(\SI \OM C)^q\SI\OM X\rtarr \SI^{\infty}C^q X. \]
By further diagram chases showing compatibility with faces and degeneracies, 
these maps specify a map of simplicial spectra. Its geometric realization is a map
\begin{equation}\label{yes3} B(\SI^{\infty}\SI , \OM C\SI, \OM X)\rtarr B(\SI^{\infty} , C, X) = \bE X. 
\end{equation}
Composing (\ref{yes1}), (\ref{yes2}), and (\ref{yes3}), we have the required map of spectra
\[ \om\colon \SI \bE \OM X \rtarr \bE X. \]
Passing to adjoints and $0^{th}$ spaces, we find that the following diagram commutes.
\[ \xymatrix{
& \OM X \ar[dl]_{\et} \ar[dr]^{\OM \et} & \\
\bE_0\OM X \ar[rr]_{\tilde{\om}_0} && \OM \bE_0 X. \\}\]
Since $\et$ is a group completion in general, both $\et$ and $\OM\et$ in the diagram
are equivalences when $X$ is connected and therefore $\om$ is then an equivalence.
\end{proof}

As was observed in \cite[3.4]{MayPer}, 
if $G$ is a monoid in $\sO[\sT]$, then $BG$ is an $\sO$-space such that $G\rtarr \OM BG$ 
is a map of $\sO$-spaces.  Since $(\OM\bE X)_1 = \OM \bE_1 X \iso \bE_0 X$,
Theorem \ref{OME2} has the following consequence \cite[3.7]{MayPer}.  

\begin{cor}\label{BBG}\label{BGOM}  If $G$ is a monoid in $\sO[\sT]$, Then $\bE G$, 
$\bE\OM BG$, and $\OM \bE BG$ are naturally equivalent spectra. Therefore the first delooping 
$\bE_1G$ and the classical classifying space $BG\htp \bE_0BG$ are equivalent as $\sO$-spaces.  
\end{cor}

\end{document}